\documentclass{article}%
\usepackage{amsmath}
\usepackage{amssymb}
\usepackage{amsmath}
\usepackage{amsfonts}
\usepackage{graphicx}
\usepackage{stmaryrd,hyperref,url}
\DeclareMathOperator*{\esssup}{ess\,sup}
\DeclareMathOperator*{\essinf}{ess\,inf}
\usepackage[english]{babel}
\usepackage{dsfont}
\usepackage{url}
\usepackage{authblk} 

\newtheorem{theorem}{Theorem}

\newtheorem{corollary}[theorem]{Corollary}

\newtheorem{definition}[theorem]{Definition}
\newtheorem{example}[theorem]{Example}

\newtheorem{lemma}[theorem]{Lemma}

\newtheorem{proposition}[theorem]{Proposition}
\newtheorem{remark}[theorem]{Remark}

\newenvironment{proof}[1][Proof]{\noindent\textbf{#1.} }{\ \rule{0.5em}{0.5em}}
\usepackage{vmargin}
\setmarginsrb{3.5cm}{4cm}{3.5cm}{4cm}{0cm}{0cm}{0cm}{0.5cm}
\begin{document}
	\title{Doubly reflected BSDEs with default time under stochastic Lipschitz coefficients: Filtration links and generalized Dynkin games}
	%
	
\author[1]{Badr ELMANSOURI\thanks{Corresponding author}} 
	\author[2]{Mohamed EL OTMANI} 
	\affil[1]{Cadi Ayyad University (UCA), National School of Applied Sciences of Marrakech (ENSA-M), BP 575, Avenue Abdelkrim Khattabi, 40000, Guéliz, Marrakech, Morocco}
	
	\affil[2]{Laboratory of Analysis, Mathematics, and Applications (LAMA), 
		Faculty of Sciences Agadir, 
		Ibnou Zohr University, Agadir, Morocco} 
	
	\affil[ ]{\texttt{\url{badr.elmansouri@edu.uiz.ac.ma}} \& \texttt{\url{b.elmansouri@uca.ac.ma}
	}} 
	\affil[ ]{\texttt{\url{m.elotmani@uiz.ac.ma}}} 
	\date{}
	\maketitle
	\begin{abstract}
	We study doubly reflected backward stochastic differential equations (DRBSDEs) on a random horizon $T \wedge \tau$ generated by a default time $\tau$ in a progressively enlarged filtration. We work within the change-of-measure framework previously developed for progressive enlargement, under which the reference Brownian motion stopped at the default time remains a Brownian motion in the enlarged filtration. The DRBSDEs are driven by stochastic Lipschitz generators and involve two RCLL completely separated barriers.
	
	We first establish existence and uniqueness of solutions in a weighted square-integrable framework, allowing for an additional martingale component orthogonal to the stopped Brownian martingale. We then investigate the relationship between DRBSDEs formulated in the enlarged filtration and related equations in the reference Brownian filtration. 
	
	Finally, we formulate a nonlinear Dynkin game on the random horizon $T\wedge\tau$ and characterize its value process in terms of the solution of the associated DRBSDE. Under suitable regularity assumptions, we also identify a saddle point through the first hitting times of the lower and upper barriers.
	\end{abstract}
	
	
	
	\section{Introduction}
	The concept of backward stochastic differential equations (BSDEs) in the context of Brownian motion was initially introduced by Bismut \cite{bismut1973conjugate} in the linear framework. It gained significant attention, especially after the groundbreaking work of Pardoux and Peng \cite{pardoux1990adapted}. Subsequent developments in this field have been motivated by diverse applications in finance \cite{el1997backward}, stochastic control and differential games \cite{ELKAROUI2003145, HamadeneJP}, as well as partial differential equations \cite{pardoux1999bsdes,pardoux2005backward}.
	
	Reflected BSDEs (RBSDEs) represent a subtype of BSDEs in which the solution is constrained to remain above a designated process known as an obstacle or barrier. Initially formulated in \cite{el1997reflected} within the context of a Brownian filtration and a continuous obstacle, subsequent extensions by various authors \cite{hamadene2002reflected,hamadene2016reflected,quenez2014reflected} have broadened its scope to include scenarios involving a not necessarily continuous obstacle and/or a filtration larger than the Brownian one. RBSDEs have proven effective and efficient for investigating issues in various mathematical domains  \cite{Elmansouri2026Default,hamadene2010switching,hamadene2000reflected}.
	
	Cvitanic and Karatzas in \cite{cvitanic1996backward} introduced doubly reflected BSDEs (DRBSDEs) within the context of continuous barriers with Brownian filtration. The solutions to such equations are constrained to remain within the region bounded by two adapted processes. DRBSDEs have been extensively investigated by many researchers in scenarios involving right-continuous barriers and/or filtrations associated with a Brownian motion and a random Poisson measure  \cite{Crepy, essaky2005backward, HassaniSaid, hamadene2009bsdes}.
	
	In the present paper, we work on a complete probability space
	$(\Omega,\mathcal{F},\mathbb{P})$ supporting a $d$-dimensional standard
	Brownian motion $B=(B_t)_{t\geq 0}$, and we denote by
	$\mathbb{F}=(\mathcal{F}_t)_{t\geq 0}$ its completed natural filtration.
	We consider an arbitrary positive random time $\tau$ which is not
	necessarily an $\mathbb{F}$-stopping time. In the context of credit-risk
	modeling, $\tau$ is interpreted as a default time.
	
	To incorporate the information generated by the occurrence of $\tau$,
	we consider the progressive enlargement
	$\mathbb{G}=(\mathcal{G}_t)_{t\geq 0}$ of $\mathbb{F}$ with $\tau$.
	The progressive enlargement of filtrations and the corresponding
	semimartingale decompositions have been extensively studied since the
	seminal works of Jeulin and Yor; see, for instance,
	\cite{book:63874,jeulin2006grossissement,azema1993theoreme,mansuy2006random}.
	In general, an $\mathbb{F}$-martingale need not remain a martingale after
	progressive enlargement, and therefore the immersion hypothesis
	$(\mathcal{H})$ is often imposed in BSDE models involving default times.
	
	In this work, we do not impose the immersion hypothesis. Instead, we
	adopt the change-of-measure framework developed in
	\cite{choulli2022explicit,alsheyab2023optimal,choulli2024optimal}.
	More precisely, under the positivity condition \textbf{[P]} on the
	Azéma supermartingale associated with $\tau$, we consider an equivalent
	probability measure $\mathbb{Q}$ under which the stopped Brownian motion
	\[
	B^\tau=(B_{t\wedge\tau})_{t\geq 0}
	\]
	is a $(\mathbb{G},\mathbb{Q})$-Brownian motion. Throughout the paper,
	we refer to $(\Omega,\mathcal{F},\mathbb{G},\mathbb{Q})$ as the
	defaultable setup. We emphasize that the construction of the measure
	$\mathbb{Q}$ and the underlying enlargement-of-filtration machinery are
	taken from the aforementioned literature and are used here as the
	probabilistic framework for the analysis of doubly reflected BSDEs.
	
	Our first objective is to study doubly reflected BSDEs on the random
	horizon
	\[
	\tau^T:=T\wedge\tau,
	\]
	where $T\in(0,+\infty)$ is a fixed deterministic horizon. The equation
	is associated with a terminal condition $\xi$, two RCLL completely
	separated barriers $L$ and $U$, and an $\mathbb{F}$-progressively
	measurable generator $f$ satisfying stochastic Lipschitz conditions
	with respect to the variables $(y,z)$. The solution contains, in
	addition to the Brownian stochastic integral, a martingale component
	orthogonal to $B^\tau$. We establish suitable a priori estimates and
	an existence and uniqueness result in weighted square-integrable spaces.
	For the existence part, we adapt the penalization and fixed-point
	arguments developed for doubly reflected BSDEs to the present
	default-time framework.
	
	The main purpose of the paper is then to investigate how the solution
	of a DRBSDE changes when passing between the reference filtration
	$\mathbb{F}$ and the enlarged filtration $\mathbb{G}$. This question is
	particularly relevant in the presence of a default time, since the two
	filtrations represent different flows of information.
	
	We establish two types of relationships.
	
	\begin{itemize}
		\item \textit{First link.}
		We establish an explicit connection between DRBSDE solutions in the
		$(\mathbb{G},\mathbb{Q})$ framework and related DRBSDEs in the
		$(\mathbb{F},\mathbb{P})$ framework. This part extends to the
		doubly reflected setting the one-barrier approach developed in
		\cite{alsheyab2023optimal}. In particular, we identify the
		transformed terminal condition, generator and barriers appearing
		in the reference filtration and derive explicit relationships
		between the corresponding solution processes.
		
		\item \textit{Second link.}
		We investigate the $\mathbb{F}$-optional projection of the
		$(\mathbb{G},\mathbb{Q})$-solution under the same probability
		measure $\mathbb{Q}$. Under an additional independence assumption
		between $\tau$ and $\mathcal{F}_T$ under $\mathbb{Q}$, we show that
		the projected state process satisfies a doubly reflected BSDE in
		the $(\mathbb{F},\mathbb{Q})$ framework, with an additional linear
		term arising from the change of probability measure. This second
		result is therefore presented as a particular correspondence under
		stronger assumptions and is not required for the general
		well-posedness theory developed in the paper.
	\end{itemize}
	
	Thus, the contribution of the present work does not consist in the
	construction of the enlarged-filtration probability measure itself.
	Rather, our contribution lies in the study of doubly reflected BSDEs
	within this default-time framework, in the extension of previously
	known one-barrier filtration relationships to the double-reflection
	case, and in the explicit identification of the corresponding
	$\mathbb{F}$- and $\mathbb{G}$-solutions.
	
	Finally, we apply the developed DRBSDE framework to a generalized
	Dynkin game with default on the random horizon $T\wedge\tau$. The
	players evaluate the payoff through the nonlinear $f$-expectation
	associated with the underlying BSDE. Under the assumptions imposed on
	the terminal payoff and on the two barriers, we characterize the value
	process of the game by the first component of the associated DRBSDE and
	derive a saddle point in terms of the first hitting times of the lower
	and upper barriers.
	To set the stage, let us first recall some results on the classical Dynkin game problem, which has been widely studied  \cite{alario2006dynkin,bismut1977probleme,ElmansouriElOtmani2026, kobylanski2014dynkin,lepeltier1984jeu}: For any pair $(\sigma_{1},\sigma_{2})$ of $\mathbb{G}$-stopping times valued in $[0, T]$, the terminal time of the game is given by $\sigma_{1} \wedge \sigma_{2}$ and the criterion of the game is defined as:
	\begin{equation*}
		\mathcal{J}(\sigma_{1},\sigma_{2}):=\mathbb{E}_{\mathbb{Q}}\left[ L_{\sigma_{1}} \mathds{1}_{\{\sigma_{1} \leq \sigma_{2}\}}+U_{\sigma_{2}} \mathds{1}_{\{\sigma_{2} < \sigma_{1}\}}\right],
		\label{SolPNL}
	\end{equation*}
	where $\mathbb{E}_{\mathbb{Q}}\left[\cdot\right]$ denotes the classical linear expectation under the probability measure $\mathbb{Q}$. This familiar Dynkin game has been found to have a common value without adopting Mokobodzki's condition \cite{lepeltier1984jeu}, meaning that
	$$
	\inf_{\sigma_{2}} \sup_{\sigma_{1}} \mathcal{J}(\sigma_{1},\sigma_{2})= \sup_{\sigma_{1}}\inf_{\sigma_{2}} \mathcal{J}(\sigma_{1},\sigma_{2}).
	$$ 
	The authors in \cite{lepeltier1984jeu} also show that the game has a saddle point when $L$ is left upper semi-continuous, and $U$ is left lower semi-continuous. Related to the theory of DRBSDEs, it is well-known that the value function for the Dynkin game problem can be characterized as the solution of the DRBSDE at time $t=0$ with a lower barrier $L$, an upper barrier $U$, and a zero driver (see, for instance, \cite{SMD7DS,HMDNAM} among others).
	
	This work deals with a type of generalized Dynkin game problems that support a default possibility for one of the stoppers, followed by a late payment with an agreed penalty imposed on the defaultable part. In this context, we refer to $\mathcal{E}^f_{t,\sigma}(\chi)$ as the $\mathcal{E}^f$-expectations at time $t$ of the $\mathcal{G}_{\sigma}$-measurable random variable $\chi$. To the best of our knowledge, Dumitrescu et al. \cite{Roxana} were the first to introduce this type of game. The term \textit{generalized} in the context of the Dynkin game refers to the existence of a \textit{non-linear $f$-expectation} rather than the classical linear expectation $\mathbb{E}_{\mathbb{Q}}\left[\cdot\right]$. In our context and in broad terms, we consider a double-stopping game with an expiry time given by the $\mathbb{G}$-stopping time $\tau^T:=T \wedge \tau$, such that for all pairs $(\sigma_{1},\sigma_{2})$ of $\mathbb{G}$-stopping times valued in $[0, \tau^T]$, i.e., $\mathbb{Q}\left(0 \leq \sigma_{1} \vee \sigma_{2} \leq \tau^T\right)=1$, the terminal time of the game is given by $\sigma_{1} \wedge \sigma_{2}$, and the payoff used to evaluate the outcome is defined by the following random variable:
	\begin{equation*}
		\mathcal{J}\left(\sigma_1,\sigma_2\right)=L_{\sigma_1} \mathds{1}_{\{\sigma_1 < \sigma_2\}}+U_{\sigma_2} \mathds{1}_{\{\sigma_2 <\sigma_1\}}+Q_{\sigma_{1}} \mathds{1}_{\{\sigma_1 = \sigma_2<\tau^T\}}+\xi \mathds{1}_{\{\sigma_1=\sigma_2=\tau^T\}},
	\end{equation*}
	where $(Q_t)_{t \leq T}$ is an $\mathbb{F}$-progressively measurable process such that $L_t \leq Q_t \leq U_t$ for all $t \in [0,\tau^T[$, $\mathbb{Q}$-a.s., and the relationship between the payoffs before and in the event of default is specified in the game rules. Then, using results from the first part of the present paper, we show that when $L$ and $U$ are completely separated, satisfying some integrability conditions that are stopped at $\tau^T$, and if the terminal condition satisfies $L_{\tau^T} \leq \xi \leq U_{\tau^T}$, $\mathbb{Q}$-a.s., there exists a value function for the $\mathcal{E}^f$-Dynkin game with default, and this value can be characterized as the solution of a DRBSDE
	$$
	\inf_{\sigma_{2}} \sup_{\sigma_{1}}\mathcal{E}^f_{0,\sigma_{1}\wedge\sigma_{2}}\left(  \mathcal{J}(\sigma_{1},\sigma_{2})\right) = \sup_{\sigma_{1}}\inf_{\sigma_{2}} \mathcal{E}^f_{0,\sigma_{1}\wedge\sigma_{2}}\left(  \mathcal{J}(\sigma_{1},\sigma_{2})\right)=Y_0,
	$$
	where $Y_0$ denotes the state process of the solution at time $t=0$ of the DRBSDE with stochastic Lipschitz driver $f$, barriers $(L,U)$ and terminal condition $\xi$. Moreover, as the barriers have the same default jumps as $Y$, we prove that there exists a saddle point for the $\mathcal{E}^f$-Dynkin game.
	
	The paper is organized as follows: In Section \ref{Notations}, we introduce the notation and the default time framework, including the mathematical structure of the random variable $\tau$, along with some preliminary results and assumptions used throughout the paper. Section \ref{DRBSDE Problem section} is dedicated to the full mathematical analysis of the DRBSDEs associated with a stochastic Lipschitz driver $f$ and completely separated barriers $(L_t,U_t)_{t \geq 0}$ without assuming the so-called Mokobodzki's assumption. Here, we consider the DRBSDE on the random horizon $\tau \wedge T$,
	where $T \in (0,+\infty)$ is a fixed deterministic horizon. In particular, we establish the existence and uniqueness of the solution to this equation. Additionally, we provide various integrability relationships for data within the filtrations $\mathbb{F}$ and $\mathbb{G}$, and explicit relationships between solutions across filtrations $\mathbb{F}$ and $\mathbb{G}$ depending on whether we consider the same probability $\mathbb{Q}$ or different equivalent probability measures $\mathbb{P}$ and $\mathbb{Q}$. In Section \ref{Dynkin games}, we present our generalized Dynkin game problem with default, formulating the problem both mathematically and abstractly. The key finding, characterizing the value of the game issue as the solution to our DRBSDE, is presented, along with the existence of a saddle point.
	\section{Notations, default time framework and preliminary results}
	\label{Notations}
	This section provides the general notation, assumptions, and preliminaries that will be used throughout the paper.
	
	Let $\mathbb{C}:=\left(\mathcal{C}_t\right)_{t \geq 0}$ be a generic filtration that satisfies the usual conditions of completeness and right continuity defined on a complete probability space $(\Omega,\mathcal{F},\mathsf{P})$. \\
	Recall that if $\mathcal{C}\left(\mathbb{C}\right)$ is a class of processes that are adapted to $\mathbb{C}$, then $\mathcal{C}_{loc}\left(\mathbb{C}\right)$ is the localized class of $\mathcal{C}\left(\mathbb{C}\right)$, i.e., the set of processes $(X_t)_{t \geq 0}$ for which there exists a sequence of $\mathbb{C}$-stopping times $\left\lbrace \sigma_n\right\rbrace_{n \in \mathbb{N}}$ such that $\sigma_n \nearrow +\infty$, $\mathsf{P}$-a.s., and the stopped process $X^{\sigma_n}:=\left( X_{t \wedge \sigma_n}\right)_{t \geq 0}$ belongs to $\mathcal{C}\left(\mathbb{C}\right)$ for each $n \in \mathbb{N}$. In this context, $\mathcal{M}_{loc} \left(\mathbb{C},\mathsf{P}\right)$ (respectively, $\mathcal{M} \left(\mathbb{C},\mathsf{P}\right)$) denotes the set of all $\mathbb{C}$-local martingales (respectively, martingales) under $\mathsf{P}$. Let $(X_t)_{t \geq 0} \in \mathcal{C}\left(\mathbb{C}\right)$, if $(X_t)_{t \geq 0}$ is an RCLL process, then we define $X_{s-}=\lim\limits_{u \nearrow s} X_u$ as the left limit of $X$ at $s \in (0,T]$, and $\Delta X_{s}=X_{s}-X_{s-}$ as the size of the jump at time $s$, with the convention $\Delta X_{0}=0$ (equivalently $X_{0-}=X_0$). For an $\mathcal{F} \otimes \mathcal{B}\left(\mathbb{R}_{+}\right)$-measurable process $V:=(V_t)_{t \geq 0}$ with finite variation, by $V^{o,\mathbb{C},\mathsf{P}}$ (resp. $V^{p,\mathbb{C},\mathsf{P}}$), we mean the $\mathbb{C}$-dual optional (resp. predictable) projection of $V$ when it exists under the probability measure $\mathsf{P}$. Next, for a given two $(0,+\infty)$-valued $\mathbb{C}$-stopping times $\sigma_1$ and $\sigma_2$,  $\mathfrak{T}_{\sigma_1,\sigma_2}(\mathbb{C})$ (respectively, $\mathfrak{T}(\mathbb{C})$) denotes the set of all $\mathbb{C}$-stopping times $\sigma$ such that $\mathsf{P}\left(\sigma_1 \leq \sigma \leq\sigma_2\right)=1$  (respectively, $\mathsf{P}\left(\sigma \in (0,+\infty) \right)=1$).\\
	For a given $\mathbb{C}$-adapted semimartingale $(X_t)_{t \geq 0}$, the set $\mathcal{I}(\mathbb{C}, \mathsf{P}, X)$ represents the collection of $\mathbb{C}$-predictable processes that are integrable with respect to the semimartingale $X$ under the probability measure $\mathsf{P}$. If we have two stopping times, $\sigma_{1}$ and $\sigma_{2}$, both belonging to $\mathfrak{T}(\mathbb{C})$ and satisfying $\sigma_{1} \leq \sigma_{2}$ almost surely under $\mathsf{P}$, we use the notation $\llbracket \sigma_1 , \sigma_{2}\rrbracket$ to represent the stochastic interval defined as $\llbracket \sigma_1 , \sigma_{2}\rrbracket=\{(\omega,t)\in \Omega \times [0,+\infty): \sigma_{1}(\omega)\leq t \leq \sigma_{2}(\omega)\}$. The same definition applies to other types of stochastic intervals \cite[Definition 3.14]{bookHe}. The equality $X=Y$ between any two processes $(X_t)_{t \geq 0}$ and $(Y_t)_{t \geq 0}$ must be understood in the indistinguishable sense, meaning that $\mathsf{P}\left(\omega:X_t(\omega)=Y_t(\omega), \forall t \geq 0\right)=1$. The same interpretation applies to $X \leq Y$ and $X < Y$. Finally, $\mathbb{E}_{\mathsf{P}}\left[\cdot\right]$ denotes the expectation operator under the probability measure $\mathsf{P}$. For a vector $x := (x_1, \ldots, x_d) \in \mathbb{R}^{d}$ ($d \geq 1$), the Euclidean norm is denoted by $\left\|x\right\|^2=\sum_{i=1}^d \left|x_i\right|^2$, and $\left\|x\right\|=\left|x\right|$ for $d=1$. Additionally, for $x \in \mathbb{R}$, we remember that $x^{+}=\max(x,0)$ and $x^{-}=\max(-x,0)=-\min(x,0)$.
	
	Consider a complete probability space $(\Omega, \mathcal{F}, \mathbb{P})$ on which we define a $d$-dimensional standard Brownian motion $(B_t)_{t \geq 0}$. Let $\mathbb{F} = (\mathcal{F}_t)_{t \geq 0}$ denote its completed natural filtration.
	Consider an arbitrary random time $\tau$, which need not be an $\mathbb{F}$-stopping time. This random time $\tau$ is characterized as an $\mathcal{F}$-measurable random variable with values in the interval $(0,+\infty)$ such that $\mathbb{P}\left(\{\tau=0\}\cup \{\tau=+\infty\}\right)=0$. This random variable $\tau$ models a \textit{default time}. We associate with $\tau$ the default indicator process $(D_t)_{t \geq 0}$ and introduce a new filtration $\mathbb{G} = (\mathcal{G}_t)_{t \geq 0}$ defined as follows
	\begin{equation}
		D_t = \mathds{1}_{\{\tau \leq t\}}, \quad \mathcal{G}^0_t
		:=
		\mathcal{F}_t
		\vee
		\sigma(\tau\wedge t)
		=
		\mathcal{F}_t
		\vee
		\sigma\big(\{\tau\leq s\}:0\leq s\leq t\big), \quad\mbox{and}\quad \mathcal{G}_t = \bigcap_{\epsilon > 0} \mathcal{G}^0_{t+\epsilon}.
		\label{Filtration}
	\end{equation}
	We adopt the convention $\mathcal{G}_{\infty} := \bigvee_{t \geq 0} \mathcal{G}_t=\mathcal{F}$. The filtration $\mathbb{G}$ is the smallest right-continuous complete filtration that contains $\mathbb{F}$ and makes $\tau$ a stopping time. This choice is motivated by the fact that the information related to the occurrence of the random time $\tau$ may not be entirely contained in $\mathbb{F}$ but rather in $\mathbb{G}$.\\
	In addition to $\mathbb{G}$ and $D$, the random time $\tau$ needs to be parameterized through $\mathbb{F}$ by the important pair $(G,\tilde{G})=(G_t,\tilde{G}_t)_{t \geq 0}$ of survival probabilities. These probabilities are defined by the following two $\mathbb{F}$-optional processes
	\begin{equation*}
		G_t:=\mathbb{P}\left(\tau>t \mid \mathcal{F}_t\right),\quad\mbox{and}\quad \tilde{G}_t:= \mathbb{P}\left(\tau \geq t \mid \mathcal{F}_t\right),\quad t \geq 0.
		\label{Azemas}
	\end{equation*}
	The process $G$, known as the \textit{Azéma supermartingale} of $\tau$ with respect to $\mathbb{F}$, exhibits right-continuous trajectories with left limits  (see Theorem 9 in \cite{bookProtter}, pp. 8). It serves as the $\mathbb{F}$-optional projection of $\mathds{1}_{\llbracket 0,\tau \llbracket}$. The process $\tilde{G}$, referred to as the \textit{Azéma optional supermartingale}, possesses finite right and left limits only and it is the $\mathbb{F}$-optional projection of $\mathds{1}_{\llbracket 0,\tau \rrbracket}$ (for more comprehensive details, please refer to \cite{azema1972quelques}).
	These processes $G$ and $\tilde{G}$ play a fundamental role in characterizing the behavior of the random time $\tau$ within the context of the filtration $\mathbb{F}$.
	
	For technical reasons, throughout the paper, we assume the following condition:
	\begin{center}
		\textbf{[P] } \text{ For }$\mathbb{P}$-almost all $\omega \in \Omega$, for all $t \geq 0$, 
		$G_t(\omega) > 0$.
	\end{center}
	Note that under the assumption \textbf{[P]}, we have (see Lemma 2.4 in \cite{choulli2022explicit}, pp. 234):
	\begin{itemize}
		\item $G_{-}$ and $\tilde{G}$ are positive processes.
		\item $G=G_0 \mathcal{E}\left(\Gamma^G \right)\mathcal{E}\left(\Gamma^{\tilde{G}} \right)=\mathcal{E}\left(\Gamma^G \right)\mathcal{E}\left(\Gamma^{\tilde{G}} \right)$, where $G_0=\mathbb{P}\left(\tau>0 \mid \mathcal{F}_0\right)=\mathbb{P}\left(\tau>0 \right)=1$. Here, the operator $\mathcal{E}\left(\mathcal{X}\right)$ represents the Doléans-Dade exponential, defined on the set of $\mathbb{G}$-semimartingales $(\mathcal{X}_t)_{t \geq 0}$ \cite[Theorem II.8.37]{bookProtter}. Additionally, 
		$\Gamma^G:=\int_{0}^{\cdot} G^{-1}_{s-} d m_s$, $\Gamma^{\tilde{G}}:=-\int_{0}^{\cdot} \tilde{G}^{-1}_{s} d D^{o,\mathbb{F},\mathbb{P}}_s$, 
		and $m:=D^{o,\mathbb{F},\mathbb{P}}+G$ is a BMO $\mathbb{F}$-martingale (see Proposition 1.46 in \cite{aksamit2017enlargement}, pp. 21), with $\Delta m=\tilde{G}-G_{-}$. For simplicity, we adopt the following notations throughout the paper: 
		\begin{equation}
			\mathcal{E}:=\mathcal{E}\left(\Gamma^G \right)~ \text{ and }~ \tilde{\mathcal{E}}:=\mathcal{E}\left(\Gamma^{\tilde{G}} \right).
			\label{Definition of E and E tilde}
		\end{equation}
	\end{itemize}
	
	On the one hand, it should be noted that the stability of the semimartingale property between $\mathbb{F}$ and $\mathbb{G}$, known as the $(\mathcal{H}^{\prime})$-Hypothesis, does not always hold. This legitimate question was first raised by P.A. Meyer and subsequently addressed by M. Yor in \cite[Theorem 1, pp. 62]{JeulinDecom}, where the author proved that the $(\mathcal{H}^{\prime})$-Hypothesis holds for all $\left(\mathbb{F},\mathbb{P}\right)$-semimartingales stopped at $\tau$. More precisely, we recall the following result:
	\begin{theorem}
		For any $\left(\mathbb{F},\mathbb{P}\right)$-semimartingale $\left(X_t\right)_{t \geq 0}$, the process $X^{\tau}:=\left(X_{t \wedge \tau}\right)_{t \geq 0}$ is a $\left(\mathbb{G},\mathbb{P}\right)$-semimartingale.
	\end{theorem}
	

	
	The change of probability measure used throughout the paper is taken from the progressive-enlargement framework developed in \cite{choulli2022explicit} and subsequently used in the context of reflected BSDEs in \cite{alsheyab2023optimal}. We recall below the result in the form needed for our analysis.
	
	More precisely, under condition \textbf{[P]}, the density process associated with the enlargement allows one, for every fixed finite horizon $T$, to define an equivalent probability measure on $(\Omega,\mathcal G_T)$. The following proposition is therefore recalled from \cite[Proposition 4.3]{choulli2022explicit} and \cite[Proposition 2.3]{alsheyab2023optimal}.
	
	\begin{proposition}\label{lemma.Q}
		The process $1/\mathcal{E}^{\tau}$ belongs to
		$\mathcal{M}(\mathbb{G},\mathbb{P})$. Moreover, for any fixed
		$T\in(0,+\infty)$, the probability measure $\mathbb{Q}$ defined on
		$(\Omega,\mathcal{G}_T)$ by
		\begin{equation}
			d\mathbb{Q}
			:=
			\Psi_T\,d\mathbb{P},
			\qquad
			\Psi:=\frac{1}{\mathcal{E}^{\tau}},
			\label{local continuity of Q}
		\end{equation}
		is well defined.
	\end{proposition}
	
	
	Additionally, under \textbf{[P]} and using Yor's formula \cite[Theorem II.38, pp. 86]{bookProtter}, the density process $\Psi$ may be expressed in the following form:  $\Psi^{}=\mathcal{E}\left(\Gamma^{\mathcal{T}(m)}\right)$,  where  $\Gamma^{\mathcal{T}(m)}:=-\int_{0}^{\cdot}G^{-1}_{s-} d\mathcal{T}(m)_s$, and $\mathcal{T}(\cdot)$ defined by (\ref{Theorem of the operator}). In particular, from \cite[Remark 2.5, pp. 235]{choulli2022explicit} and under \textbf{[P]}, we can see that $1+\Delta\Gamma^{\mathcal{T}(m)}>0$, which implies that the process $\left(\Psi_t\right)_{t\geq 0}$ is a strictly positive $\left(\mathbb{G},\mathbb{P}\right)$-martingale.
	
	%
	\begin{remark}
		In the infinite-horizon case ($T=+\infty$), defining a probability
		measure via
		\[
		d\mathbb{Q}:=\Psi_{\infty}\,d\mathbb{P},
		\qquad
		\Psi_{\infty}:=\lim_{t\to\infty}\Psi_t,
		\]
		requires the $(\mathbb{G},\mathbb{P})$-martingale
		$(\Psi_t)_{t\geq0}$ to be uniformly integrable; see Remark~1 in
		\cite{choulli2024optimal} and Proposition~4.3 in
		\cite{choulli2022explicit}. Without this assumption, $\mathbb{Q}$
		may not be well defined on $(\Omega,\mathcal{F})$.
		
		In the present finite-horizon framework, however, such uniform
		integrability is not required. For each fixed
		$T\in(0,+\infty)$, the probability measure $\mathbb{Q}$ defined on
		$(\Omega,\mathcal{G}_T)$ by \eqref{local continuity of Q} is well
		defined.
	\end{remark}
	
	The preceding change of probability measure yields the fundamental property of the stopped Brownian motion that will be used throughout the paper. This property is a consequence of the enlargement-of-filtration results recalled above and has already been established in the related literature; see, in particular, \cite{choulli2022explicit,alsheyab2023optimal}. We state it here for convenience and to fix the probabilistic framework of the subsequent DRBSDEs.
	
	\begin{lemma}
		Under assumption \textbf{[P]}, for any fixed horizon time
		$T\in(0,+\infty)$, the stopped process
		$B^{\tau\wedge T}$ is a
		$(\mathbb{G},\mathbb{Q})$-Brownian motion.
		\label{Lemma of the Brownian motion}
	\end{lemma}
	\begin{remark}
		\begin{itemize}
			\item Lemma \ref{Lemma of the Brownian motion} plays an important role in our work. 
			In the literature on progressive enlargement of Brownian filtrations in the field of  BSDEs, the classical assumption is to suppose that $B \in \mathcal{M}(\mathbb{F},\mathbb{P})$ remains a $\mathbb{G}$-Brownian motion under the original probability measure $\mathbb{P}$ (i.e. $B \in \mathcal{M}(\mathbb{G},\mathbb{P})$), known as the $(\mathcal{H})$-Hypothesis or {\it Immersion property} in credit risk modeling. However, since this is not the case in our setting, and given that the $(\mathcal{H})$-Hypothesis is considered a very strong and restrictive condition, Lemma \ref{Lemma of the Brownian motion} provides an alternative result under weaker assumptions, allowing us to retain the martingale characterization of $B^{T \wedge \tau}$ in $\mathbb{G}$ under $\mathbb{Q}$ for each $T \in (0,+\infty)$.
			
			\item General enlargement formulas in the progressive enlargement of a Brownian filtration framework are available in Mansuy and Yor \cite{mansuy2006random}. See also Jeanblanc at al.  \cite{jeanblanc2009progressive} for the particular case where the random time $\tau$ is an {\it initial time} and Barlow  \cite{barlow1978study} for the case where $\tau$ is an {\it honest time}.
		\end{itemize}
	\end{remark}

	Throughout the paper, we maintain the assumption \textbf{[P]}.

	\section{Doubly reflected BSDEs in the defaultable setup $(\Omega,\mathcal{F},\mathbb{G},\mathbb{Q})$}
	\label{DRBSDE Problem section}
	\subsection{Problem formulation}
	Let $\mathbb{C}:=(\mathcal{C}_t)_{t\geq0}$ be an arbitrary filtration
	on $(\Omega,\mathcal{F})$ satisfying the usual conditions, and let
	$\mathsf{Q}$ be a probability measure on $(\Omega,\mathcal{F})$
	equivalent to $\mathbb{P}$. \\
	Consider a non-negative $\mathbb{C}$-adapted process $\left(\alpha_t\right)_{t \geq 0}$, a parameter $\beta > 0$, and a giving random time $\theta \in \mathfrak{T}\left(\mathbb{C}\right)$. We define the increasing $\mathbb{C}$-adapted continuous process $\mathcal{A}_t := \int_{0}^{t} \alpha_s^2 ds$ for all $t \geq 0$. \\
	We now introduce the following spaces:
	\begin{itemize}
		\item $ \mathcal{S}^2_{ \theta}\left( \beta,\mathbb{C},\mathsf{Q}\right)$: The space of $\mathbb{R}$-valued and $\mathbb{C}$-adapted RCLL  processes $\left(Y_t\right)_{t \geq 0}$ such that $Y^{\theta}=Y$ and
		$$
		\left\|Y\right\|^2_{\mathcal{S}^2_{\theta}(\beta,\mathbb{C},\mathsf{Q})}=\mathbb{E}_{\mathsf{Q}}\left[\sup_{t \geq 0} e^{\beta \mathcal{A}_{t \wedge \theta}} \left|Y_{t \wedge \theta}\right|^2 \right]<+\infty,
		$$
		with the convention $ \mathcal{D}^2_{\theta}\left( \mathbb{C},\mathsf{Q}\right):=\mathcal{S}^2_{\theta}(0,\mathbb{C},\mathsf{Q})$
		\item $ \mathcal{S}^{2,\alpha}_{ \theta}\left( \beta,\mathbb{C},\mathsf{Q}\right)$: The space of $\mathbb{R}$-valued and $\mathbb{C}$-adapted RCLL  processes $\left(Y_t\right)_{t \geq 0}$ such that $Y^{\theta}=Y$ and
		$$
		\left\|Y\right\|^2_{\mathcal{S}^{2,\alpha}_{\theta}(\beta,\mathbb{C},\mathsf{Q})}=\mathbb{E}_{\mathsf{Q}}\left[\int_{0}^{\theta}e^{\beta \mathcal{A}_{s}} \left|Y_{s} \alpha_s\right|^2  ds \right]<+\infty.
		$$
		\item $ \mathcal{H}^2_{ \theta}\left( \beta,\mathbb{C},\mathsf{Q}\right)$: The space of $\mathbb{R}^d$-valued and $\mathbb{C}$-predictable processes $\left(Z_t\right)_{t \geq 0}$ such that $Z=Z \mathds{1}_{\llbracket 0,\theta \rrbracket}$ and
		$$
		\left\|Z\right\|^2_{\mathcal{H}^2_{\theta}(\beta,\mathbb{C},\mathsf{Q})}=\mathbb{E}_{\mathsf{Q}}\left[\int_{0}^{\theta} e^{\beta \mathcal{A}_s} \left\|Z_s\right\|^2 ds\right]<+\infty,
		$$
		with the convention $\mathcal{H}^2_{\theta}(\mathbb{C},\mathsf{Q}):=\mathcal{H}^2_{\theta}(0,\mathbb{C},\mathsf{Q})$.
		\item $ \mathcal{S}^2_{\theta}\left( \mathbb{C},\mathsf{Q}\right)$: The space of  $\mathbb{C}$-adapted continuous increasing processes $\left(K_t\right)_{t \geq 0}$ such that $K=K^{\theta}$ and
		$$
		\left\|K\right\|^2_{\mathcal{S}^2_{\theta}(\mathbb{C},\mathsf{Q})}=\mathbb{E}_{\mathsf{Q}}\left[  \left|K_{ \theta}\right|^2 \right] <+\infty.
		$$
		\item $ \mathcal{M}^2_{\theta}\left( \beta,\mathbb{C},\mathsf{Q}\right)$: The space of $\mathbb{R}$-valued martingales $M \in \mathcal{M}(\mathbb{C},\mathsf{Q})$ orthogonal to $B^{\theta}$ such that $M=M^{\theta}$ and
		$$
		\left\|M\right\|^2_{\mathcal{M}^2_{\theta}(\beta,\mathbb{C},\mathsf{Q})}=\mathbb{E}_{\mathsf{Q}}\left[\int_{0}^{\theta} e^{\beta \mathcal{A}_s} d\left[M,M\right]_s\right]<+\infty.
		$$
		\item $\mathcal{L}^p_{\theta}(\beta,\mathbb{C},\mathsf{Q})$: The set of $\mathcal{C}_{\theta}$-measurable random variables $\chi$ such that 
		$$
		\left\|\chi\right\|^p_{\mathcal{L}^p_{\theta}(\beta,\mathbb{C},\mathsf{Q})}=\mathbb{E}_{\mathsf{Q}}\left[ e^{\beta \mathcal{A}_{\theta}} \left|\chi\right|^p \right]<+\infty,\quad \text{with }~p \in \{1,2\},
		$$
		with the convention $\mathcal{L}^p_{\theta}(\mathbb{C},\mathsf{Q})=\mathcal{L}^p_{\theta}(0,\mathbb{C},\mathsf{Q})$.
		\item $\mathfrak{C}^2_{\theta}\left( \beta,\mathbb{C},\mathsf{Q}\right):=\mathcal{S}^2_{ \theta}\left( \beta,\mathbb{C},\mathsf{Q}\right) \cap  \mathcal{S}^{2,\alpha}_{ \theta}\left( \beta,\mathbb{C},\mathsf{Q}\right)$.
		\item $\mathfrak{U}^2_{\theta}\left( \beta,\mathbb{C},\mathsf{Q}\right):=\mathfrak{C}^2_{\theta}\left( \beta,\mathbb{C},\mathsf{Q}\right) \times \mathcal{H}^2_{ \theta}\left( \beta,\mathbb{C},\mathsf{Q}\right)  \times \mathcal{M}^2_{ \theta}\left(\beta,\mathbb{C},\mathsf{Q}\right)$.
		\item $\mathfrak{B}^2_{\theta}\left( \beta,\mathbb{C},\mathsf{Q}\right):=\mathfrak{C}^2_{\theta}\left( \beta,\mathbb{C},\mathsf{Q}\right) \times \mathcal{H}^2_{ \theta}\left( \beta,\mathbb{C},\mathsf{Q}\right) \times \mathcal{S}^2_{\theta}\left( \mathbb{C},\mathsf{Q}\right) \times \mathcal{S}^2_{\theta}\left( \mathbb{C},\mathsf{Q}\right) \times \mathcal{M}^2_{\theta}\left( \beta,\mathbb{C},\mathsf{Q}\right)$.
	\end{itemize}
	Now, we present the setup carrying our DRBSDE. We operate within the framework of the complete filtered space $(\Omega, \mathcal{F}, \mathbb{G})$, under the probability measure $\mathbb{Q}$. The filtration $\mathbb{G}$ is defined as per (\ref{Filtration}), and $\mathbb{Q}$ is established in accordance with (\ref{local continuity of Q}). \\
	In the first part of this paper, we are looking for a quintuplet of $\mathbb{G}$-adapted processes $(Y_t,Z_t,K^{+}_t,K^{-}_t,M_t)_{t \geq 0}$ such that for any deterministic horizon $T \in (0,+\infty)$, we have:
	\begin{equation}
		\left\{
		\begin{split}
			\text{(i)} &~ \mathbb{Q}\text{-a.s. for all } t \in [0,T]\\
			& Y_{t \wedge \tau}=\xi+\int_{t \wedge \tau }^{T \wedge \tau} f(s,Y_s,Z_s)ds+\left( K^{+}_{ T \wedge \tau}-K^{+}_{t \wedge \tau}\right)-\left( K^{-}_{ T \wedge \tau}-K^{-}_{t \wedge \tau}\right)\\
			&\qquad\qquad -\int_{t \wedge \tau }^{T \wedge \tau} Z_s d B_s-\int_{t \wedge \tau }^{T \wedge \tau}dM_s;\\
			\text{(ii)} &~  L_t \leq Y_t \leq U_t,~  \forall t \in [0,T \wedge \tau[,~\mathbb{Q}\text{-a.s.};\\
			\text{(iii)} &~ L_{T \wedge \tau} \leq \xi \leq U_{T \wedge \tau},~\mathbb{Q}\text{-a.s.};\\
			\text{(iv)} &~   Y_t=\xi,~Z_t=dK^{+}_t=dK^{-}_t=dM_t=0,\text{ on the set }\{t \geq T \wedge\tau\},~\mathbb{Q}\text{-a.s.};\\
			\text{(v)} &~ \text{Skorokhod condition:}  \int_0^{T \wedge \tau}(Y_{t}-L_{t})dK^{+}_t=\int_0^{T \wedge \tau}(U_{t}-Y_{t})dK^{-}_t=0,~\mathbb{Q}\text{-a.s.}\\
		\end{split}
		\right.
		\label{basic equation}
	\end{equation}

	\textbf{Note:} For simplicity, we fix a deterministic horizon time $T \in (0, +\infty)$ throughout the rest of the paper.\\

	We now present the definition of a solution for the DRBSDE (\ref{basic equation}) in our setting. 
	\begin{definition}
		Let $\beta>0$ and $(\alpha_t)_{t \geq 0}$ a non negative $\mathbb{F}$-adapted process. A solution to DRBSDE (\ref{basic equation}) with jumps associated with parameters $(\xi,f,L,U)$ is a quintuple of $\mathbb{G}$-adapted processes $(Y_t,Z_t,K^{+}_t,K^{-}_t,M_t)_{t \geq 0}$ which satisfy (\ref{basic equation}) and belongs to $\mathfrak{B}^2_{T \wedge \tau}\left( \beta,\mathbb{G},\mathbb{Q}\right)$.
		\label{definition}
	\end{definition}
	\begin{remark}
		\begin{itemize}
			\item[(a)] Let $(Y_t,Z_t,K^{+}_t,K^{-}_t,M_t)_{t \geq 0}$ be a solution of the DRBSDE (\ref{basic equation}) in the sense of Definition \ref{definition}. Then, 
			$$
			Y_t=\xi,\quad Z_t=dK^{+}_t=dK^{-}_t=dM_t=0 ~ \text{ on the set } \{t \geq T \wedge  \tau\}~~\mathbb{Q}\text{-a.s.,}
			$$
			and we can express (\ref{basic equation})-(i) as follows (see also \cite[Lemma 3.3]{topolewski2018reflected}): $\mathbb{Q}$-a.s. $\forall t \in [0,T]$,
			\begin{equation}
				Y_{t}=\xi+\int_{t }^{T} \mathds{1}_{\{s \leq \tau\}} f(s,Y_s,Z_s)ds+\left( K^{+}_{T}-K^{+}_{t}\right)-\left( K^{-}_{ T }-K^{-}_{t}\right)-\int_{t }^{T } Z_s d B_s-\int_{t}^{T}dM_s.
				\label{Writing BSDE forwadrdly}
			\end{equation}
			\item[(b)] 
			We point out that, according to a classical result in the general theory of stochastic processes  \cite[Theorem IV.84]{dellacherie1975probabilites}, a quintuplet $(Y,Z,K^{+},K^{-},M)$ satisfies (\ref{basic equation})-(i) or (\ref{Writing BSDE forwadrdly}) equivalently if and only  
			$$Y_{\sigma_1}=Y_{\sigma_2}+\int_{\sigma_1}^{\sigma_2}  f(s,Y_s,Z_s) ds +\left(K^{+}_{\sigma_2}-K^{+}_{\sigma_1}\right)-\left(K^{-}_{\sigma_2}-K^{-}_{\sigma_1}\right)-\int_{\sigma_1}^{\sigma_2} Z_s dB_s-\int_{\sigma_1}^{\sigma_2} dM_s$$  $\mathbb{Q}$-a.s. for all $\sigma_1,\sigma_2 \in \mathfrak{T}_{0,T \wedge \tau}\left(\mathbb{G}\right)$ such that $\sigma_1 \leq \sigma_2$, $\mathbb{Q}$-a.s.
		\end{itemize}
		\label{Remark on the construction of the BSDE}
	\end{remark}
	Now, we provide a brief literature review to highlight the key differences between our framework and those explored in previous studies.\\
	Due to its close connection to semi-linear elliptic partial differential equations (PDEs) and its role in providing a probabilistic formula for the viscosity solution of such PDEs, a significant part of the literature on BSDEs without reflection (i.e., when $L\equiv-\infty$ and $U\equiv\infty$) has focused on the case where the terminal  time $\tau$ is a $\mathbb{F}$ stopping time. For a comprehensive discussion of this topic, see \cite{darling1997backwards,pardoux1998backward,StoppingPeng,royer2004bsdes} among others.\\
	In the same context, the study of reflected BSDEs (i.e. when $L\equiv-\infty$ or $U=\infty$) and doubly reflected BSDEs in the case where $\tau$ is a not necessarily finite $\mathbb{F}$-stopping time has been extensively explored in the framework of Brownian motion, as presented in our primary setting $(\Omega,\mathcal{F}, \mathbb{F},\mathbb{P})$ (see, e.g., \cite{Akdim}). More precisely, consider the case when $\tau \in \mathfrak{T}(\mathbb{F})$. Then, in this scenario, the data triplet $(f, L, U)$ is $\mathbb{F}$-progressively measurable, the terminal value $\xi$ belongs to $\mathcal{F}_{\tau}$, and the DRBSDE (\ref{basic equation}) simplifies to:
	\begin{equation}
		\left\{
		\begin{split}
			\text{(i)} &~ \mathbb{P}\text{-a.s. for all } t \geq 0\\
			&~~ Y_{t}=\xi+\int_{t \wedge \tau }^{\tau} f(s,Y_s,Z_s)ds+\left( K^{+}_{\tau}-K^{+}_{t \wedge \tau}\right)-\left( K^{-}_{\tau}-K^{-}_{t \wedge \tau}\right) -\int_{t \wedge \tau }^{\tau} Z_s d B_s;\\
			\text{(ii)} &~  L_t \leq Y_t \leq U_t,~  \forall t \geq 0,~\mathbb{P}\text{-a.s.};\\
			\text{(iii)} &~  Y_t=\xi,~Z_t=dK^{+}_t=dK^{-}_t=0,\text{ on the set }\{t \geq \tau\},~\mathbb{P}\text{-a.s.};\\
			\text{(iv)} &~ \text{ Skorokhod condition:}  \int_0^{ \tau}(Y_{t}-L_{t})dK^{+}_t=\int_0^{ \tau}(U_{t}-Y_{t})dK^{-}_t=0,~\mathbb{P}\text{-a.s.}
		\end{split}
		\right.
		\label{basic equation in F}
	\end{equation}
	In this scenario, the data $(f, L, U)$ are $\mathbb{F}$-progressively measurable and $\xi$ belongs to $\mathcal{F}_{\tau}$. Compared to previous works, it is clear from the DRBSDEs  (\ref{basic equation}) and (\ref{basic equation in F}) that:
	\begin{enumerate}
		\item The solution's definition does not include an orthogonal martingale term, indicating that $M \equiv 0$. This is because $(B_t)_{t \geq 0}$ is an $(\mathbb{F},\mathbb{P})$-Brownian motion, and the data $(\xi, f, L, U)$ are defined in the filtration $\mathbb{F}$. Consequently, on one hand, $B$ possesses the martingale representation property with respect to the historical probability measure $\mathbb{P}$, and, on the other hand, this fact enables us to construct the solution using the stochastic integral solely with respect to $B$.
		\item There is no need to change the probability measure from $\mathbb{P}$ to the new probability $\mathbb{Q}$ for DRBSDE \eqref{basic equation in F}. This is because the problem is well-defined within the $(\mathbb{F},\mathbb{P})$ setting. However, in our setup, as $\tau$ is not an $\mathbb{F}$-stopping time, and $\mathbb{F}$ only represents a partial flow of information generated by $(B_t)_{t \geq 0}$ under $\mathbb{P}$, we emphasize the importance of constructing a new framework, $\left(\mathbb{G},\mathbb{Q}\right)$, from the initial setup $\left(\mathbb{F},\mathbb{P}\right)$ where $\mathbb{G}$ represents the global flow of information. This is done to ensure the following:
		\begin{enumerate}
			\item The default terminal time $\tau$ becomes a stopping time through the progressive enlargement of the filtration, enabling the applications of classical results from stochastic analysis.
			\item The stopped process $(B^{\tau}_t)_{t \geq 0}$ retains its status as a Brownian motion in $\mathbb{G}$ under the probability measure $\mathbb{Q}$, rather than being a $\mathbb{G}$-semi-martingale under $\mathbb{P}$.
			
			\item Local martingales in the $(\mathbb{G},\mathbb{Q})$-framework admit,
			on the random horizon $\llbracket 0,T\wedge\tau\rrbracket$, a
			decomposition into a stochastic integral with respect to
			$B^{T\wedge\tau}$ and an additional martingale orthogonal to
			$B^{T\wedge\tau}$. This feature allows us to address various classical BSDEs and reflected BSDEs with default times, all originating from the broader framework $\mathbb{F}$.
		\end{enumerate}
	\end{enumerate}
	We note that the enlarged filtration $\mathbb{G}$ defined by (\ref{Filtration}) can be regarded as a general filtration for the $(\mathbb{F},\mathbb{P})$-Brownian motion $(B_t)_{t \geq 0}$. Most studies in the literature focus on various types of BSDEs starting from an initial setup $(\mathbb{F},\mathbb{P})$ that \textit{supports} the process $(B_t)_{t \geq 0}$, or even a setup $(\mathbb{F},\mathbb{P})$ that \textit{supports} a general RCLL martingale \cite[among others]{elmansourielotmani24,CRBN,BDEM,karoui1997general}. In such scenarios, the problem is well-posed within the $(\mathbb{F},\mathbb{P})$ framework due to the predictable representation properties with an additional orthogonal martingale term. These properties are discussed in Section III.4 of \cite{bookJacod} (Lemma 4.24, pp. 185) and also in \cite{Tianyang} (Remark 2.1, pp. 323).\\
	However, problem (\ref{basic equation}) lacks well-posedness in both $\mathbb{F}$ generated by $(B_t)_{t \geq 0}$ and in $\mathbb{G}$ under $\mathbb{P}$. This is because $\tau$ may not be a stopping time, and the Immersion property is not assumed to hold. The DRBSDE considered below is formulated in the enlarged filtration $\mathbb G$ under the equivalent probability measure $\mathbb Q$ recalled in Section~\ref{Notations}. This choice allows us to work with the $(\mathbb G,\mathbb Q)$-Brownian motion $B^\tau$ on the random horizon $T\wedge\tau$, without imposing the immersion hypothesis under the original probability measure $\mathbb P$. In addition, the enlarged filtration naturally gives rise to an orthogonal martingale component, which will be incorporated into the solution of the BSDE.
	
	%
	\subsection{Existence and uniqueness results of (\ref{basic equation})}
	We provide a comprehensive representation of the data $(\xi, f, L, U)$ involved in our DRBSDE (\ref{basic equation}), along with their measurability and integrability properties. These integrability results are consequences of the
	change-of-measure identities established in
	\cite{alsheyab2023optimal}. We state them separately for later use in the
	DRBSDE estimates.
	
	\textbf{Assumption} ($\mathbf{H_1}$)\textbf{.} We assume that:
	\begin{itemize}
		\item $\xi$ is an $\mathcal{F}_{T \wedge \tau}$-measurable random variable, where the definition of  $\mathcal{F}_{T \wedge \tau}$ is borrowed from \cite[Ch 0. page  6]{mansuy2006random} (see also \cite{Chung}):
		$$
		\mathcal{F}_{T \wedge \tau}=\sigma\left\{\zeta_{T \wedge \tau}, \left(\zeta_t\right)_{t \geq 0} \text{ any } \mathbb{F}\text{-optional process}\right\}.
		$$
		Therefore, we may choose an $\mathbb{F}$-adapted RCLL process $\left(\zeta_t\right)_{t \geq 0}$ such that 
		\begin{equation*}
			\xi=\zeta_{T \wedge \tau},\quad \mathbb{P}\text{-a.s.}
			\label{terminal variables}
		\end{equation*}
		\item $\xi$ belongs to $\mathcal{L}^2_{T \wedge \tau}\left(\beta,\mathbb{G},\mathbb{Q}^{}\right)$.
	\end{itemize}
	As the terminal variable $\xi$ is represented using the $\mathbb{F}$-optional processes $(\zeta_t)_{t \geq 0}$, we can establish a characterization of $\xi$'s integrability within $(\mathcal{G}_{T\wedge \tau},\mathbb{Q}^{})$ in terms of $\zeta$ within $(\mathcal{F}_{T},\mathbb{P})$. This characterization is provided in the following lemma:
	\begin{lemma}
		Let $V^{\mathbb{F}}:=1-\tilde{\mathcal{E}}$ and $V^{\ast \mathbb{F}}:=\int_{0}^{\cdot} \left| \zeta_s \right|^2 d V^{\mathbb{F}}_s$. Then, $\xi \in \mathcal{L}^2_{T \wedge \tau}(\beta,\mathbb{G},\mathbb{Q}^{})$ whenever $\left( V^{\ast \mathbb{F}}_T,\zeta_T \right) \in \mathcal{L}^1_T(\beta,\mathbb{F},\mathbb{P}) \times \mathcal{L}^2_T(\beta,\mathbb{F},\mathbb{P})$.
		\label{Lemma Merties}
	\end{lemma}
	\begin{proof}
		First, let's define the non-negative $\mathbb{F}$-adapted RCLL process $\zeta^{\ast}_t := e^{\beta \mathcal{A}_t} \left| \zeta_t \right|^2$. Therefore, using the fact that $\mathbb{P}\left(\tau=0\right)=0$, we have 
		\begin{equation}
			\begin{split}
				\mathbb{E}_{\mathbb{Q}}\left[\zeta^{\ast}_{T \wedge \tau}\right]=\mathbb{E}\left[\zeta^{\ast}_{T \wedge \tau}\Psi_{T \wedge \tau}\right]&=\mathbb{E}^{}\left[\zeta^{\ast}_{T }\Psi_T \mathds{1}_{\{T <\tau\}}+\zeta^{\ast}_{\tau }\Psi_{\tau}\mathds{1}_{\{0<\tau \leq T\}}+\zeta^{\ast}_{0}\Psi_{0}\mathds{1}_{\{\tau=0 \}} \right]\\
				&=\mathbb{E}^{}\left[\zeta^{\ast}_{T }\Psi_T \mathds{1}_{\{T <\tau\}}+\zeta^{\ast}_{\tau }\Psi_{\tau}\mathds{1}_{\{0<\tau \leq T\}} \right].
				\label{im lost}
			\end{split}
		\end{equation}
		Next, using the definition of the density process $\Psi$ as given by (\ref{local continuity of Q}), we derive that $\Psi_T \mathds{1}_{\{T<\tau\}}=\mathcal{E}^{-1}_T \mathds{1}_{\{T<\tau\}}$ and $\Psi_{\tau}=\mathcal{E}^{-1}_{\tau}$. Inserting these expressions into (\ref{im lost}), we obtain
		\begin{equation}
			\begin{split}
				\mathbb{E}_{\mathbb{Q}}\left[\zeta^{\ast}_{T \wedge \tau}\right]&=\mathbb{E}^{}\left[\zeta^{\ast}_{T }\mathcal{E}^{-1}_T \mathds{1}_{\{T <\tau\}}+\zeta^{\ast}_{\tau } \mathcal{E}^{-1}_{\tau} \mathds{1}_{\{0<\tau \leq T\}} \right]
				=\mathbb{E}^{}\left[\zeta^{\ast}_{T }\mathcal{E}^{-1}_T \mathds{1}_{\{T <\tau\}}+\int_{0}^{T}\zeta^{\ast}_{s }\mathcal{E}^{-1}_s dD_s \right]
				\label{AAA}
			\end{split}
		\end{equation}
		Next, since $\mathcal{E}$ is an $\mathbb{F}$-optional process, by passing to the $\mathbb{F}$-optional projection for $\mathds{1}_{\{T <\tau\}}$ and the $\mathbb{F}$-dual optional projection of the defaultable process $D$ in (\ref{AAA}), we get
		\begin{equation}
			\begin{split}
				\mathbb{E}_{\mathbb{Q}}\left[\zeta^{\ast}_{T \wedge \tau}\right]&=\mathbb{E}^{}\left[\zeta^{\ast}_{T }\mathcal{E}^{-1}_T G_T+\int_{0}^{T}\zeta^{\ast}_{s }\mathcal{E}^{-1}_s dD^{o,\mathbb{F},\mathbb{P}}_s \right]\\
				&=\mathbb{E}^{}\left[\zeta^{\ast}_{T }\mathcal{E}^{-1}_T G_T+\int_{0}^{T}e^{\beta \mathcal{A}_s}\left|\zeta_s \right|^2 \mathcal{E}^{-1}_s dD^{o,\mathbb{F},\mathbb{P}}_s \right]
				\label{im lost hey}
			\end{split}
		\end{equation}
		Moreover, from the equalities $\tilde{\mathcal{E}}_{s-}=\frac{\tilde{G}_s}{G_s} \tilde{\mathcal{E}}_s$ and $G_s=\mathcal{E}_s\tilde{\mathcal{E}}_s$, along with the definition $dV^{\mathbb{F}}_s=\tilde{G}^{-1}_s\tilde{\mathcal{E}}_{s-}dD^{o,\mathbb{F},\mathbb{P}}_s$, we deduce that $dV^{\mathbb{F}}_s=\mathcal{E}^{-1}_s dD^{o,\mathbb{F},\mathbb{P}}_s$ and $ \mathcal{E}^{-1}_T G_T=\tilde{\mathcal{E}}_T$.\\ 
		Since $V^{\ast \mathbb{F}}_t$ is an $\mathbb{F}$-optional non-negative increasing process, employing an integration by part formula, we can write $\mathbb{P}\text{-a.s.,}$
		\begin{equation}
			\begin{split}
				\int_{0}^{T}e^{\beta \mathcal{A}_s}\left|\zeta_s \right|^2 \mathcal{E}^{-1}_s dD^{o,\mathbb{F},\mathbb{P}}_s=\int_{0}^{T} e^{\beta \mathcal{A}_s} \left| \zeta_s \right|^2 d V^{\mathbb{F}}_s
				&=e^{\beta \mathcal{A}_T} V^{\ast \mathbb{F}}_T-\beta \int_{0}^{T} e^{\beta \mathcal{A}_s} \alpha_s^2 V^{\ast \mathbb{F}}_s ds\\
				&\leq e^{\beta \mathcal{A}_T} V^{\ast \mathbb{F}}_T,
				\label{Classico asser}
			\end{split}
		\end{equation}
		where the first equality arises from the definition of the process $\zeta^{\ast}$, while the last inequality is derived from the positivity of the process $V^{\ast \mathbb{F}}$.\\
		On the other hand, we are aware that the process $\tilde{\mathcal{E}}$ satisfies the following forward stochastic differential equation: $\tilde{\mathcal{E}}_t=1-\int_{0}^{t}\tilde{\mathcal{E}}_{s-}\tilde{G}^{-1}_s dD^{o,\mathbb{F},\mathbb{P}}_s$. Then, since $1+\Delta \Gamma^{\tilde{G}}=G_s \tilde{G}^{-1}_s>0$ under \textbf{[P]}, we deduce that $\tilde{\mathcal{E}}>0$, and thus $\tilde{\mathcal{E}} \leq 1$. By plugging this result and (\ref{Classico asser}) into (\ref{im lost hey}), we derive
		\begin{equation*}
			\mathbb{E}_{\mathbb{Q}}\left[e^{\beta \mathcal{A}_{T \wedge \tau}} \left|\xi \right|^2\right]=\mathbb{E}_{\mathbb{Q}}\left[\zeta^{\ast}_{T \wedge \tau}\right]\leq \mathbb{E}\left[e^{\beta \mathcal{A}_T} \left( V^{\ast \mathbb{F}}_T + \left| \zeta_T \right|^2  \right)  \right].
		\end{equation*}
		This completes the proof.
	\end{proof}
	\textbf{Assumption} ($\mathbf{H_2}$)\textbf{.} We assume that:
	\begin{itemize}
		\item The mapping $f:\Omega \times [0,T] \times \mathbb{R} \times \mathbb{R}^d \rightarrow \mathbb{R}$ is such that
		\begin{itemize}
			\item[(i)] For all $(y,z)$, the stochastic process $f(\cdot,y,z)$ is $\mathbb{F}$-progressively measurable.
			\item[(ii)] There exist two non-negative $\mathbb{F}$-adapted processes $(\kappa_t)_{t \leq T}$ and $(\gamma_t)_{t \leq T}$ such that
			\begin{itemize}
				\item[(a)] for all $t \in [0,T]$, $y$, $y^{\prime} \in \mathbb{R}$ and $z$, $z^{\prime} \in \mathbb{R}^d$, 
				$$
				\left| f(t,y,z)-f(t,y^{\prime},z^{\prime}) \right| \leq \kappa_t  \left|y-y^{\prime} \right|+\gamma_t  \left\|z-z^{\prime} \right\|,\quad d\mathbb{Q} \otimes dt\text{-a.e.}
				$$
				\item[(b)] There exists $\epsilon >0$ such that $\alpha^2:=\kappa+\gamma^2\geq \epsilon$. 
			\end{itemize}
		\end{itemize}
		\item $\frac{f(\cdot,0,0)}{\alpha_{\cdot}}$ is in $\mathcal{H}^2_{T}(\beta,\mathbb{F},\mathbb{P})$.
	\end{itemize}

	\begin{remark}
		Condition $(\mathbf{H_2})$-(ii)(b) is not an additional restrictive assumption. Indeed, the stochastic Lipschitz coefficients $\kappa$ and $\gamma$ appearing in $(\mathbf{H_2})$-(ii)(a) are not uniquely determined. If a pair $(\kappa,\gamma)$ satisfies the stochastic Lipschitz condition, these processes may be replaced, if necessary, by larger progressively measurable processes without violating $(\mathbf{H_2})$-(ii)(a). Therefore, without loss of generality, the coefficients can always be chosen so that condition $(\mathbf{H_2})$-(ii)(b) is also satisfied. We nevertheless state this condition explicitly since it will be used throughout the subsequent estimates.
	\end{remark}

	The following example is presented in order to illustrate the practical relevance of stochastic Lipschitz conditions in financial modeling. The example is drawn from the pricing of European options in a Black–Scholes framework with stochastic parameters.
	\begin{example}[Stochastic Lipschitz driver]
		Let $T > 0$. Within the flow of information $\mathbb{F}=\left(\mathcal{F}_t\right)_{t \leq T}$ generated by a Brownian motion $B$, consider a complete financial market in which asset prices evolve according to a Black–Scholes model with stochastic coefficients. The market consists of a risk-free asset $(S^0_t)_{t \leq T}$ and a risky asset $(S_t)_{t \leq T}$ governed by the system:
		\begin{equation}\label{BS}
			\left\lbrace
			\begin{split}
				dS^0_t &= r_t S^0_t\, dt, \qquad\qquad\qquad S^0_0 = 1, \\
				dS_t &= S_t \left( \mu_t\, dt + \sigma_t\, dB_t \right), \quad S_0 = x > 0,
			\end{split}
			\right.
		\end{equation}
		where $(r_t, \mu_t, \sigma_t)$ are $\mathbb{F}$-predictable processes satisfying
		$$
		\int_0^T \left(r_s + |\mu_s| + |\sigma_s|^2 \right)\, ds < +\infty, \quad \mathbb{P}\text{-a.s.}
		$$
		
		We consider the pricing of a European call option with strike $K$ and maturity $T$. Its payoff is given by $\xi = (S_T - K)^+$. A financial strategy is represented by a pair of processes $\varPi = (\pi^0_t, \pi_t)_{t \in [0,T]}$, where $\pi^0$ is an $\mathbb{F}$-adapted process representing the amount invested in the risk-free asset $S^0$, and $\pi$ is an $\mathbb{F}$-predictable process representing the amount invested in the risky asset $S$. The corresponding portfolio value process is given by:
		$$
		\mathbf{V}^{\varPi}_t := \pi^0_t S^0_t + \pi_t S_t, \quad t \in [0,T].
		$$
		If the strategy $\varPi = (\pi^0, \pi)$\footnote{Provided that the stochastic integrals appearing in the dynamics are well defined.} is self-financing, then
		\begin{equation*}
			\begin{split}
				d\mathbf{V}^{\varPi}_t &= \pi^0_t\,dS^0_t + \pi_t\,dS_t
				= r_t \mathbf{V}^{\varPi}_t\, dt + \sigma_t \pi_t\, S_t \left(\theta_t\, dt + dB_t\right),
			\end{split}
		\end{equation*}
		where $\theta_t := \frac{\mu_t - r_t}{\sigma_t}$. Setting $Z^{\varPi}_t := \sigma_t \pi_t S_t$, we obtain:
		$$
		d\mathbf{V}^{\varPi}_t = \left(r_t \mathbf{V}^{\varPi}_t + \theta_t Z^{\varPi}_t \right)\, dt + Z^{\varPi}_t\, dB_t.
		$$
		Thus, the value process under a replicating self-financing strategy satisfies the following BSDE:
		\begin{equation*}\label{EDSR lin}
			\left\lbrace 
			\begin{split}
				d \mathbf{V}^{\varPi}_t &= -f(t,\mathbf{V}^{\varPi}_t,Z^{\varPi}_t)\, dt + Z^{\varPi}_t\, dB_t, \quad t \in [0,T], \\
				\mathbf{V}^{\varPi}_T &= \xi,
			\end{split}
			\right. 
		\end{equation*}
		with driver $f(t,y,z) = -r_t y - \theta_t z$, which is clearly a stochastic Lipschitz driver. This illustrates how such drivers naturally arise in financial applications involving stochastic market parameters, including time-varying interest rates and volatilities. 
		
		Additionally, under the Immersion property, the use of stochastic Lipschitz drivers for valuing contingent claims in a complete market model with default and stochastic parameters using the theory of DRBSDEs has been recently presented in \cite{EM2025}.
	\end{example}

	Now, following the spirit of Lemma \ref{Lemma Merties}, we establish a characterization of the integrability of the driver for our DRBSDE \eqref{basic equation} within the frameworks $(\mathcal{G}_{T \wedge \tau}, \mathbb{Q})$ and $(\mathcal{F}_T, \mathbb{P})$.
	\begin{lemma}
		If $\left( \frac{f(t,0,0)}{\alpha_{t}}\right)_{t \geq 0} \in \mathcal{H}^2_T(\beta,\mathbb{F},\mathbb{P})$, then $\left( \frac{f(t,0,0)}{\alpha_{t}}\right)_{t \geq 0}\in \mathcal{H}^2_{T \wedge \tau}(\beta,\mathbb{G},\mathbb{Q})$.
		\label{proof of lemma integrable}
	\end{lemma}
	\begin{proof}
		Let $\Phi$ be the continuous increasing  process defined by $\Phi_t:=\int_{0}^{t} e^{\beta \mathcal{A}_s} \left| \frac{f(s,0,0)}{\alpha_{s}}\right|^2  ds$ and
		consider  a $\mathbb{G}$-stopping time $\sigma \in \mathfrak{T}_{0,T}(\mathbb{G})$. The Bayes rule implies that
		\begin{equation}
			\mathbb{E}_{\mathbb{Q}}\left[\Phi_{\sigma \wedge \tau}\right]
			=\mathbb{E}^{}\left[\int_{0}^{\sigma \wedge \tau}  e^{\beta \mathcal{A}_s}\left|  \frac{f(s,0,0)}{\alpha_{s}}\right|^2  \Psi_s ds\right]=\mathbb{E}\left[\int_{0}^{\sigma \wedge \tau}   \Psi_s d\Phi_s\right].
			\label{Bayes Rule}
		\end{equation}
		Since $\Phi$ is an $\mathbb{F}$-adapted process with continuous paths, we can apply an integration by parts formula, leading to
		\begin{equation}
			\Phi_{\sigma \wedge \tau} \Psi_{\sigma \wedge \tau}=\int_{0}^{\sigma \wedge \tau} \Phi_s d \Psi_s +\int_{0}^{\sigma \wedge \tau}   \Psi_s d\Phi_s,\quad \mathbb{P}^{}\text{-a.s.}
			\label{taking the expectation with fundame}
		\end{equation}
		Now, following a similar calculation to the one used in the proof of Lemma \ref{Lemma Merties}, we obtain
		\begin{equation}
			\mathbb{E}\left[\Phi_{\sigma \wedge \tau} \Psi_{\sigma \wedge \tau}\right]= \mathbb{E}\left[\tilde{\mathcal{E}}_{\sigma} \Phi_{\sigma}+\int_{0}^{\sigma} \Phi_s d V^{\mathbb{F}}_s\right],
			\label{equation we use this}
		\end{equation}
		where the process $V^{\mathbb{F}}$ is defined in the statement of Lemma \ref{Lemma Merties}. Additionally, 
		recall that $\tilde{\mathcal{E}} \leq 1$, which implies that $V^{\mathbb{F}} \leq 1$. Consequently, since $dV^{\mathbb{F}}$ is a positive measure, from (\ref{equation we use this}), we obtain
		\begin{equation}
			\mathbb{E}\left[\Phi_{\sigma \wedge \tau} \Psi_{\sigma \wedge \tau}\right]
			\leq 2  \mathbb{E}\left[ \sup_{t \geq 0}\left| \Phi_{t \wedge \sigma}\right|  \right] 
			\leq 2 \mathbb{E}\left[ \int_{0}^{T} e^{\beta \mathcal{A}_s} \left| \frac{f(s,0,0)}{\alpha_{s}}\right|^2  ds \right].
			\label{lebesgue}
		\end{equation}
		Now, let  $\{\sigma_k\}_{k\in \mathbb{N}}$ be a fundamental sequence for the $\mathbb{G}$-local martingale $\int_{0}^{\cdot} \Phi_s d \Psi_s$. Since $\sigma \in \mathfrak{T}_{0,T}(\mathbb{G})$ in (\ref{lebesgue}) is arbitrary, we can choose $\sigma=\sigma_k \wedge T$ in (\ref{taking the expectation with fundame}). Then, by taking the expectation with respect to  $\mathbb{P}$ along with (\ref{Bayes Rule}), we deduce:
		\begin{equation}
			\mathbb{E}\left[\Phi_{\sigma_k \wedge T \wedge \tau} \Psi_{\sigma_k \wedge T \wedge \tau}\right]=\mathbb{E}\left[\int_{0}^{\sigma_k \wedge T \wedge \tau} \Psi_s d \Phi_s \right]=\mathbb{E}_{\mathbb{Q}}\left[\Phi_{\sigma_k \wedge T \wedge \tau}\right].
			\label{Resume}
		\end{equation}
		Finally, by combining the continuity of the processes $(\Phi_t)_{t \geq 0}$ and $\left(\int_{0}^{t} \Psi_s d \Phi_s\right)_{t \geq 0}$ with Fatou's Lemma, the Lebesgue dominated convergence theorem, (\ref{lebesgue}), and (\ref{Resume}), we derive the desired result.
	\end{proof}
	\textbf{Assumption} ($\mathbf{H_3}$)\textbf{.} We assume that: 
	\begin{itemize}
		\item The obstacles $L:=(L_t)_{t \leq T}$ and  $U:=(U_t)_{t \leq T}$ are RCLL $\mathbb{F}$-adapted processes with jumps occurring only at $\tau$, such that 
		\begin{equation}
			\left\{
			\begin{split}
				&	L_t < U_t,~\forall t \in [0,T\wedge \tau[ \text{ and } L_{t-} < U_{t-},~\forall t \in [0,T\wedge \tau],~\mathbb{Q}\text{-a.s.}\\
				&L_{T \wedge \tau} \leq \xi \leq U_{T \wedge \tau},~
				L_t=L_{T \wedge \tau} \text{ and } U_t=U_{T \wedge \tau} \text{ on the set } \{ t \geq T \wedge \tau\}.
			\end{split}
			\right.
			\label{Barriers Definition}
		\end{equation}
		\item $L^{+} \in \mathcal{S}^2_{T}(2\beta,\mathbb{F},\mathbb{P})$ and $U^{-} \in \mathcal{S}^2_{T}(2\beta,\mathbb{F},\mathbb{P})$.
	\end{itemize}

	\begin{remark}
		Since $\mathbb{F}\subseteq\mathbb{G}$, the $\mathbb{F}$-adapted barriers $L$ and $U$ are also $\mathbb{G}$-adapted. Moreover, since $\tau$ is a $\mathbb{G}$-stopping time, the restrictions of the barriers to the stochastic interval $\llbracket 0,T\wedge\tau\rrbracket$ are well defined in the enlarged filtration. Hence, all the barrier conditions appearing in Assumption $(\mathbf{H}_3)$ are understood in the $(\mathbb{G},\mathbb{Q})$ framework on the random horizon $T\wedge\tau$.
	\end{remark}

	%
	Now, by applying the same reasoning used in the proof of Lemmas \ref{Lemma Merties} and \ref{proof of lemma integrable}, we can establish the following result:
	\begin{lemma}
		If $(L_t,U_t)_{t \leq T}$ satisfies (\ref{Barriers Definition}), and $(L^{+},U^{-}) \in \left( \mathcal{S}^2_T(2\beta,\mathbb{F},\mathbb{P})\right)^2$, then $(L^{+},U^{-}) \in \left( \mathcal{S}^2_{T \wedge \tau}(2\beta,\mathbb{G},\mathbb{Q})\right)^2$.
		\label{proof of measurability L,U}
	\end{lemma}
	\textbf{Assumption} ($\mathbf{H_4}$)\textbf{.} The strictly separated hypothesis on the barriers (\ref{Barriers Definition}) can be strengthened by assuming the existence of a semimartingale $(\mathcal{J}_t)_{t \geq 0}$ of the form: 
	\begin{equation*}
		\left\{
		\begin{split}
			\text{(i)} &~ \mathbb{Q}\text{-a.s. for all } t \in [0,T]\\
			&~~ \mathcal{J}_{t \wedge \tau}=\xi+\left( \mathcal{K}^{+}_{ T \wedge \tau}-\mathcal{K}^{+}_{t \wedge \tau}\right)-\left( \mathcal{K}^{-}_{ T \wedge \tau}-\mathcal{K}^{-}_{t \wedge \tau}\right) -\int_{t \wedge \tau }^{T \wedge \tau} \mathcal{Z}_s d B_s-\int_{t \wedge \tau }^{T \wedge \tau}d\mathcal{N}_s;\\
			\text{(ii)} &~  L_t \leq \mathcal{J}_t \leq U_t,~  \forall t \in [0,T \wedge \tau],~\mathbb{Q}\text{-a.s.};\\
			\text{(iii)} &~   \mathcal{J}_t=\xi,~\mathcal{Z}_t=d\mathcal{K}^{+}_t=d\mathcal{K}^{-}_t=d\mathcal{N}_t=0,\text{ on the set }\{t \geq T \wedge\tau\},~\mathbb{Q}\text{-a.s.}
		\end{split}
		\right.
	\end{equation*}
	Moreover, we assume that $\mathcal{Z} \in \mathcal{H}^2_{T \wedge \tau}(\mathbb{G},\mathbb{Q})$, $\mathcal{N} \in \mathcal{M}^2_{T \wedge \tau}(\mathbb{G},\mathbb{Q})$, and that $\mathcal{K}^{+}$, $\mathcal{K}^{-}$ are two non decreasing continuous processes satisfying $\mathbb{E}\left| \mathcal{K}^{\pm}_{T \wedge \tau}\right|^2<+\infty$.

	\begin{remark}
		Assumption $(\mathbf{H_4})$ should not be confused with the classical
		Mokobodzki's condition. It is a weaker semimartingale compatibility
		condition adapted to the completely separated-barrier framework
		considered here. In particular, as established in \cite[pp. 8--9]{Elmansouri2025},
		the assumption
		$(\mathbf{H_4})$ is weaker than the classical Mokobodzki's condition
		usually imposed in the theory of doubly reflected BSDEs. Finally, the role of $(\mathbf{H_4})$ is to provide the semimartingale
		compatibility required in the existence argument.
	\end{remark}

	The four conditions ($\mathbf{H_1}$)-($\mathbf{H_4}$) on the data $(\xi, f, L, U)$ mentioned above are denoted collectively as \textbf{(H)}.
	
	Let $(\xi^i,f^i,L^i,U^i)_{i=1,2}$ be two sets of data satisfying assumption \textbf{(H)}. Let $(Y^i,Z^i,K^{i,+},K^{i,-},M^i)$ denote a solution of the DRBSDE (\ref{basic equation})  with data $(\xi^i,f^i,L^i,U^i)$. Set $\bar{\mathfrak{R}}=\mathfrak{R}^1-\mathfrak{R}^2$ for $\mathfrak{R}=Y, Z, K^{\pm}, M, \xi, f, L$ and $U$.\\
	By using a similar line of reasoning as in the proof of Proposition 3.1 in \cite{BDEM}, we can prov the following result:
	\begin{proposition}
		For $\beta>3$, there exists a constant $\mathfrak{c}_{\beta}>0$ such that 
		\begin{equation*}
			\begin{split}
				&\left\|\bar{Y}\right\|^2_{\mathcal{S}^{2}_{T \wedge \tau}(\beta,\mathbb{G},\mathbb{Q})}+\left\|\bar{Y}\right\|^2_{\mathcal{S}^{2,\alpha}_{T \wedge \tau}(\beta,\mathbb{G},\mathbb{Q})}+\left\|\bar{Z}\right\|^2_{\mathcal{H}^2_{T \wedge \tau}(\beta,\mathbb{G},\mathbb{Q})}+\left\|\bar{M}\right\|^2_{\mathcal{M}^2_{T \wedge \tau}(\beta,\mathbb{G},\mathbb{Q})}\\
				&\leq \mathfrak{c}_{\beta}\mathbb{E}_{\mathbb{Q}}\left[ e^{\beta \mathcal{A}_{T \wedge \tau}} \left|\bar{\xi} \right|^2
				+\int_{0}^{T \wedge \tau}  e^{\beta \mathcal{A}_{s}} \left|\dfrac{\bar{f}(s,Y^2_s,Z^2_s)}{\alpha_{s}} \right|^2 ds +\int_{0}^{T \wedge \tau}e^{\beta \mathcal{A}_s}\left(\bar{L}_sd\bar{K}^+_s+\bar{U}_sd\bar{K}^-_s\right) \right].
			\end{split}
		\end{equation*}
		\label{basic Estimastion for BSDE} 
	\end{proposition}
	A direct consequence of Proposition \ref{basic Estimastion for BSDE} is the uniqueness result of the solution for the DRBSDE (\ref{basic equation}), stated as follows:
	\begin{corollary}
		Under assumption \textbf{(H)}, there exists at most one $\mathbb{G}$-adapted process $\left(Y_t, Z_t ,K^{+}_t,K^{-}_t, M_t  \right)_{t \geq 0}$ which belongs to $\mathfrak{B}^2_{T \wedge \tau}\left( \beta,\mathbb{G},\mathbb{Q}\right)$ and solves equation (\ref{basic equation}).
		\label{uniqueness}
	\end{corollary}

	We now state the existence and uniqueness result for the DRBSDE in the present default-time framework. Its formulation and proof are closely related to the corresponding result for DRBSDEs driven by an RCLL martingale in a general filtration established in \cite{BDEM}, with the additional difficulties arising from the presence of the second reflecting barrier. The existence argument relies on the penalization method, while the general stochastic Lipschitz case is obtained through a fixed-point argument.
	
	\begin{theorem}
		Assume that \textbf{(H)} holds for a sufficiently large $\beta >0$. Then the DRBSDE (\ref{basic equation}) has a unique solution $(Y_t,Z_t,K^{+}_t,K^{-}_t,M_t)_{t \geq 0}$.
		\label{Existence and Uniqueness Theorem}
	\end{theorem}
	\begin{proof}
		For the existence, as usual, the proof is divided into two stages:
		\begin{itemize}
			\item[(i)] \textit{Step 1:  The case when $f$ does not depend on $(y,z)$.}\\
			We first consider the case when the driver $f$ does not depend on the solution, i.e. $f$ satisfies the following $f(\omega,t,y,z)=:g(\omega,t)$ for all $(\omega,t,y,z) \in \Omega \times [0,T] \times \mathbb{R} \times \mathbb{R}^d$. 
			By using Theorems \ref{Existence and uniquess theorem BSDE}, \ref{Comparison theorem BSDE}, Lemmas \ref{proof of lemma integrable}, \ref{proof of measurability L,U}, and Remark \ref{Remark on the construction of the BSDE}, one can show, proceeding as in \cite[Theorem 4.1]{BDEM}, that there exists a solution $(Y_t,Z_t,K^+_t,K^-_t,M_t)_{t \geq 0}$ of the DRBSDE (\ref{basic equation}) associated with $(\xi,g,L,U)$ in the sense of Definition \ref{definition}.
			\item[(ii)] \textit{Step 2:  Case of a driver $f$ depending on $(y,z)$.}\\
			The proof in the case of a general stochastic Lipschitz driver is similar to the proof of Theorem 4.2 in \cite{BDEM}, relying on a usual fixed-point argument and a priori estimates for DRBSDE (\ref{basic equation}) provided in Proposition \ref{basic Estimastion for BSDE}.
		\end{itemize}
	\end{proof}
	\subsection{Explicit connections between DRBSDE solutions in $\mathbb{F}$ and $\mathbb{G}$ filtrations}
	We aim to establish an explicit relationship between the state solutions of DRBSDEs in the global flow of information $\mathbb{G}$ and the partial flow $\mathbb{F}$, particularly when the generator does not depend on the variables $(y,z)$, i.e., $g(\omega,t):=f(\omega,t,y,z)$ for all $(\omega,t,y,z) \in \Omega \times [0,T] \times \mathbb{R} \times \mathbb{R}^d$. 
	
	Two distinct linkages are employed to do this:
	\begin{enumerate}
		\item The first one gives the connection between DRBSDE solutions in the  $\left(\mathbb{G},\mathbb{Q}\right)$ and $\left(\mathbb{F},\mathbb{P}\right)$ frameworks.
		
		\item The second link connects the DRBSDE solutions in the  $\left(\mathbb{G},\mathbb{Q}\right)$ and $\left(\mathbb{F},\mathbb{Q}\right)$ frameworks. Furthermore, we provide the corresponding result when the driver $f$ is
		linear with respect to the $y$-variable.
	\end{enumerate}
	These findings include the important relationships between the stated filtrations $\mathbb{F}$ and $\mathbb{G}$ for both the newly introduced probability $\mathbb{Q}$ and the reference probability $\mathbb{P}$. 
	\subsubsection{First Link} 
	Our goal is to connect the solutions of DRBSDEs in the defaultable framework $\left(\mathbb{G},\mathbb{Q}\right)$ and the initial setting $\left(\mathbb{F},\mathbb{P}\right)$. This approach is inspired by the work of Alsheyab and Choulli \cite{alsheyab2023optimal}.
	\begin{theorem}
		Let's define:
		\begin{equation*}
			g^{\mathbb{F}}(s):=\tilde{\mathcal{E}}_sg(s),~U^{\mathbb{F}}_s=\tilde{\mathcal{E}}_s U_s,~L^{\mathbb{F}}_s:= \tilde{\mathcal{E}}_s L_{s},~ \xi^{\mathbb{F}}=\tilde{\mathcal{E}}_T\zeta_{T} \text{ and } V^{\mathbb{F}}=1-\tilde{\mathcal{E}}.
			\label{Definition of parametrs}
		\end{equation*} 
		Assume that conditions \textbf{(H)} hold, and that there exist a special semimartingale $\left( \mathcal{J}_t^{\mathbb{F}}\right)_{t \leq T}$ of the following form:
		\begin{equation*}
			\begin{split}
				\mathcal{J}^{\mathbb{F}}_{t}=\xi^{\mathbb{F}}+\left( \mathcal{K}^{\mathbb{F},+}_{T}-\mathcal{K}^{\mathbb{F},,+}_{t}\right)-\left( \mathcal{K}^{\mathbb{F},-}_{ T}-\mathcal{K}^{\mathbb{F},-}_{t}\right) -\int_{t }^{T} \mathcal{Z}^{\mathbb{F}}_s d B_s,\quad t\in[0,T],
			\end{split}
		\end{equation*}
		with $L^{\mathbb{F}}_t \leq \mathcal{J}^{\mathbb{F}}_t \leq U^{\mathbb{F}}_t$,  $\forall t \in [0,T]$, $\mathbb{P}\text{-a.s.}$, $\left(\mathcal{K}^{\mathbb{F},\pm}_t\right)_{t \leq T} \in \mathcal{S}^2_T(\mathbb{F},\mathbb{P})$ and $\left( \int_{0}^{t}\mathcal{Z}^{\mathbb{F}}_s dB_s\right)_{t \leq T} \in \mathcal{H}^2_{T}(\mathbb{F},\mathbb{P})$. Moreover, we assume:
		\begin{equation*}
			\mathbb{E}\left[e^{\beta \mathcal{A}_T}|\zeta_T|^2 + \int_{0}^{T} e^{\beta \mathcal{A}_s}|\zeta_s|^2 d V^{\mathbb{F}}_s \right] < +\infty.
			\label{H}
		\end{equation*}
		
		Let $(Y^{\mathbb{F}}_t,Z^{\mathbb{F}}_t,K^{\mathbb{F},+}_t,K^{\mathbb{F},-}_t)_{t \leq T}$ be the   $(\mathbb{F},\mathbb{P})$-solution of the following DRBSDE:
		\begin{equation}
			\left\{
			\begin{split}
				\text{(i)} &~ \mathbb{P}\text{-a.s. for all } t \in [0,T]\\
				& Y^{\mathbb{F}}_{t}=\xi^{\mathbb{F}}+\int_{t}^{T} g^{\mathbb{F}}(s)ds+\int_{t}^{T}\zeta_s d V^{\mathbb{F}}_s+\left( K^{\mathbb{F,}+}_{ T }-K^{\mathbb{F},+}_{t}\right)\\
				&\qquad\qquad\qquad-\left( K^{\mathbb{F},-}_{T}-K^{\mathbb{F},-}_{t}\right) -\int_{t }^{T} Z^{\mathbb{F}}_s d B_s;\\
				\text{(ii)} &~  L^{\mathbb{F}}_t \leq Y^{\mathbb{F}}_t \leq U^{\mathbb{F}}_t,~  \forall t \in [0,T ],~\mathbb{P}\text{-a.s.};\\
				\text{(iii)} &~ \int_0^{T }(Y^{\mathbb{F}}_{t}-L^{\mathbb{F}}_{t})dK^{\mathbb{F},+}_t=\int_0^{T \wedge \tau}(U^{\mathbb{F}}_{t}-Y^{\mathbb{F}}_{t})dK^{\mathbb{F},-}_t=0,~\mathbb{P}\text{-a.s.},
			\end{split}
			\right.
			\label{basic equation comparison}
		\end{equation}
		associated with data $\left(\xi^{\mathbb{F}},g^{\mathbb{F}},L^{\mathbb{F}},U^{\mathbb{F}}\right)$.
		
		Let $(Y_t, Z_t, K^{+}_t, K^{-}_t, M_t)_{t \geq 0}$ be the $(\mathbb{G},\mathbb{Q})$-solution of the DRBSDE (\ref{basic equation}) associated with $(\xi,g,L,U)$. 
		Then, we have 
		\begin{equation*}
			\left\{
			\begin{split}
				&Y_t= \dfrac{Y^{\mathbb{F}}_t}{\tilde{\mathcal{E}}_t} \mathds{1}_{\{t<T \wedge\tau \}}+\xi \mathds{1}_{\{ t \geq T \wedge \tau \}},\quad K^{\pm}_t=\int_{0}^{t} \mathds{1}_{\{s \leq \tau\}} \dfrac{1}{\tilde{\mathcal{E}}_{s-}} dK^{\mathbb{F},\pm}_{s},\\
				&Z_t=\mathds{1}_{\{t \leq T\wedge \tau\}}\dfrac{Z^{\mathbb{F}}_t}{\tilde{\mathcal{E}}_{t-}},\quad \text{ and }\quad M_t=\int_{0}^{t}\mathds{1}_{\{s \leq T\}} \left(\zeta_s-\dfrac{Y^{\mathbb{F}}_{s}}{\tilde{\mathcal{E}}_{s}}\right) dN^{\mathbb{G}}_s
				,\quad\forall t \geq 0,~\mathbb{P}\text{-a.s.},
			\end{split}
			\right.
		\end{equation*}
		where $(N^{\mathbb{G}}_t)_{t \geq 0}$ is a $(\mathbb{G},\mathbb{P})$-martingale with integrable variation (see \cite{choulli2020martingale}) defined by:
		$$
		N^{\mathbb{G}}_t:=D_t-\int_{0}^{t} \mathds{1}_{\{s \leq \tau\}} \tilde{G}^{-1}_s dD^{o,\mathbb{F},\mathbb{P}}_s,\quad t \geq 0.
		$$
		\label{Delicate thm}
	\end{theorem}
	\begin{proof}
		In this proof, $\mathfrak{c}>0$ indicates a constant that remains independent of the index of any stochastic process sequences and may vary from line to line.
		
		The proof of Theorem \ref{Delicate thm} relies on some auxiliary results from the theory of reflected BSDEs and generalized reflected BSDEs, as well as some new results concerning the optimal control problems across $(\mathbb{F},\mathbb{P})$ and $(\mathbb{G},\mathbb{Q})$, which will be addressed briefly.
		
		First, it's worth noting that the existence and uniqueness of the solution for the DRBSDE (\ref{basic equation comparison}) are established in \cite[Theorem 3.3]{klimsiak2015reflected}. Further, the $\mathbb{L}^2$-integrability property for the $(\mathbb{F},\mathbb{P})$-solution can be classically obtained as in \cite{BDEM} under our hypothesis. Meanwhile, the existence and uniqueness result for the DRBSDE (\ref{basic equation}) is provided in Theorem \ref{Existence and Uniqueness Theorem}.
		
		The rest of the proof is divided into four main steps.
		\paragraph*{Step 1: Construction of the RBSDE penalization schemes in $(\mathbb{G},\mathbb{Q})$.}\emph{}\\
		Let $(Y^n_t,Z^n_t,K^{n,+}_t,M^n_t)_{t \geq 0}$ be the unique solution of the RBSDE associated with parameters $(\xi,g(t \wedge \tau)d(t \wedge \tau)-n(y-U_t)^{+},L)$. This means that for each $n \in \mathbb{N}$, we have
		\begin{equation}
			\left\{
			\begin{split}
				\text{(i)} &~ \mathbb{Q}\text{-a.s., for all } t \in [0,T],\\
				&~ Y^n_{t}=\xi+\int_{t}^{T} \mathds{1}_{\{s \leq \tau\}} g(s)ds-n\int_{t}^{T}(Y^n_s-U_s)^{+}ds+\left(  K^{n,+}_{T}-K^{n,+}_{t}\right) \\ &\qquad\qquad-\int_{t}^{T} Z^n_s d B_s-\int_{t }^{T }dM^n_s;\\
				\text{(ii)} &~  L_t \leq Y^n_t,~  \forall t \in [0,T \wedge \tau[,~\mathbb{Q}\text{-a.s.};\\
				\text{(iii)} &~ L_{T \wedge \tau} \leq Y_{T \wedge \tau},~\mathbb{Q}\text{-a.s.};\\
				\text{(iv)} &~   Y^n_t=\xi,~Z^n_t=dK^{n,+}_t=dM^n_t=0,\text{ on the set }\{t \geq T \wedge\tau\},~\mathbb{Q}\text{-a.s.};\\
				\text{(v)} &~ \text{Skorokhod condition:}  \int_0^{T \wedge \tau}(Y^n_{t}-L_{t})dK^{n,+}_t=0,~\mathbb{Q}\text{-a.s.}
			\end{split}
			\right.
			\label{basic equation RBSDE}
		\end{equation}
		The existence and uniqueness of the RBSDE
		\eqref{basic equation RBSDE} in the space
		$\mathfrak{B}^2_{T\wedge\tau}
		\left(\beta,\mathbb{G},\mathbb{Q}\right)$
		follow from the corresponding existence and uniqueness result for
		reflected BSDEs in the present default-time framework; see
		\cite[Theorem 8]{alsheyab2023optimal}.\\
		We denote $K^{n,-}_t=n\int_{0}^{t}(Y^n_s-U_s)^{+}ds$. It is easy to see that $K^{n,-}_0=0$  and $K^{n,-}_t=K^{n,-}_{t \wedge T \wedge \tau}$ for any $t \geq 0$. Next, by employing a similar argument to the one used in Step 4 of the proof of Theorem 4.1 in \cite{el2022bsdes}, we can prove that the sequence $\left\{Y^n,Z^n,K^{n,+},K^{n,-},M^n\right\}_{n \in \mathbb{N}}$ forms a Cauchy sequence in $\mathfrak{B}^2_{T \wedge \tau}\left( \beta,\mathbb{G},\mathbb{Q}\right)$. Consequently, there exists a process $(Y,Z,K^{+},K^{-},M)$ that belongs to $\mathfrak{B}^2_{T \wedge \tau}\left( \beta,\mathbb{G},\mathbb{Q}\right)$ and serves as the limit of the sequence $\left\{Y^n,Z^n,K^{n,+},K^{n,-},M^n\right\}_{n \in \mathbb{N}}$ in the relevant spaces. This limit process  also satisfies the DRBSDE (\ref{basic equation}).
		\paragraph*{Step 2: RBSDE (\ref{basic equation RBSDE}) and related optimal stopping problems across $(\mathbb{G},\mathbb{Q})$ and $(\mathbb{F},\mathbb{P})$.}\emph{}\\
		Using the optimal stopping representation for reflected BSDEs established in \cite{alsheyab2023optimal,Elmansouri2026Default}, applied to the RBSDE \eqref{basic equation RBSDE}, we obtain
		\begin{equation*}
			\begin{split}
				Y^n_t+\int_{0}^{t \wedge \tau}g(s)ds-\int_{0}^{t \wedge \tau}d K^{n,-}_s=\esssup_{\theta \in \mathfrak{T}_{t \wedge \tau}^{T \wedge \tau}\left(\mathbb{G}\right)}\mathbb{E}_{\mathbb{Q}}\left[X^{\mathbb{G},n}_{\theta} \mid \mathcal{G}_t\right]=:S^{\mathbb{G},\mathbb{Q},n}_t,
			\end{split}
			\label{Mocci oujda}
		\end{equation*}
		where  $(X^{\mathbb{G},n}_{t})_{t \geq 0}$ is the RCLL $\mathbb{G}$-adapted process stopped at $T \wedge \tau$, defined by:
		\begin{equation}
			X^{\mathbb{G},n}=\int_{0}^{\cdot \wedge T \wedge \tau} g(s)ds-\int_{0}^{\cdot \wedge T \wedge \tau}dK^{n,-}_s+L_{}\mathds{1}_{\llbracket 0,T \wedge \tau \llbracket}+\zeta_{T \wedge \tau} \mathds{1}_{\llbracket T \wedge \tau,+\infty \llbracket}.
			\label{reward process}
		\end{equation}
		
		It's clear that
		$
		\int_{0}^{\cdot \wedge T \wedge \tau}dK^{n,-}_s=n\int_{0}^{\cdot \wedge T} \left(Y^n_s \mathds{1}_{\{s <\tau\}}-U_s \mathds{1}_{\{s <\tau\}}\right)^{+} ds
		$. Additionally, since $Y^n=(Y^n)^{\tau}$, using Theorem 2 in \cite{choulli2024optimal}, we derive the existence of a unique RCLL $\mathbb{F}$-progressively measurable process $Y^{n,\mathbb{F}}$ such that  $
		Y^n_s \mathds{1}_{\{s <\tau\}}=Y^{n,\mathbb{F}}_s \mathds{1}_{\{s <\tau\}}
		$. This transformation switches the process $(K^{n,-}_t)_{t \geq 0}$ into $
		K^{n,-}=n\int_{0}^{\cdot \wedge T \wedge \tau}  \left(Y^{n,\mathbb{F}}_s-U_s\right)^{+} ds
		$. Furthermore, Theorem 2 in \cite{choulli2024optimal} implies that the process $\left(Y^{n,\mathbb{F}}_t\right)_{t \leq T}$ can be expressed as: $Y^{n,\mathbb{F}}_t=G^{-1}_t\mathbb{E}\left[Y^n_s \mathds{1}_{\{s <\tau\}} \mid \mathcal{F}_t\right]$ for any $t \in [0,T]$.\\ 
		Let $\beta>0$. Using Jensen's inequality, the fact that $Y^n=\xi$ on $\llbracket T \wedge \tau,+\infty\llbracket$, that the process $(\mathcal{A}_t)_{t \leq T}$ is $\mathbb{F}$-adapted, Lemma \ref{Lemma of integrability}, Remark \ref{Remark of epsilon} and Lemma 4.1 in \cite{BDEM}, we conclude that $Y^{n,\mathbb{F}} \in \mathcal{S}^2_{T}\left(\beta,\mathbb{F},\mathbb{P}\right)$ with $\left\|Y^{n,\mathbb{F}}\right\|^2_{\mathfrak{B}^2_T(\beta,\mathbb{F},\mathbb{P})} \leq \mathfrak{c}$ for any $n \in \mathbb{N}$.
		
		The idea is to find a suitable way to apply Theorem 3 in  \cite{choulli2024optimal}. To do so, we first  need to identify the unique pair $(X^{\mathbb{F},n},k^{(pr),n})$ associated with the process $X^{\mathbb{G},n}$ defined in (\ref{reward process}). This requires finding an $\mathbb{F}$-optional process $X^{\mathbb{F},n}$ and an $\mathbb{F}$-progressively measurable process $k^{(pr),n}$ such that  $X^{\mathbb{G},n}=X^{\mathbb{F},n}\mathds{1}_{\llbracket 0,\tau \llbracket}+\int_{0}^{\cdot} k^{(pr),n}_s dD_s$. However, by performing simple computations, we get:
		\begin{equation*}
			\left\{
			\begin{split}
				X^{\mathbb{F},n}&=\int_{0}^{\cdot \wedge T} g(s)ds-n\int_{0}^{\cdot \wedge T } \left(Y^{n,\mathbb{F}}_s-U_s \right)^{+} ds+L_{}\mathds{1}_{\llbracket 0,T   \llbracket}+\zeta_{T} \mathds{1}_{\llbracket T ,+\infty \llbracket},\\
				k^{(pr),n}&=\int_{0}^{\cdot \wedge T} g(s)ds-n\int_{0}^{\cdot \wedge T } \left(Y^{n,\mathbb{F}}_s-U_s \right)^{+} ds+\zeta_{T \wedge \cdot}.
			\end{split}
			\right.
			\label{Associated process}
		\end{equation*}
		Let $\mathcal{X}^{\mathbb{F},n}:=\left(\tilde{\mathcal{E}}X^{\mathbb{F},n} -\int_{0}^{\cdot}k^{(op),n}_s d\tilde{\mathcal{E}}_s\right)^T$. Using  integration by part formula in the $(\Omega,\mathcal{F},\mathbb{F},\mathbb{P})$-setup and the fact that $d V^{\mathbb{F}}=-d\tilde{\mathcal{E}}$, we derive
		\begin{equation}
			\begin{split}
				\mathcal{X}^{\mathbb{F},n}&=\int_{0}^{\cdot \wedge T} g^{\mathbb{F}}(s)ds-n\int_{0}^{\cdot \wedge T }  \left(\hat{Y}^{n,\mathbb{F}}_s-U^{\mathbb{F}}_s \right)^{+} ds+\int_{0}^{\cdot \wedge T}\zeta_s dV^{\mathbb{F}}_s+L^{\mathbb{F}}\mathds{1}_{\llbracket 0,T   \llbracket}+\xi^{\mathbb{F}} \mathds{1}_{\llbracket T ,+\infty \llbracket},
			\end{split}
			\label{Chef de la vie}
		\end{equation}
		with $	\hat{Y}^{n,\mathbb{F}}:=\tilde{\mathcal{E}}Y^{n,\mathbb{F}}$ in $[0,T]$. As $0<\tilde{\mathcal{E}}\leq 1$, $\tilde{\mathcal{E}}$ is an RCLL $\mathbb{F}$-adapted process, and $Y^{n,\mathbb{F}} \in \mathcal{S}^2_{T}\left(\beta,\mathbb{F},\mathbb{P}\right)$, we can derive $\hat{Y}^{n,\mathbb{F}} \in \mathcal{S}^2_{T}\left(\beta,\mathbb{F},\mathbb{P}\right)$ with $\left\|\hat{Y}^{n,\mathbb{F}}\right\|^2_{\mathfrak{B}^2_T(\beta,\mathbb{F},\mathbb{P})} \leq \mathfrak{c}$ for all $n \in \mathbb{N}$.\\
		By applying Theorem 3 in \cite{choulli2024optimal}, we have:
		\begin{equation}
			\begin{split}
				&Y^n+\int_{0}^{\cdot \wedge T \wedge \tau}g(s)ds-\int_{0}^{\cdot \wedge T \wedge \tau}d K^{n,-}_s\\
				&=\dfrac{\mathcal{\hat{S}}^{\mathbb{F},n}}{\tilde{\mathcal{E}}^T}\left(\mathds{1}_{\llbracket 0, \tau\llbracket}\right)^T+\int_{0}^{\cdot \wedge T}\left(k^{(pr),n}_s-\dfrac{\int_{0}^{s}k^{(pr),n}_u d\tilde{\mathcal{E}}_u}{\tilde{\mathcal{E}}_s}\right) dN^{\mathbb{G}}_s,
			\end{split}
			\label{Stable 2}
		\end{equation}
		where $\mathcal{\hat{S}}^{\mathbb{F},n}_t=\esssup_{\theta \in \mathfrak{T}_{t}^{T}(\mathbb{F})}\mathbb{E}\left[\mathcal{X}^{\mathbb{F},n}_{\theta} \mid \mathcal{F}_t\right]$ for $t \in [0,T]$, with $\left( \mathcal{X}^{\mathbb{F},n}_t\right)_{t \geq 0}$ given by (\ref{Chef de la vie}).
		\paragraph*{Step 3: Optimal stopping problem and related generalized RBSDE in $(\mathbb{F},\mathbb{P})$}
		The optimal stopping problem
		$\left(\mathcal{\hat{S}}^{\mathbb{F},n}_t\right)_{t \leq T}$
		associated with the reward process defined in
		(\ref{Chef de la vie}) leads us to consider the following generalized
		reflected BSDE (GRBSDE, for short):
		\begin{equation}
			\left\{
			\begin{split}
				\text{(i)}&~ \mathbb{P}\text{-a.s., for all } t \in [0,T],\\ 
				& \hat{Y}^{n,\mathbb{F},\ast}_t=\xi^{\mathbb{F}}+\int_{t}^{T}dV^{n,\mathbb{F},\ast}_s+(K^{n,\mathbb{F},+}_T-K^{n,\mathbb{F},\ast}_t)-\int_{t  }^{T} \hat{Z}^{n,\mathbb{F},\ast}_s d B_s;\\
				\text{(ii)} &~  L^{\mathbb{F}}_t \leq \hat{Y}^{n,\mathbb{F},\ast}_t,~  \forall t \in [0,T],~\mathbb{P}\text{-a.s.};\\
				\text{(iii)} &~ \text{Skorokhod condition:}  \int_0^{T }(\hat{Y}^{n,\mathbb{F},\ast}_{t}-L^{\mathbb{F}}_{t})dK^{n,\mathbb{F},+}_t=0,~\mathbb{P}\text{-a.s.},
			\end{split}
			\right.
			\label{basic equation RBSDE in F,P}
		\end{equation}
		where the finite-variation generator is given by
		\[
		V^{n,\mathbb{F},\ast}_t
		=
		\int_{0}^{t}
		\big\{
		g^{\mathbb{F}}(s)
		-
		n(\hat{Y}^{n,\mathbb{F}}_s-U^{\mathbb{F}}_s)^{+}
		\big\}ds
		+
		\int_{0}^{t}\zeta_s\,dV^{\mathbb{F}}_s,
		\qquad t\in[0,T].
		\] 

		Note that in the GRBSDE (\ref{basic equation RBSDE in F,P}), the process $K^{n,\mathbb{F},+}$ is continuous due to the fact that $L^{\mathbb{F}}=\tilde{\mathcal{E}}L$ could have a jump only at $\tau$. However, the jumps arising from the process $\tilde{\mathcal{E}}$ are equal to those of $\hat{Y}^{n,\mathbb{F},\ast}$ since $\Delta V^{\mathbb{F}}=-\Delta\tilde{\mathcal{E}}$, and this latter jumps when $\Delta D^{o,\mathbb{F},\mathbb{P}}$ does from (\ref{Definition of E and E tilde}), or equivalently when the defaultable $(D_t)_{t \geq 0}$ does at the default time $\tau$ due to the optional duality relation for jump processes (see Theorem 5.27 in \cite{bookHe}, pp. 150). This fact will be used implicitly throughout the rest of the proof when
		considering the reflection processes of the other GRBSDEs.
		
		We now briefly study the existence, uniqueness, and convergence
		properties of the GRBSDE
		(\ref{basic equation RBSDE in F,P}) toward the solution of a particular
		GRBSDE in the $(\mathbb{F},\mathbb{P})$-framework.
		These results will be used in the fourth step to establish the main
		result.
		
		The GRBSDE (\ref{basic equation RBSDE in F,P}) is a particular case of the GRBSDE considered in \cite{fakhouri20192} or in \cite{klimsiak2015reflected} with an associated coefficient $f\equiv 0$. We will use of the results from \cite{fakhouri20192} to obtain a strong $\mathbb{L}^2$-integrability result for the associated solution of (\ref{basic equation RBSDE in F,P}) as we have $\mathbb{L}^2$-integrable data. To begin, noting that $\tilde{\mathcal{E}} \leq 1$, we can infer $L^{\mathbb{F}} \leq \mathbb{E}\left[\sup_{0 \leq t \leq T} (L_s)^{+} \mid \mathcal{F}_t\right]$. Then, under assumption \textbf{(H3)}, we conclude that the last term is a square integrable $(\mathbb{F},\mathbb{P})$-martingale. Our goal now is to show that  $\left|V^{n,\mathbb{F},\ast}\right|^2_T \in \mathbb{L}^2\left(\Omega,\mathcal{F}_T,\mathbb{P}\right)$, where $\left|V^{n,\mathbb{F},\ast}\right|$ denotes the total variation of the process $\left(V^{n,\mathbb{F},\ast}_t\right)_{t \leq T}$. Since $dV^{\mathbb{F}}$ is a positive measure, we have
		\begin{equation*}
			\begin{split}
				\mathbb{E}\left[\left|V^{n,\mathbb{F},\ast}\right|^2_T\right]
				& \leq \dfrac{3}{\beta} \mathbb{E}\left[\int_{0}^{T} e^{\beta \mathcal{A}_s} \left| g(s)\right|^2 ds \right]+3\mathbb{E}\left[\left( \int_{0}^{T}e^{-\beta \mathcal{A}_s}dV^{\mathbb{F}}_s\right)\left( e^{\beta \mathcal{A}_s} \left|\zeta_{s}\right|^2dV^{\mathbb{F}}_s\right) \right]\\
				&\qquad\qquad+3T^2n^2 \mathbb{E}\left[\sup_{0 \leq t \leq T}e^{\beta \mathcal{A}_t} \left\{\left|\hat{Y}^{n,\mathbb{F}}_t\right|^2+\left| U^{-}_t\right|^2\right\}\right]^{\frac{1}{2}}.
			\end{split}
		\end{equation*}
		From Remark \ref{Remark of epsilon}, we derive that $V^{\mathbb{F}} \leq 1$, then $\int_{0}^{T}e^{-\beta \mathcal{A}_s}dV^{\mathbb{F}}_s \leq 1$. Consequently, $\mathbb{E}\left[\left|V^{n,\mathbb{F},\ast}\right|^2_T\right]<+\infty$. Theorem 3 in \cite{fakhouri20192} guarantees the existence of a unique triplet of $\mathbb{F}$-progressively measurable processes $\left(\hat{Y}^{n,\mathbb{F},\ast}_t,K^{n,\mathbb{F},+}_t,\hat{Z}^{n,\mathbb{F},\ast}_t\right)_{t \leq T}$ satisfying GRBSDE (\ref{basic equation RBSDE in F,P}) with  $\left(\hat{Y}^{n,\mathbb{F},\ast},\hat{Z}^{n,\mathbb{F},\ast},K^{n,\mathbb{F},+}_T\right) \in \mathcal{D}^2_T\left(\mathbb{F},\mathbb{P}\right) \times \mathcal{H}^2_T(\mathbb{F},\mathbb{P}) \times \mathcal{S}^2_T\left(\mathbb{F},\mathbb{P}\right)$.
		
		Considering the construction of the process
		$\left(\hat{Y}^{n,\mathbb{F}}_t\right)_{t\leq T}$, we have, for every
		$t\in[0,T]$,
		\begin{equation}
			\hat{Y}^{n,\mathbb{F}}_t
			=
			\tilde{\mathcal{E}}_t Y^{n,\mathbb{F}}_t
			=
			\frac{\tilde{\mathcal{E}}_t}{G_t}
			\mathbb{E}\left[
			Y^n_t \mathds{1}_{\{t<\tau\}}
			\,\middle|\,
			\mathcal{F}_t
			\right].
			\label{projection-Yn-F}
		\end{equation}
		
		Since $U$ is $\mathbb{F}$-adapted, $U_t$ is
		$\mathcal{F}_t$-measurable. Therefore, by the definition of the Az\'ema
		supermartingale $G$,
		\begin{equation}
			\mathbb{E}\left[
			U_t \mathds{1}_{\{t<\tau\}}
			\,\middle|\,
			\mathcal{F}_t
			\right]
			=
			U_t
			\mathbb{E}\left[
			\mathds{1}_{\{t<\tau\}}
			\,\middle|\,
			\mathcal{F}_t
			\right]
			=
			U_t G_t.
		\end{equation}
		
		Recalling that
		$U^{\mathbb{F}}_t=\tilde{\mathcal{E}}_t U_t$, we consequently obtain
		\begin{align}
			\hat{Y}^{n,\mathbb{F}}_t-U^{\mathbb{F}}_t
			&=
			\frac{\tilde{\mathcal{E}}_t}{G_t}
			\mathbb{E}\left[
			Y^n_t\mathds{1}_{\{t<\tau\}}
			\,\middle|\,
			\mathcal{F}_t
			\right]
			-
			\tilde{\mathcal{E}}_t U_t
			\nonumber\\
			&=
			\frac{\tilde{\mathcal{E}}_t}{G_t}
			\mathbb{E}\left[
			\left(Y^n_t-U_t\right)
			\mathds{1}_{\{t<\tau\}}
			\,\middle|\,
			\mathcal{F}_t
			\right].
			\label{difference-Yn-U}
		\end{align}
		
		Hence, using the inequality
		$\left(\mathbb{E}[X\mid\mathcal{F}_t]\right)^{+}
		\leq
		\mathbb{E}[X^{+}\mid\mathcal{F}_t]$, we deduce that
		\begin{equation}
			\left(
			\hat{Y}^{n,\mathbb{F}}_t-U^{\mathbb{F}}_t
			\right)^{+}
			\leq
			\frac{\tilde{\mathcal{E}}_t}{G_t}
			\mathbb{E}\left[
			\left(Y^n_t-U_t\right)^{+}
			\mathds{1}_{\{t<\tau\}}
			\,\middle|\,
			\mathcal{F}_t
			\right].
			\label{positive-part-Yn-U}
		\end{equation}
		Now, employing a similar uniform estimation result from the paper by El Otmani et al. \cite{el2022bsdes} (see estimation (18) on page 14) which holds under our conditions with respect to the probability measure $\mathbb{Q}$, we derive the existence of a constant $\mathfrak{c}>0$ such that
		\begin{equation*}
			\sup_{n \in \mathbb{N}}\mathbb{E}\left(n\int_{0}^{T}(\hat{Y}^{n,\mathbb{F}}_s-U^{\mathbb{F}}_s)^{+}ds\right)^2 \leq  \mathfrak{c}.
		\end{equation*}
		Using this with a standard localization procedure via the sequence of $\mathbb{F}$-stopping times $\{\sigma_k\}_{k \in \mathbb{N}}$ defined by $\sigma_k=\inf\left\{t>0: \mathcal{A}_t>k\right\} \wedge T$ following the same techniques of Proposition \ref{basic Estimastion for BSDE}, we can establish the following uniform a priori estimates:
		\begin{equation}
			\begin{split}
				&\mathbb{E}\sup_{s \in [0, T]} e^{\beta \mathcal{A}_s} \left|\hat{Y}^{n,\mathbb{F},\ast}_{s}\right|^2+\mathbb{E}\int_{0}^{T}e^{\beta \mathcal{A}_s} \left|\hat{Y}^{n,\mathbb{F},\ast}_s \alpha_s\right|^2ds+\mathbb{E}\int_{0}^{T}e^{\beta \mathcal{A}_s}\left|\hat{Z}^{n,\mathbb{F},\ast}_s\right|^2 ds+\mathbb{E}\left| K^{n,\mathbb{F},+}_{T}\right|^2\\
				& \leq \mathfrak{c}\left(  \mathbb{E}\int_{0}^{T}e^{\beta \mathcal{A}_s}\left|\dfrac{g^{\mathbb{F}}(s)}{\alpha_s}\right|^2 ds  + \mathbb{E} \sup_{s \in  [0,T]} e^{2\beta \mathcal{A}_s}\left| \left( L^{\mathbb{F}}_s\right)^{+}\right|^2  +\mathbb{E}\int_{0}^{T}e^{\beta \mathcal{A}_s} \left| \zeta_s\right|^2  dV^{\mathbb{F}}_s \right).
			\end{split}
			\label{EST 2}
		\end{equation}
		From the uniform estimate (\ref{EST 2}) of the sequence $\left\{\hat{Z}^{n,\mathbb{F},\ast}, K^{n,\mathbb{F},+}\right\}_{n \in \mathbb{N}}$ and the Hilbert structure of the associated $\mathbb{L}^2$-spaces, we can extract further subsequences, still denoted by the sequence index $n$, which exhibit a weak convergence in the pertinent topological spaces to a given pair process $\left(\hat{Z}^{\mathbb{F},\ast}, K^{\mathbb{F},+}\right) \in \mathcal{H}^2_T\left(\beta,\mathbb{F},\mathbb{P}\right) \times \mathcal{S}^2_T(\mathbb{F},\mathbb{P})$ with $K^{\mathbb{F},+}$ $\mathbb{F}$-predictable. On the other hand, using Remark \ref{Remark of epsilon} and (\ref{EST 2}), and the Step 4 of the proof of Theorem 4.1 in \cite{el2022bsdes} , we conclude that, the increasing sequence of $\mathbb{F}$-adapted continuous processes $\left\lbrace K^{n,\mathbb{F},-}:=n\int_{0}^{\cdot}(\hat{Y}^{n,\mathbb{F}}_s-U^{\mathbb{F}}_s)^{+}ds\right\rbrace_{n \in \mathbb{N}}$ is a Cauchy sequence which respect to the $\mathcal{S}^2_T(\mathbb{F},\mathbb{P})$-norm. Thus, there exists an $\mathbb{F}$-adapted continuous increasing process $K^{\mathbb{F},-} \in \mathcal{S}^2_T(\mathbb{F},\mathbb{P})$ serving as the strong limit in the $\mathcal{S}^2_T(\mathbb{F},\mathbb{P})$-norm of $\{K^{n,\mathbb{F},-}\}_{n \in \mathbb{N}}$. \\
		Now, for $n>p\geq 0$, applying once again It\^{o}'s formula to the process
		$
		e^{\frac{\beta}{2}\mathcal{A}_t}
		|\hat{Y}^{n,\mathbb{F},\ast}_{t}
		-\hat{Y}^{p,\mathbb{F},\ast}_{t}|^2,
		$
		together with the Skorokhod condition, the Cauchy--Schwarz inequality,
		and assumption $(\mathbf{H_2})$-(b), we obtain
		\begin{equation}
			\begin{split}
				&\mathbb{E}\int_{0}^{T}e^{\frac{\beta}{2} \mathcal{A}_s} \left|\hat{Y}^{n,\mathbb{F},\ast}_s-\hat{Y}^{p,\mathbb{F},\ast}_{s} \right|^2 \alpha_s^2ds+\mathbb{E}\int_{0}^{T}e^{\frac{\beta}{2} \mathcal{A}_s}\left|\hat{Z}^{n,\mathbb{F},\ast}_s-\hat{Z}^{p,\mathbb{F},\ast}_s\right|^2 ds\\
				& \leq  2 \left( \mathbb{E}\int_{0}^{T}e^{\beta \mathcal{A}_s} \left|\hat{Y}^{n,\mathbb{F},\ast}_{s}-\hat{Y}^{p,\mathbb{F},\ast}_{s}\right|^2 \alpha_s^2 ds\right)^{\frac{1}{2}}  \left( \mathbb{E} \left|K^{p,\mathbb{F},-}_T-K^{n,\mathbb{F},-}_T\right|^2\right)^{\frac{1}{2}} \\
				&\leq \mathfrak{c}\left(\mathbb{E} \left|K^{p,\mathbb{F},-}_T-K^{n,\mathbb{F},-}_T\right|^2\right)^{\frac{1}{2}} .
			\end{split}
			\label{Real}
		\end{equation}
		The $\mathbb{L}^2$-strong convergence of the sequence $\{K^{n,\mathbb{F},-}\}_{n \in \mathbb{N}}$ implies: for any $\beta>0$
		$$
		\lim\limits_{n,p\rightarrow+\infty} 
		\left\{\mathbb{E}\int_{0}^{T}e^{\frac{\beta}{2} \mathcal{A}_s} \left|\hat{Y}^{n,\mathbb{F},\ast}_s-\hat{Y}^{p,\mathbb{F},\ast}_{s} \right|^2 \alpha_s^2ds+\mathbb{E}\int_{0}^{T}e^{\frac{\beta}{2} \mathcal{A}_s}\left|\hat{Z}^{n,\mathbb{F},\ast}_s-\hat{Z}^{p,\mathbb{F},\ast}_s\right|^2 ds\right\} =0.
		$$
		Applying the Burkholder-Davis-Gundy inequality (B-D-G) \cite[Theorem IV.48, pp. 193]{bookProtter} with (\ref{Real}), allows us to conclude that, for any $\beta>0$
		$$
		\lim\limits_{n,p\rightarrow+\infty} \mathbb{E}\left[\sup_{s \in [0, T]}e^{\frac{\beta}{2} \mathcal{A}_{s}} \left|\hat{Y}^{n,\mathbb{F},\ast}_{s}-\hat{Y}^{p,\mathbb{F},\ast}_{s}\right|^2\right]=0, 
		$$
		Then there exists a unique RCLL $\mathbb{F}$-adapted process $\left(\hat{Y}^{\mathbb{F},\ast}_{t}\right)_{t \leq T}$ such that the constructed pair $\left(\hat{Y}^{\mathbb{F},\ast},\hat{Z}^{\mathbb{F},\ast}\right)$ satisfies
		$$
		\lim\limits_{n\rightarrow+\infty} \left\{\left\|\hat{Y}^{n,\mathbb{F},\ast}-\hat{Y}^{\mathbb{F},\ast}\right\|^2_{\mathcal{S}^2_T\left( \frac{\beta}{2},\mathbb{F},\mathbb{P}\right)}+\left\|\hat{Z}^{n,\mathbb{F},\ast}-\hat{Z}^{\mathbb{F},\ast}\right\|^2_{\mathcal{H}^2_T\left( \frac{\beta}{2},\mathbb{F},\mathbb{P}\right) } \right\}=0.
		$$ 
		From (\ref{basic equation RBSDE in F,P}), we can also obtain that $\{K^{n,\mathbb{F},+}\}_{n \in \mathbb{N}}$ is a Cauchy sequence in $\mathcal{D}^2_T(\mathbb{F},\mathbb{P})$,
		therefore 
		$$
		\lim\limits_{n\rightarrow+\infty}  \left\|K^{n,\mathbb{F},+}_T-K^{\mathbb{F},+}_T\right\|^2_{\mathcal{D}^2_T(\mathbb{F},\mathbb{P})}=0,$$ 
		and the process $K^{\mathbb{F},+}$ exhibits continuous paths.\\ 
		Passing to the limit element by element in (\ref{basic equation RBSDE in F,P}), we derive that the quadruplet $\left(\hat{Y}^{\mathbb{F},\ast}_{t},\hat{Z}^{\mathbb{F},\ast}_{t},K^{\mathbb{F},+}_{t},K^{\mathbb{F},-}_{t}\right)_{t \leq T}$ satisfies the following BSDE, for any $t \in [0,T]$,  
		$$
		\hat{Y}^{\mathbb{F},\ast}_t=\xi^{\mathbb{F}}+\int_{t}^{T}g^{\mathbb{F}}(s)ds+\int_{t}^{T}\zeta_s dV^{\mathbb{F}}_s+\left( K^{\mathbb{F},+}_T-K^{\mathbb{F},+}_t\right) -\left( K^{\mathbb{F},-}_T-K^{\mathbb{F},-}_t\right) -\int_{t}^{T} \hat{Z}^{\mathbb{F},\ast}_s d B_s,
		$$
		and $\left(\hat{Y}^{\mathbb{F},\ast}_{t},\hat{Z}^{\mathbb{F},\ast}_{t},K^{\mathbb{F},+}_{t},K^{\mathbb{F},-}_{t}\right)_{t \leq T} \in \mathfrak{C}^2_{T}\left( \beta,\mathbb{F},\mathbb{P}\right) \times \mathcal{H}^2_{T}\left( \beta,\mathbb{F},\mathbb{P}\right) \times \mathcal{S}^2_T\left( \mathbb{F},\mathbb{P}\right)\times \mathcal{S}^2_T\left( \mathbb{F},\mathbb{P}\right)$.\\
		Additionally, it is clear that $\hat{Y}^{\mathbb{F},\ast}_t \geq L_t$ $\forall t \in [0,T]$.\\
		We have just seen that the sequence $\left(\hat{Y}^{n,\mathbb{F},\ast},\hat{K}^{n,\mathbb{F},+}\right)$ tends to $(\hat{Y}^{\mathbb{F},\ast}, K^{\mathbb{F},+})$ uniformly in $t$ in probability. Then the measure $dK^{n,\mathbb{F},+}$ tends to $dK^{\mathbb{F},+}$  weakly in probability, hence
		$$
		\int_{0}^{T}\left(\hat{Y}^{n,\mathbb{F},\ast}_s-L^{\mathbb{F}}_s\right)d\hat{K}^{n,\mathbb{F},+}_s \xrightarrow[n \rightarrow+\infty]{\mathbb{P}} \int_{0}^{T}\left(\hat{Y}^{\mathbb{F},\ast}_s-L^{\mathbb{F}}_s\right)dK^{\mathbb{F},+}_s.
		$$
		Thus, from (\ref{basic equation RBSDE in F,P})-(iii) we have $\int_{0}^{T}\left(\hat{Y}^{\mathbb{F},\ast}_s-L^{\mathbb{F}}_s\right)dK^{\mathbb{F},+}_s=0$.\\
		Putting together the argument presented above, we deduce that the limiting process $\left(\hat{Y}^{\mathbb{F},\ast}_{t},\hat{Z}^{\mathbb{F},\ast}_{t},K^{\mathbb{F},+}_{t},K^{\mathbb{F},-}_{t}\right)_{t \leq T} $ satisfies the following GRBSDE
		\begin{equation}
			\left\{
			\begin{split}
				\text{(i)} &~ \mathbb{P}\text{-a.s., for all } t \in [0,T],\\
				& \hat{Y}^{\mathbb{F},\ast}_t=\xi^{\mathbb{F}}+\int_{t}^{T}g^{\mathbb{F}}(s)ds+\int_{t}^{T}\zeta_s dV^{\mathbb{F}}_s+\left( K^{\mathbb{F},+}_T-K^{\mathbb{F},+}_t\right)\\\ &\qquad\qquad-\left(K^{\mathbb{F},-}_T-K^{\mathbb{F},-}_t\right) -\int_{t}^{T} \hat{Z}^{\mathbb{F},\ast}_s d B_s;\\
				\text{(ii)} &~  L^{\mathbb{F}}_t \leq \hat{Y}^{\mathbb{F},\ast}_t,~  \forall t \in [0,T],~\mathbb{P}\text{-a.s.};\\
				\text{(iii)} &~ \text{Skorokhod condition:}  \int_0^{T }(\hat{Y}^{\mathbb{F},\ast}_{t}-L^{\mathbb{F}}_{t})dK^{\mathbb{F},+}_t=0,~\mathbb{P}\text{-a.s.},
			\end{split}
			\right.
			\label{What wa got}
		\end{equation}
		\paragraph*{Step 4: Connection between solutions of DRBSDE in $(\mathbb{G},\mathbb{Q})$ and in $(\mathbb{F},\mathbb{P})$}
		The main result will be obtained by combining the optimal stopping problems satisfied by the sequence $\{\hat{Y}^{n,\mathbb{F},\ast}\}_{n \in \mathbb{N}}$ in $\left(\mathbb{F},\mathbb{P}\right)$ and the one satisfied by $\{Y^n\}_{n \in \mathbb{N}}$ in  $\left(\mathbb{G},\mathbb{Q}\right)$, the relationships between these two problems described by
		(\ref{Stable 2}) in \textbf{Step 2}, and the convergence results of these sequences giving in \textbf{Steps 1} and \textbf{3}. Finally, applying the same idea and procedure to the RBSDE penalization schemes (\ref{basic equation RBSDE}) in $\left(\mathbb{G},\mathbb{Q}\right)$ but this time with respect to the lower barrier $L$ and data $(\xi,f,U)$ allows us to derive the equality for values processes of DRBSDEs on $\llbracket0,T\wedge\tau\llbracket$, and then for other solution component by a simple integration by parts formula.\\
		To begin, using the GRBSDE (\ref{basic equation RBSDE in F,P}) and Corollary 2.9 in \cite{klimsiak2015reflected}, we have
		$$
		\hat{Y}^{n,\mathbb{F},\ast}_t=\esssup_{\theta \in \mathfrak{T}_{t}^{T}(\mathbb{F})}\mathbb{E}\left[\int_{t}^{\theta} dV^{n,\mathbb{F},\ast}_s+L^{\mathbb{F}}_{\theta} \mathds{1}_{\{\theta<T\}}+\xi^{\mathbb{F}}\mathds{1}_{\{\theta=T\}}  \mid \mathcal{F}_t\right]
		$$
		Therefore
		$$
		\hat{Y}^{n,\mathbb{F},\ast}+\int_{0}^{\cdot \wedge T}g^{\mathbb{F}}(s)ds+\int_{0}^{\cdot \wedge T}\zeta_s dV^{\mathbb{F}}_s-n\int_{0}^{\cdot \wedge T} (\hat{Y}^{n,\mathbb{F}}_s-U^{\mathbb{F}}_s)^{+}ds=\mathcal{\hat{S}}^{\mathbb{F},n}.
		$$
		Then, from (\ref{Stable 2}), we drive that:
		\begin{equation}
			\begin{split}
				&Y^n+\int_{0}^{\cdot \wedge T \wedge \tau}g(s)ds-\int_{0}^{\cdot \wedge T \wedge \tau}d K^{n,-}_s\\
				&=\dfrac{\hat{Y}^{n,\mathbb{F},\ast}+V^{n,\mathbb{F},\ast}}{\tilde{\mathcal{E}}^T}\left(\mathds{1}_{\llbracket 0, \tau\llbracket}\right)^T+\int_{0}^{\cdot \wedge T}\left(\dfrac{k^{(pr),n}_s\tilde{\mathcal{E}}_s-\int_{0}^{s}k^{(pr),n}_u d\tilde{\mathcal{E}}_u}{\tilde{\mathcal{E}}_s}\right) dN^{\mathbb{G}}_s.
			\end{split}
			\label{Stable 2,2}
		\end{equation}
		Let us consider the term on the right hand side of the second line of (\ref{Stable 2,2}). Using an integration by part formula, we get
		\begin{equation*}
			k^{(pr),n}_s\tilde{\mathcal{E}}_s-\int_{0}^{s}k^{(pr),n}_u d\tilde{\mathcal{E}}_u= V^{n,\mathbb{F},\ast}_s+\zeta_{ s}\tilde{\mathcal{E}}_s \mathds{1}_{\{s<T\}}+\xi^{\mathbb{F}}\mathds{1}_{\{s \geq T\}}
			\label{MN6}
		\end{equation*}
		Plugging this into (\ref{Stable 2,2}), we get
		\begin{equation}
			\begin{split}
				&Y^n+\int_{0}^{\cdot \wedge T \wedge \tau}g(s)ds-\int_{0}^{t \wedge T \wedge \tau}d K^{n,-}_s\\
				&=\dfrac{\hat{Y}^{n,\mathbb{F},\ast}}{\tilde{\mathcal{E}}^T}\left(\mathds{1}_{\llbracket 0, \tau\llbracket}\right)^T+\dfrac{V^{n,\mathbb{F},\ast}}{\tilde{\mathcal{E}}^T}\left(\mathds{1}_{\llbracket 0, \tau\llbracket}\right)^T+\int_{0}^{\cdot \wedge T}\left(\dfrac{V^{n,\mathbb{F},\ast}_s}{\tilde{\mathcal{E}}_s}\right) dN^{\mathbb{G}}_s\\
				&\qquad+\int_{0}^{\cdot \wedge T}\zeta_{s}\mathds{1}_{\{s<T\}}dN^{\mathbb{G}}_s+\int_{0}^{\cdot \wedge T} \frac{\zeta_{T}\tilde{\mathcal{E}}_T}{\tilde{\mathcal{E}}_s} \mathds{1}_{\{s\geq T\}}dN^{\mathbb{G}}_s.
			\end{split}
			\label{Stable 2,2,1}
		\end{equation}
		Using $\left(\mathds{1}_{\llbracket 0, \tau\llbracket}\right)^T=\mathds{1}_{\{T < \tau\}}+\mathds{1}_{\llbracket 0,T \wedge \tau \llbracket}\mathds{1}_{\{\tau \leq T\}}$ and $\xi \mathds{1}_{\{T < \tau\}}=\zeta_{T \wedge \tau}\mathds{1}_{\{T < \tau\}}=\zeta_{T}\mathds{1}_{\{T < \tau\}}$, we have
		\begin{equation}
			\begin{split}
				\dfrac{\hat{Y}^{n,\mathbb{F},\ast}}{\tilde{\mathcal{E}}^T}\left(\mathds{1}_{\llbracket 0, \tau\llbracket}\right)^T
				=\dfrac{\hat{Y}^{n,\mathbb{F},\ast}}{\tilde{\mathcal{E}}}\mathds{1}_{\llbracket 0,T \wedge \tau \llbracket}+\xi\mathds{1}_{\{T < \tau\}}\mathds{1}_{\llbracket T \wedge \tau, +\infty  \llbracket},
			\end{split}
			\label{MN9}
		\end{equation}
		After doing a few simple calculations, we establish:
		\begin{equation}
			\int_{0}^{\cdot \wedge T}\zeta_{s}\mathds{1}_{\{s<T\}}dN^{\mathbb{G}}_s+\int_{0}^{\cdot \wedge T} \frac{\zeta_{T}\tilde{\mathcal{E}}_T}{\tilde{\mathcal{E}}_s} \mathds{1}_{\{s\geq T\}}dN^{\mathbb{G}}_s=\xi \mathds{1}_{\{\tau \leq T\}}\mathds{1}_{\llbracket T \wedge \tau,+\infty \llbracket}-\int_{0}^{\cdot \wedge T \wedge \tau}\frac{\zeta_{s}}{\tilde{G}_s} dD^{o,\mathbb{F},\mathbb{P}}_s.
			\label{MN7}
		\end{equation}
		Furthermore, by applying Lemma 4 in \cite{choulli2024optimal}, we get:
		\begin{equation}
			\begin{split}
				\int_{0}^{\cdot \wedge T}\frac{ V^{n,\mathbb{F},\ast}_s}{\tilde{\mathcal{E}}_s}dN^{\mathbb{G}}_s
				=\int_{0}^{\cdot \wedge T \wedge \tau}g(s)ds-\int_{0}^{t \wedge T \wedge \tau}d K^{n,-}_s+\int_{0}^{\cdot \wedge T \wedge \tau}\frac{\zeta_{ s}}{\tilde{G}_s} dD^{o,\mathbb{F},\mathbb{P}}.
			\end{split}
			\label{MN8}
		\end{equation}
		Plugging (\ref{MN9}), (\ref{MN7}) (\ref{MN8}) in (\ref{Stable 2,2,1}), we derive 
		\begin{equation*}
			\begin{split}
				&Y^n+\int_{0}^{\cdot \wedge T \wedge \tau}g(s)ds-\int_{0}^{t \wedge T \wedge \tau}d K^{n,-}_s\\
				&=\dfrac{\hat{Y}^{n,\mathbb{F},\ast}}{\tilde{\mathcal{E}}}\mathds{1}_{\llbracket 0,T \wedge \tau \llbracket}+\int_{0}^{\cdot \wedge T \wedge \tau}g(s)ds-\int_{0}^{t \wedge T \wedge \tau}d K^{n,-}_s+\xi\mathds{1}_{\llbracket T \wedge \tau, +\infty  \llbracket}.
			\end{split}
		\end{equation*}   
		Thus
		\begin{equation}
			\begin{split}
				Y^n=\dfrac{\hat{Y}^{n,\mathbb{F},\ast}}{\tilde{\mathcal{E}}}\mathds{1}_{\llbracket 0,T \wedge \tau \llbracket}+\xi\mathds{1}_{\llbracket T \wedge \tau, +\infty  \llbracket}.
			\end{split}
			\label{Stable 2,2,1,1}
		\end{equation}
		Taking the limit as $n \rightarrow+\infty$ in (\ref{Stable 2,2,1,1}), and using the above convergence results for the sequences $\{Y^n\}_{n \in \mathbb{N}}$ and $\{\hat{Y}^{n,\mathbb{F},\ast}\}_{n \in \mathbb{N}}$ along an almost sure sequence, we deduce, by using the results from \textbf{Steps 1-3}, that 
		\begin{equation}
			Y=\dfrac{\hat{Y}^{\mathbb{F},\ast}}{\tilde{\mathcal{E}}}\mathds{1}_{\llbracket 0,T \wedge \tau \llbracket}+\xi\mathds{1}_{\llbracket T \wedge \tau, +\infty  \llbracket}.
			\label{MN10}
		\end{equation}
		Following a similar procedure, we can prove the existence of a unique quadruplet of processes $\left(\bar{Y}^{\mathbb{F},\ast}_t,\bar{Z}^{\mathbb{F},\ast}_t,\bar{K}^{\mathbb{F},+}_t,\bar{K}^{\mathbb{F},-}_t\right)_{t \leq T}$ satisfying the following GRBSDE:
		\begin{equation}
			\left\{
			\begin{split}
				\text{(i)} &~ \mathbb{P}\text{-a.s., for all } t \in [0,T],\\
				&~  \bar{Y}^{\mathbb{F},\ast}_t=\xi^{\mathbb{F}}+\int_{t}^{T}g^{\mathbb{F}}(s)ds+\int_{t}^{T}\zeta_s dV^{\mathbb{F}}_s+\left( \bar{K}^{\mathbb{F},+}_T-\bar{K}^{\mathbb{F},+}_t\right)\\ &\qquad\qquad-\left(\bar{K}^{\mathbb{F},-}_T-\bar{K}^{\mathbb{F},-}_t\right) -\int_{t}^{T} \bar{Z}^{\mathbb{F},\ast}_s d B_s;\\
				\text{(ii)} &~  \bar{Y}^{\mathbb{F},\ast}_t \leq U^{\mathbb{F}}_t,~  \forall t \in [0,T],~\mathbb{P}\text{-a.s.};\\
				\text{(iii)} &~ \text{Skorokhod condition:}  \int_0^{T }(U^{\mathbb{F}}_{t}-\bar{Y}^{\mathbb{F},\ast}_{t})d\bar{K}^{\mathbb{F},-}_t=0,~\mathbb{P}\text{-a.s.},
			\end{split}
			\right.
			\label{What wa got again}
		\end{equation}
		such that
		\begin{equation}
			Y=\dfrac{\bar{Y}^{\mathbb{F},\ast}}{\tilde{\mathcal{E}}}\mathds{1}_{\llbracket 0,T \wedge \tau \llbracket}+\xi\mathds{1}_{\llbracket T \wedge \tau, +\infty  \llbracket}.
			\label{MN11}
		\end{equation}
		From BSDEs (\ref{What wa got}), (\ref{What wa got again}), equalities (\ref{MN10}), (\ref{MN11}) and the uniqueness of the solution, we may derive 
		\begin{equation*}
			Y=\dfrac{Y^{\mathbb{F}}}{\tilde{\mathcal{E}}}\mathds{1}_{\llbracket 0,T \wedge \tau \llbracket}+\xi\mathds{1}_{\llbracket T \wedge \tau, +\infty  \llbracket},
		\end{equation*}
		where $Y^{\mathbb{F}}$ is the state process of the DRBSDE (\ref{basic equation comparison}).\\
		Now, to completes the proof, it suffices to apply an integration by parts formula to the process $\left(\frac{Y^{\mathbb{F}}_t}{\tilde{\mathcal{E}}_t} \mathds{1}_{\{t<T \wedge\tau \}}+\xi \mathds{1}_{\{ t \geq T \wedge \tau \}}\right)_{t\geq 0}$, which yields to the following dynamic:
		\begin{equation*}
			\begin{split}
				-dY_t&=-d\left(\frac{Y^{\mathbb{F}}_t}{\tilde{\mathcal{E}}_t} \mathds{1}_{\{t<T \wedge\tau \}}+\xi \mathds{1}_{\{ t \geq T \wedge \tau \}}\right)\\
				&=\mathds{1}_{\{t \leq T\wedge \tau\}} g(t)dt+\dfrac{\mathds{1}_{\{t \leq T\wedge \tau\}}}{\tilde{\mathcal{E}}_{t-}} dK^{\mathbb{F},+}_{t }-\dfrac{\mathds{1}_{\{t \leq T\wedge \tau\}}}{\tilde{\mathcal{E}}_{t-}}dK^{\mathbb{F},-}_{t } +\mathds{1}_{\{t \leq T\wedge \tau\}}\dfrac{Z^{\mathbb{F}}_t}{\tilde{\mathcal{E}}_{t-}}dB_t,
			\end{split}
		\end{equation*}
		and $Y_t=\xi$  on the set $\{t \geq T \wedge \tau\} $ by construction. Then the proof is complete, using the DRBSDE (\ref{basic equation}) and the uniqueness of the solution stated in Proposition \ref{basic Estimastion for BSDE}.
	\end{proof}

	\begin{remark}
		A direct reduction from the $(\mathbb{G},\mathbb{Q})$-framework to the
		$(\mathbb{F},\mathbb{P})$-framework may also be considered, in particular
		at the level of the associated stopping problems, by using the reduction
		of $\mathbb{G}$-stopping times before $\tau$ to $\mathbb{F}$-stopping
		times.
		
		In the present work, however, we follow the penalization approach because
		our objective is not only to identify the value process $Y$, but also to
		establish the relationship between the complete DRBSDE solutions,
		including the reflection processes associated with the lower and upper
		barriers. The penalization procedure allows us to treat the two reflection
		mechanisms separately through one-barrier reflected BSDEs and then to pass
		to the limit while preserving the corresponding Skorokhod conditions.
	\end{remark}

	\subsubsection{Second link} 
	Here, for each  we will observe that the behavior of the state process of the DRBSDE (\ref{basic equation}) with a generator $f$ independent of the solution, i.e., $f(t,y,z)=g(t)$ for all $(t,y,z) \in [0,T] \times \mathbb{R}^{1+d}$, under the information provided by $\mathbb{F}$ under $\mathbb{Q}$, resembles that of a solution to a special form of DRBSDE in the $(\mathbb{F},\mathbb{Q})$-setup with a linear generator in the $z$-variable. To put it simply, we're interested in understanding the relationship between the $\mathbb{F}$ conditional expectation of the state process of the DRBSDE (\ref{basic equation}) under $\mathbb{Q}$ and a solution of another special DRBSDE in $\mathbb{F}$ with linear generator in the $z$-variable. Additionally, we generalize this to include a linear generator depending on the $y$-variable. 
	
	In this subsection, we impose the additional assumption that
	$\tau$ and $\mathcal{F}_T$ are independent under $\mathbb{Q}$.
	This assumption is not a consequence of the construction of the
	probability measure $\mathbb{Q}$ introduced in Section \ref{Notations} and should
	be regarded as an additional structural condition specific to the
	results of the present subsection.  Consequently, the Second Link should be understood as
	a conditional result valid within this more restrictive framework.
	
	\begin{lemma}
		Let $(Z_t)_{t\geq 0}$ be a $\mathbb{G}$-predictable process such that
		$Z\in\mathcal{H}^2_{T\wedge\tau}(\beta,\mathbb{G},\mathbb{Q})$.
		Assume, in addition, that $\tau$ and $\mathcal{F}_T$ are independent
		under $\mathbb{Q}$. Then, for every $t\in[0,T]$,
		\[
		\mathbb{E}_{\mathbb{Q}}\left[
		\int_0^t Z_s\,dB_s
		\,\middle|\,
		\mathcal{F}_t
		\right]
		=
		\int_0^t
		\mathbb{E}_{\mathbb{Q}}\left[
		Z_s
		\,\middle|\,
		\mathcal{F}_s
		\right]dB_s.
		\]
		\label{result of interchanging the conditional expectation}
	\end{lemma}
	%
	%
	The proof of this lemma can be found in Appendix \ref{Appendix B}.

	\begin{remark}
		Since $Z=Z \mathds{1}_{\rrbracket 0,T \wedge \tau \rrbracket}$ for $Z \in \mathcal{H}^2_{T \wedge \tau}\left( \beta,\mathbb{G},\mathbb{Q}\right)$, we can  write
		$$
		\mathbb{E}_{\mathbb{Q}}\left[\int_{0}^{t \wedge \tau} Z_s dB_s \mid \mathcal{F}_t \right]=\int_{0}^{t} \mathbb{E}_{\mathbb{Q}}\left[ Z_s \mathds{1}_{\left\lbrace s < \tau\right\rbrace } \mid \mathcal{F}_s \right]  dB_s=\int_{0}^{t} \mathbb{E}_{\mathbb{Q}}\left[ Z_s  \mid \mathcal{F}_s \right]  dB_s ,~ t \in [0,T].
		$$ 
	\end{remark}
	%
	
	\begin{remark}
		Let $\Psi$ denote the $(\mathbb{G},\mathbb{P})$-density process associated with
		the probability measure $\mathbb{Q}$, defined by
		\begin{equation*}
			d\mathbb{Q}=\Psi_T\,d\mathbb{P}
			\quad \text{on } (\Omega,\mathcal{G}_T).
		\end{equation*}
		
		We define the $\mathbb{F}$-adapted process
		$\rho=(\rho_t)_{t\leq T}$ by
		\begin{equation}
			\rho_t
			:=
			\mathbb{E}\left[
			\Psi_T \mid \mathcal{F}_t
			\right],
			\qquad t\in[0,T].
			\label{density-rho}
		\end{equation}
		
		Then $\rho$ is a strictly positive uniformly integrable
		$(\mathbb{F},\mathbb{P})$-martingale and is the density process of the
		restriction of $\mathbb{Q}$ to $\mathbb{F}$.
		
		Since $B$ has the martingale representation property in
		$(\mathbb{F},\mathbb{P})$, there exists an $\mathbb{F}$-predictable process
		$\theta^{\ast}=(\theta^{\ast}_t)_{t\leq T}$ such that
		\begin{equation*}
			\int_0^T |\theta^{\ast}_s|^2\,ds<+\infty,
			\qquad \mathbb{P}\text{-a.s.},
		\end{equation*}
		and
		\begin{equation}
			\rho_t
			=
			\rho_0
			+
			\int_0^t \theta^{\ast}_s\,dB_s,
			\qquad t\in[0,T].
			\label{RP}
		\end{equation}
		
		Consequently, the process
		\begin{equation*}
			B^{[\rho]}_t
			:=
			B_t
			-
			\int_0^t
			\frac{1}{\rho_{s-}}\,
			d\langle \rho,B\rangle_s,
			\qquad t\in[0,T],
		\end{equation*}
		is an $(\mathbb{F},\mathbb{Q})$-Brownian motion possessing the martingale
		representation property in $(\mathbb{F},\mathbb{Q})$.
	\end{remark}

	The purpose of the following result is to describe the dynamics of the
	$\mathbb{G}$-state process when only the information carried by the
	reference filtration $\mathbb{F}$ is retained. More precisely, by
	projecting the solution of the DRBSDE onto $\mathbb{F}$ under
	$\mathbb{Q}$, the result identifies explicitly the different effects
	generated by the passage from the enlarged filtration $\mathbb{G}$ to
	the reference filtration $\mathbb{F}$. These effects appear through the
	survival-weighted driver
	$
	\hat{g}^{\mathbb{F}}(t)
	=
	\mathbb{Q}(t\leq\tau)\,g(t),
	$
	the Girsanov correction
	$
	\Theta^{\rho,\ast}_t\hat{Z}^{\mathbb{F}}_t,
	$
	and the additional martingale representation term
	$\Theta^{\mathbb{F}}$. Hence, the following theorem provides an
	$\mathbb{F}$-adapted characterization of the
	$(\mathbb{G},\mathbb{Q})$-solution and makes explicit how the default
	time affects the dynamics observed in the smaller filtration.
	
	\begin{theorem}
		Let  $\Theta^{\rho,\ast}:=(\Theta^{\rho,\ast}_t)_{t \leq T}$ be the positive $\mathbb{F}$-predictable process defined by $\Theta^{\rho,\ast}_t:=\dfrac{\theta^{\ast}_t}{\rho_{t}}$. 
		Let $(Y_t,Z_t,K^{+}_t,K^{-}_t,M_t)_{t \geq 0}$ be the unique solution of the DRBSDE (\ref{basic equation}) associated with $(\xi,g,L,U)$ satisfying \textbf{(H)}.\\
		Assume that $\tau$ and $\mathcal{F}_T$ are independent under $\mathbb{Q}$, that $\xi \in \mathcal{L}^2_{T}(\beta,\mathbb{G},\mathbb{Q})$, that $
		\Theta^{\rho,\ast} \in \mathcal{H}^2_{T}(\mathbb{F},\mathbb{Q})$, and that $(L_t)_{t \leq T}$, $(U_t)_{t \leq T} \in \mathcal{D}^2_{T}\left( \mathbb{F},\mathbb{Q}\right)$. \\
		Then, there exists a unique $\mathbb{F}$-predictable process $(\Theta^{\mathbb{F}}_t)_{t \leq T} \in \mathcal{H}^2_T(\mathbb{F},\mathbb{Q})$ and a unique $\mathbb{F}$-predictable continuous increasing processes $\left(K^{\mathbb{F},+},K^{\mathbb{F},-}\right):=(K^{\mathbb{F},+}_t,K^{\mathbb{F},-}_t)_{t \leq T}$ such that the $\mathbb{F}$-adapted process  $\left(\hat{Y}^{\mathbb{F}},\hat{Z}^{\mathbb{F}},K^{\mathbb{F},+},K^{\mathbb{F},-},\Theta^{\mathbb{F}}\right):=\left(\mathbb{E}_{\mathbb{Q}}\left[Y_t  \mid \mathcal{F}_t\right],\mathbb{E}_{\mathbb{Q}}\left[Z_t \mid \mathcal{F}_t\right],K^{\mathbb{F},+}_t,K^{\mathbb{F},-}_t,\Theta^{\mathbb{F}}_t\right)_{t \leq T}$ satisfies
		\begin{equation*}
			\left\{
			\begin{split}
				&\text{(i)}~\hat{Y}^{\mathbb{F}}_t= 	\hat{\xi}^{\mathbb{F}}+\int_t^T \left\lbrace \hat{g}^{\mathbb{F}}(s)+\Theta^{\rho,\ast}_s \hat{Z}^{\mathbb{F}}_s \right\rbrace  ds+\left( K^{\mathbb{F},+}_T-K^{\mathbb{F},+}_t\right)-\left( K^{\mathbb{F},-}_T-K^{\mathbb{F},-}_t\right)\\
				&\qquad\qquad\qquad-\int_t^T \hat{Z}^{\mathbb{F}}_s  d B^{[\rho]}_s-\int_{t}^{T} \Theta^{\mathbb{F}}_s  d B^{[\rho]}_s,\quad \text{for all }t \in [0,T],~\mathbb{Q}\text{-a.s.}\\
				&\text{(ii)}~L_t \leq \hat{Y}^{\mathbb{F}}_t \leq U_t,~ \forall t \in [0,T],  \mathbb{Q}\text{-a.s.} \\
				&\text{(iii)}~\int_{0}^{T}(\hat{Y}^{\mathbb{F}}_s-L_s)dK^{\mathbb{F},+}_s=\int_{0}^{T}(U_s-\hat{Y}^{\mathbb{F}}_s)dK^{\mathbb{F},-}_s =0,~ \mathbb{Q}\text{-a.s.}
			\end{split}
			\right.
		\end{equation*}
		With $\hat{\xi}^{\mathbb{F}}=\mathbb{E}_{\mathbb{Q}}\left[\xi\mid \mathcal{F}_T\right]$, $\hat{g}^{\mathbb{F}}(t)=\mathbb{Q}\left(t \leq \tau \mid \mathcal{F}_t\right) g(t)= \mathbb{Q}\left(t \leq \tau\right) g(t)$ and $\mathbb{E}_{\mathbb{Q}}\left[M_t \mid \mathcal{F}_t\right]=\mathbb{E}_{\mathbb{Q}}\left[M_{T}  \right]+\int_{0}^{t} \Theta^{\mathbb{F}}_s dB^{[\rho]}_s$, for $t \in [0,T]$. Moreover, we have  $\left(\hat{Y}^{\mathbb{F}},\hat{Z}^{\mathbb{F}},K^{\mathbb{F},+},K^{\mathbb{F},-},\Theta^{\mathbb{F}}\right)\in \mathfrak{B}^2_{T}\left( \beta,\mathbb{F},\mathbb{Q}\right)$.
		\label{Proposition of comparison}
	\end{theorem}

	\begin{remark}
		The independence assumption between $\tau$ and $\mathcal{F}_T$
		under $\mathbb{Q}$ is satisfied in a simple non-degenerate setting.
		Indeed, suppose that, under $\mathbb{P}$, the random time $\tau$ is
		independent of the Brownian filtration $\mathbb{F}$\footnote{For a similar situation, we refer readers to \cite[Proposition 2.3]{jeanblanc2009progressive} and Chapter 2 and Exercise 2.4 in \cite{aksamit2017enlargement}.} and has an
		exponential distribution with parameter $\lambda>0$. Then
		\[
		G_t
		=
		\mathbb{P}(\tau>t\mid\mathcal{F}_t)
		=
		e^{-\lambda t},
		\qquad t\geq0,
		\]
		and
		\[
		D^{o,\mathbb{F},\mathbb{P}}_t
		=
		1-e^{-\lambda t}.
		\]
		Consequently,
		\[
		m_t
		=
		G_t+D^{o,\mathbb{F},\mathbb{P}}_t
		=
		1,
		\qquad t\geq0.
		\]
		Hence,
		\[
		\Gamma^G_t
		=
		\int_0^t G^{-1}_{s-}\,dm_s
		=
		0,
		\]
		so that
		\[
		\mathcal{E}_t=1
		\qquad\text{and therefore}\qquad
		\Psi_t=1.
		\]
		It follows from the definition of $\mathbb{Q}$ that
		\[
		\mathbb{Q}=\mathbb{P}
		\quad\text{on }\mathcal{G}_T.
		\]
		Thus, since $\tau$ and $\mathcal{F}_T$ are independent under
		$\mathbb{P}$, they are also independent under $\mathbb{Q}$.
		
		This provides a non-degenerate example covered by the assumptions of
		Theorem~\ref{Proposition of comparison}, since $\tau$ remains a genuine
		random default time with
		\[
		\mathbb{Q}(\tau>t)=e^{-\lambda t}>0.
		\]
	\end{remark}

	Before prove Theorem \ref{Proposition of comparison}, we note that since $\tau$ and $\mathcal{F}_T$ are independent under $\mathbb{Q}$ by assumption,
	and $\mathcal{F}_t\subseteq\mathcal{F}_T$ for every $t\in[0,T]$, we have
	\begin{equation*}
		\mathbb{Q}\left(t\leq\tau\mid\mathcal{F}_t\right)
		=
		\mathbb{Q}\left(t\leq\tau\right),
		\qquad t\in[0,T].
		\label{independence-survival-Q}
	\end{equation*}
	Accordingly, we then define
	\begin{equation*}
		\hat{g}^{\mathbb{F}}(t)
		:=
		\mathbb{Q}\left(t\leq\tau\right)g(t),
		\qquad t\in[0,T].
		\label{definition-g-hat-F}
	\end{equation*}
	
	Additionally, 
	the independence assumption between $\tau$ and $\mathcal{F}_T$
	under $\mathbb{Q}$ is used in two places. First, it is used in
	Lemma~\ref{result of interchanging the conditional expectation} to
	identify the $\mathbb{F}$-conditional projection of the stochastic
	integral. Second, since $\mathcal{F}_t\subseteq\mathcal{F}_T$, it yields
	\[
	\mathbb{Q}\left(t\leq\tau\mid\mathcal{F}_t\right)
	=
	\mathbb{Q}\left(t\leq\tau\right),
	\qquad t\in[0,T],
	\]
	which is used when projecting the driver
	$g(t)\mathds{1}_{\{t\leq\tau\}}$ onto $\mathcal{F}_t$.
	
	\begin{proof}[Proof of Theorem \ref{Proposition of comparison}] 
		In this proof, $\mathfrak{c}>0$ indicates a constant that remains independent of the index of any stochastic process sequences and may vary from line to line. Additionally, note that every $\mathbb{F}$-optional process is $\mathbb{F}$-predictable due to the Brownian filtration $\mathbb{F}$ \cite[Corollary V.3.3]{revuz1999continuous}.\\	
		The proof is performed in eight parts.
		\paragraph*{Part 1: Penalized equations}
		\emph{}\\
		To begin, let us consider the following initial penalized BSDE in the defaultable setup $(\Omega,\mathcal{F},\mathbb{G},\mathbb{Q})$:
		\begin{equation}
			\left\{
			\begin{split}
				&\text{(i)}~Y^{n}_t= \xi+\int_0^t g(s)\mathds{1}_{\{s \leq \tau\}}ds+n\int_0^t (L_s-Y^{n}_s)^{+}\mathds{1}_{\{s \leq \tau\}}ds  \\
				&\qquad\quad -n \int_0^t (Y^{n}_s-U_s)^{+}\mathds{1}_{\{s \leq \tau\}}ds -\int_0^{t} Z^{n}_s d B_s-\int_0^{t}  dM^{n}_s, \quad t \in [0,T],~\mathbb{Q}\text{-a.s.}\\
				&\text{(ii)}~Y^n_t=\xi,~Z_t^n=dM^n_t=0 \text{ on the set } \{t \geq T \wedge \tau\},~\mathbb{Q}\text{-a.s.}
			\end{split}
			\right.
			\label{penalization equations with respect to the two reflection barriers}
		\end{equation}
		Rewriting (\ref{penalization equations with respect to the two reflection barriers}) forwards, using Remark \ref{strssfull remark}-(a), Lemma 3.9 in \cite{BDEM}, and taking the conditional expectation on both sides of (\ref{penalization equations with respect to the two reflection barriers}) together with Fubini's Theorem, we obtain, for all $t \in [0,T]$, 
		\begin{equation}
			\begin{split}
				\hat{Y}^n_t
				&=	Y_0-\int_0^t g(s)\mathbb{Q}\left(s \leq \tau \mid \mathcal{F}_s\right)ds-	K^{n,\mathbb{F},+}_t+	K^{n,\mathbb{F},-}_t \\ 
				&\qquad+\int_0^t \mathbb{E}_{\mathbb{Q}}\left[Z^{n}_s \mid \mathcal{F}_s \right] d B_s+\mathbb{E}_{\mathbb{Q}}\left[M^{n}_t \mid \mathcal{F}_t \right]. 
			\end{split}
			\label{WE just trying}
		\end{equation}
		with $\hat{Y}^{n, \mathbb{F}}_t=\mathbb{E}_{\mathbb{Q}}\left[Y^n_t \mid \mathcal{F}_t\right]$, $K^{n,\mathbb{F},+}_t=n\int_0^t \mathbb{E}_{\mathbb{Q}} \left[ (L_s-Y^{n}_s)^{+}\mathds{1}_{\{s \leq \tau\}} \mid \mathcal{F}_s \right] ds$, and $	K^{n,\mathbb{F},-}_t:=n\int_0^t\mathbb{E}_{\mathbb{Q}}\left[  (Y^{n}_s-U_s)^{+}\mathds{1}_{\{s \leq \tau\}} \mid \mathcal{F}_s \right]ds$ for $t \in [0,T]$. Additionally, from Remark \ref{RMQ of Well defined Steiltjes integrals}, we derive that the processes $\left(K^{n,\mathbb{F},+}_t\right)_{t \leq T}$ and $\left( K^{n,\mathbb{F},+}_t\right)_{t \leq T}$ are well-defined, finite-valued Stieltjes integrals, which are increasing $\mathbb{F}$-adapted and continuous.	  	
		\paragraph*{Part 2: Decomposition of the processes $(Z^n,M^n)$ in $\mathbb{F}$ with respect to $\mathbb{Q}$}
		\emph{}\\ 
		First, note that, $\left\langle \rho,B \right\rangle_t=[\rho,B]^{p,\mathbb{F},\mathbb{Q}}_t=\int_{0}^{t}\frac{1}{\rho_{s-}}d\left(\int_{0}^{s}\rho_u d[\rho,B]_u \right)^{p,\mathbb{F},\mathbb{P}} $ (see Section I-3 in \cite{JeulinDecom}). Then, from the fact that $d[\rho,B]_u=\theta^{\ast}_u du$, $\mathbb{P}$-a.s., where $\theta^{\ast}$ is defined in (\ref{RP}), we derive $\left\langle \rho,B \right\rangle_t=\int_{0}^{t} \theta^{\ast}_s ds$, $\mathbb{Q}$-a.s. Then, from Lemma \ref{lemma of new Brownian in F}, we deduce that
		$$
		B^{[\rho]}_t=B_t-\int_{0}^{t} \Theta^{\rho,\ast}_sds,\quad t \in  [0,T].
		$$
		is an $(\mathbb{F},\mathbb{Q})$-Brownian motion having the representation property. \\
		From Theorem \ref{Existence and uniquess theorem BSDE} and the B-D-G inequality, we know that the orthogonal martingale $M^n$ satisfies $M^n_{\tau}=M^n_T$ (from (\ref{penalization equations with respect to the two reflection barriers}-(ii))) and the following  integrability property:
		\begin{equation}
			\mathbb{E}_{\mathbb{Q}}\left[\sup_{0 \leq t \leq T \wedge \tau} e^{\beta \mathcal{A}_t}\left| M^n_t\right|^2 \right]<+\infty.
			\label{square integrability}
		\end{equation}
		On the other hand, using the martingale representation property for the $\mathbb{F}$-Brownian motion $B^{[\rho]}$ under $\mathbb{Q}$ (see Section 3 in \cite[Chapter IV]{bookProtter} for more details), we get: For any $t \in [0,T]$
		\begin{itemize}
			\item On the set $\{t<T \wedge \tau\}$, we have
			$$
			\mathbb{E}_{\mathbb{Q}}\left[M^n_t \mid \mathcal{F}_t\right]=\mathbb{E}_{\mathbb{Q}}\left[M^n_T \right]+\int_{0}^{t} \Theta^{n,1}_s dB^{[\rho]}_s.
			$$
			\item On the set $\{t \geq T \wedge \tau\}$, we get
			$$
			\mathbb{E}_{\mathbb{Q}}\left[M^n_t \mid \mathcal{F}_t\right]=\mathbb{E}_{\mathbb{Q}}\left[M^n_{\tau} \mid \mathcal{F}_t\right]=\mathbb{E}_{\mathbb{Q}}\left[M^n_{\tau} \right]+\int_{0}^{t} \Theta^{n,2}_s dB^{[\rho]}_s,\quad t \geq 0.
			$$
		\end{itemize}
		Further, from (\ref{square integrability}), we deduce that the  two $\mathbb{F}$-predictable processes $(\Theta^{n,1})_{t \geq 0}$ and $(\Theta^{n,2})_{t \geq 0}$ are uniquely defined and belongs to $\mathcal{H}^2_{T}(\mathbb{F},\mathbb{Q})$. Using the fact that  $\mathbb{E}_{\mathbb{Q}}\left[M^n_{\tau} \right]=\mathbb{E}_{\mathbb{Q}}\left[M^n_{T} \right]$, and Itô's Isometric formula, we obtain
		$$
		\mathbb{E}_{\mathbb{Q}}\left[\int_{0}^{t}\left|\Theta^{n,1}_s-\Theta^{n,2}_s \right|^2ds   \right]=0,\quad \forall t \in [0,T].
		$$
		Henceforth, $\Theta^{n,1}=\Theta^{n,2}$, and then the process $\left(\mathbb{E}_{\mathbb{Q}}\left[M^n_t \mid \mathcal{F}_t\right]\right)_{t \leq T}$ can be uniquely represented as:
		\begin{equation}
			\mathbb{E}_{\mathbb{Q}}\left[ M^n_t \mid \mathcal{F}_t\right]=\mathbb{E}_{\mathbb{Q}}\left[M^n_{\tau}  \right]+\int_{0}^{t} \Theta^{n,\mathbb{F}}_s dB^{[\rho]}_s,\quad t \in [0,T],
			\label{some explication}
		\end{equation}
		where $(\Theta^{n,\mathbb{F}}_t)_{t \leq T}$ is a uniquely defined $\mathbb{F}$-predictable process that belongs to $\mathcal{H}^2_{T}(\mathbb{F},\mathbb{Q})$.\\	
		Now, let's address the $\mathbb{G}$-predictable process $Z^n$. Using Lemma 4.4 in \cite{book:63874} for $\mathbb{G}$-predictable process $Z^n$, and considering the fact that $Z^n=Z^n \mathds{1}_{\llbracket0,\tau\rrbracket}$, we deduce the existence of a unique $\mathbb{F}$-predictable process $Z^{n,\mathbb{F}}$ such that $ Z^{n,\mathbb{F}} \in \mathcal{H}^2_{T \wedge \tau}(\mathbb{F},\mathbb{Q})$, and
		\begin{equation}
			Z^{n}_s=Z^{n,\mathbb{F}}_s \mathds{1}_{\{s \leq \tau\}} \text{ thus }  \mathbb{E}_{\mathbb{Q}}\left[Z^{n}_s \mid \mathcal{F}_s \right]=Z^{n,\mathbb{F}}_s \mathbb{E}_{\mathbb{Q}}\left[\mathds{1}_{\{s \leq \tau\}} \mid \mathcal{F}_s \right],\quad s \in [0,T].
			\label{SML}
		\end{equation}
		Define $\hat{Z}^{n,\mathbb{F}}_s:=\mathbb{Q}\left(s \leq \tau \mid \mathcal{F}_s\right) Z^{n,\mathbb{F}}_s$. Then, we have $ \hat{Z}^{n,\mathbb{F}} \in  \mathcal{H}^2_{T}(\mathbb{F},\mathbb{Q})$, and the BSDE (\ref{WE just trying}) becomes
		\begin{equation}
			\begin{split}
				&\hat{Y}^{n,\mathbb{F}}_t=
				\hat{Y}^{n,\mathbb{F}}_0-\int_0^t\left( \hat{g}(s)+\hat{Z}^{n,\mathbb{F}}_s \Theta^{\rho,\ast}_s\right) ds-K^{n,\mathbb{F},+}_t +K^{n,\mathbb{F},-}_t \\
				&\qquad\qquad\quad +\int_0^t \hat{Z}^{n,\mathbb{F}}_s d B^{[\rho]}_s+\int_0^t \Theta^{n,\mathbb{F}}_s d B^{[\rho]}_s, \quad t \in [0,T],~\mathbb{Q} \text{-a.s.}
			\end{split}
			\label{New penalization approximation}
		\end{equation}	
		\paragraph*{Part 3: Estimation of the processes $\left(\hat{Y}^{n,\mathbb{F}}_t,\hat{Z}^{n,\mathbb{F}}_t,K^{n,\mathbb{F},+}_t,K^{n,\mathbb{F},-}_t,\Theta^{n,\mathbb{F}}_t\right)_{t \leq T}$}	
		\emph{}\\
		Set $K^{n,+}_t:=n\int_0^t (L_s-Y^{n}_s)^{+}\mathds{1}_{\{s \leq \tau\}}ds$ and  $K^{n,-}_t:=n\int_0^t (Y^{n}_s-U_s)^{+}\mathds{1}_{\{s \leq \tau\}}ds$, for $t \in [0,T]$. 
		From \cite[Lemma 4.1]{BDEM} concerning the BSDE (\ref{penalization equations with respect to the two reflection barriers}), we have: For any $\beta >0$, there exists a constant $\mathfrak{c}$ such that
		\begin{equation}
			\begin{split}
				\sup_{n \in \mathbb{N}}&\left\lbrace \mathbb{E}_{\mathbb{Q}}\sup_{0 \leq t \leq T \wedge \tau}e^{\beta A_t}\left| Y^n_t\right|^2 + \mathbb{E}_{\mathbb{Q}}\int_{0}^{T \wedge \tau}e^{\beta A_s} \left| Y^n_s \alpha_s\right|^2 ds \right.\\
				&\left. \qquad+	\mathbb{E}_{\mathbb{Q}}\left[ \left| K^{n,+}_{T \wedge \tau} \right|^2\right]+
				\mathbb{E}_{\mathbb{Q}}\left[ \left| K^{n,-}_{T \wedge \tau} \right|^2\right]  \right\rbrace  \leq \mathfrak{c}.
			\end{split}
			\label{un Y}
		\end{equation}
		Using Jensen, Doob's Meyer quadratic inequality \cite[Theorem I.1.43]{bookJacod}, we obtain
		\begin{equation*}
			\begin{split}
				\mathbb{E}_{\mathbb{Q}}\left[ \sup_{0 \leq t \leq T}e^{\beta A_t}\left| \hat{Y}^{n,\mathbb{F}}_t\right|^2\right] &\leq  	\mathbb{E}\sup_{0 \leq t \leq T}\mathbb{E}_{\mathbb{Q}}\left[e^{\beta A_t}\left| Y^n_t\right|^2 \mid \mathcal{F}_t\right] \\
				&\leq 	\mathbb{E}\sup_{0 \leq t \leq T}\mathbb{E}_{\mathbb{Q}}\left[\sup_{0 \leq t \leq T} e^{\beta A_t}\left| Y^n_t\right|^2 \mid \mathcal{F}_t\right]\\
				&\leq 4 \mathbb{E}_{\mathbb{Q}}\left[ \sup_{0 \leq t \leq T}e^{\beta A_t}\left| Y^n_t\right|^2\right] \\
				&\leq 4 \mathbb{E}_{\mathbb{Q}}\left[ \sup_{0 \leq t \leq T \wedge \tau}e^{\beta A_t}\left| Y^n_t\right|^2\right] +4\mathbb{E}_{\mathbb{Q}}\left[ e^{\beta A_T}\left| \xi\right|^2\right] .
			\end{split}
			\label{n tends to infty}
		\end{equation*}
		On the other hand, the positive Fubini-Tonelli's theorem, Jensen's inequality allows us to conclude that
		\begin{equation*}
			\begin{split}
				\mathbb{E}_{\mathbb{Q}}\left[ \int_{0}^{T}e^{\beta A_s} \left| \hat{Y}^{n,\mathbb{F}}_s \alpha_s\right|^2 ds\right] &\leq \int_{0}^{T} \mathbb{E}_{\mathbb{Q}}\left[e^{\beta A_s} \alpha_s^2 \mathbb{E}_{\mathbb{Q}}\left[ \left| Y^n_s\right|^2 \mid \mathcal{F}_s  \right] \right]ds
				=\mathbb{E}_{\mathbb{Q}}\left[ \int_{0}^{T}e^{\beta A_s} \left| Y^n_s \alpha_s\right|^2 ds\right] \\
				& \leq 	\mathbb{E}_{\mathbb{Q}}\left[ \int_{0}^{T \wedge \tau }e^{\beta A_s} \left| Y^n_s \alpha_s\right|^2 ds\right] +\mathbb{E}_{\mathbb{Q}}\left[ e^{\beta A_T}\left| \xi\right|^2\right] .
			\end{split}
		\end{equation*}	
		From the uniform estimation satisfied by the sequence $\{K^{n,\pm}\}_{n \in \mathbb{N}}$, given by (\ref{un Y}) and Jensen inequality, we deduce that 
		\begin{equation}
			\begin{split}
				\sup_{n \in \mathbb{N}}\left\lbrace	\mathbb{E}_{\mathbb{Q}}\left[ \left| K^{n,\mathbb{F},+}_{T} \right|^2\right]+
				\mathbb{E}_{\mathbb{Q}}\left[ \left| K^{n,\mathbb{F},-}_{T } \right|^2\right] \right\rbrace \leq \mathfrak{c}.
				\label{KnF+-}
			\end{split}
		\end{equation}
		Thanks to Lemma 4.1 in \cite{BDEM}, for any $\beta >0$, we can write
		\begin{equation}
			\begin{split}
				\sup_{n \in \mathbb{N}}\left\lbrace\mathbb{E}_{\mathbb{Q}}\left[\int_{0}^{T \wedge \tau} e^{\beta A_s}\left|Z^n_s \right|^2 ds+\int_{0}^{T \wedge \tau}e^{\beta A_s} d\left[M^n\right]_s  \right]\right\rbrace 
				\leq \mathfrak{c}.
			\end{split}
			\label{with}
		\end{equation}	
		Using positive Fubini-Tonelli's theorem, we have
		$$
		\mathbb{E}_{\mathbb{Q}}\left[ \int_{0}^{T}e^{\beta A_s}\left| \hat{Z}^{n,\mathbb{F}}_s \right|^2\right]  ds\leq  \mathbb{E}_{\mathbb{Q}}\left[ \int_{0}^{T}e^{\beta A_s}\left| Z^n_s \right|^2 ds\right] =\mathbb{E}_{\mathbb{Q}}\left[ \int_{0}^{T \wedge \tau}e^{\beta A_s}\left| Z^n_s \right|^2 ds\right] \leq \mathfrak{c}.
		$$
		From the martingale representation property (\ref{some explication}), Itô's isometric formula, Jensen's and the B-D-G inequalities with (\ref{with}), we deduce that
		$$
		\mathbb{E}_{\mathbb{Q}}\left[ \int_{0}^{T}\left| \Theta^{n,\mathbb{F}}_s \right|^2 ds\right] =\mathbb{E}_{\mathbb{Q}}\left[ \left(\int_{0}^{T} \Theta^{n,\mathbb{F}}_s dB^{[\rho]}_s\right)^2\right]\leq 2 \mathbb{E}_{\mathbb{Q}}\left[ \sup_{0 \leq t \leq T \wedge \tau} e^{\beta A_t}\left| M^n_t\right|^2\right]\leq \mathfrak{c} .
		$$
		\paragraph*{Part 4: Convergence result}	
		\emph{}\\
		From Step 4 of the proof of Theorem 4.1 in \cite{BDEM},  we know that there exists a unique triplet of $\mathbb{G}$-adapted processes $(Y_t,Z_t,M_t)_{t \geq 0}$ such that
		\begin{itemize}
			\item $(Y_t)_{t \geq 0}$ is $\mathbb{G}$-adapted RCLL with $Y=\xi$ on $\llbracket T \wedge \tau,+\infty\llbracket$;
			\item $(Z_t)_{t \geq 0}$ is $\mathbb{G}$-predictable with $Z=Z\mathds{1}_{\llbracket0,T \wedge \tau \rrbracket}=Z\mathds{1}_{\llbracket0, \tau \rrbracket}$ on $[0,T]$;
			\item $(M_t)_{t \geq 0}$ is an RCLL $\mathbb{G}$-square integrable martingale orthogonal to $(B_t)_{t \geq 0}$ with $dM_{t}=dM_{t \wedge \tau}$ for all $t \in [0,T]$;
		\end{itemize}
		and
		\begin{equation}
			\lim\limits_{n \to +\infty}\mathbb{E}_{\mathbb{Q}}\left[\sup_{0 \leq t \leq T \wedge \tau}\left| Y^n_t-Y_t\right|^2+\int_{0}^{T \wedge \tau} \left|Z^n_s-Z_s \right|^2 ds+\int_{0}^{T \wedge \tau} d\left[M^n-M\right]_s  \right]=0.
			\label{CV}
		\end{equation}
		This together with (\ref{un Y}), (\ref{with}) and the conditional dominated convergence theorem, we derive the following convergences (in $\mathbb{L}^2$-space), for any  $t \in [0,T]$,
		\begin{equation}
			\left\{
			\begin{split}
				&\text{(i)}~\hat{Y}^{n,\mathbb{F}}_t= \mathbb{E}_{\mathbb{Q}}\left[Y^n_t \mid \mathcal{F}_t\right]\xrightarrow[n \to +\infty]{\mathcal{L}^2_{t}(\mathbb{F},\mathbb{Q})} \mathbb{E}_{\mathbb{Q}}\left[Y_t \mid \mathcal{F}_t\right]=:\hat{Y}^{\mathbb{F}}_t.\\			&\text{(ii)}~\hat{Z}^{n,\mathbb{F}}=\mathbb{E}_{\mathbb{Q}}\left[Z^n \mid \mathcal{F}_{\cdot}\right]\xrightarrow[n \to +\infty]{\mathcal{H}^2_t(\mathbb{F},\mathbb{Q})} \mathbb{E}_{\mathbb{Q}}\left[Z \mid \mathcal{F}_{\cdot}\right]. \\
				&\text{(iii)}~\mathbb{E}_{\mathbb{Q}}\left[M^n_T \right]+\int_{0}^{t} \Theta^{n,\mathbb{F}}_s dB^{[\rho]}_s\xrightarrow[n \to +\infty]{\mathcal{L}^2_{t}(\mathbb{F},\mathbb{Q})} \mathbb{E}_{\mathbb{Q}}\left[M_t \mid \mathcal{F}_t\right]=\mathbb{E}_{\mathbb{Q}}\left[M_{T}  \right]+\int_{0}^{t} \Theta^{\mathbb{F}}_s dB^{[\rho]}_s.
			\end{split}
			\right.
			\label{System}
		\end{equation}
		
		Using the same argument as in (\ref{SML}), there exists a unique
		$\mathbb{F}$-predictable process
		$\left(Z^{\mathbb{F}}_t\right)_{t\leq T}$ such that
		\[
		Z_t
		=
		Z^{\mathbb{F}}_t\mathds{1}_{\{t\leq\tau\}},
		\qquad t\in[0,T].
		\]
		Consequently,
		\[
		\mathbb{E}_{\mathbb{Q}}
		\left[
		Z_t\mid\mathcal{F}_t
		\right]
		=
		Z^{\mathbb{F}}_t
		\mathbb{Q}
		\left(
		t\leq\tau\mid\mathcal{F}_t
		\right),
		\qquad t\in[0,T].
		\]
		We then define
		\[
		\hat{Z}^{\mathbb{F}}_t
		:=
		Z^{\mathbb{F}}_t
		\mathbb{Q}
		\left(
		t\leq\tau\mid\mathcal{F}_t
		\right),
		\qquad t\in[0,T].
		\]
		Similarly, the process
		$\left(\Theta^{\mathbb{F}}_t\right)_{t\leq T}$
		is obtained as in (\ref{some explication}).
		
		As $(\hat{Z}^{n,\mathbb{F}},\Theta^{n,\mathbb{F}},\hat{Z}^{\mathbb{F}},\Theta^{\mathbb{F}})$ are uniquely defined as well as the processes $(Z^n,M^n,Z,M)$, and from the convergences given in (\ref{CV}) and (\ref{System}), the definition of the processes  $(\hat{Z}^{n,\mathbb{F}},\hat{Z}^{\mathbb{F}})$, using Remark  \ref{strssfull remark} for the processes $(\hat{Z}^{n,\mathbb{F}},\hat{Z}^{\mathbb{F}})$, and Jensen's inequality, we can easily deduce that for any $t \in [0,T]$,
		$$
		\lim\limits_{n\to +\infty}\mathbb{E}_{\mathbb{Q}}\left[\left(\int_{0}^{t}\hat{Z}^{n,\mathbb{F}}_s dB^{[\rho]}_s-\int_{0}^{t}\hat{Z}^{\mathbb{F}}_s dB^{[\rho]}_s\right)^2+\left(\int_{0}^{t}\Theta^{n,\mathbb{F}}_s dB^{[\rho]}_s-\int_{0}^{t}\Theta^{\mathbb{F}}_s dB^{[\rho]}_s\right)^2\right] =0,
		$$
		Moreover, since  the process $(\Theta^{\rho,\ast}_t)_{t \leq T}$ is square integrable, i.e., $\Theta^{\rho,\ast} \in \mathcal{H}^2_{T}(\mathbb{F},\mathbb{Q})$, then using the Lebesgue dominate convergence theorem, we get
		$$\int_{0}^{t} \hat{Z}^{n,\mathbb{F}}_s \Theta^{\rho,\ast}_s ds\xrightarrow[n \to +\infty]{\mathcal{L}^1_t(\mathbb{F},\mathbb{Q})} \int_{0}^{t} \hat{Z}^{\mathbb{F}}_s \Theta^{\rho,\ast}_s ds,\quad t \in [0,T].$$
		\paragraph*{Part 5: Cauchy property}
		\emph{}\\
		Following this convergence results and using the fact that $\left\{Y^n,Z^n,M^n\right\}_{n \in \mathbb{N}}$ are Cauchy sequences, along with Remark \ref{strssfull remark} and Jensen's inequality, we can deduce that $\left\{\hat{Y}^{n,\mathbb{F}},\hat{Z}^{n,\mathbb{F}},\Theta^{n,\mathbb{F}}\right\}_{n \in \mathbb{N}}$ are also Cauchy sequences. More precisely, we can derive
		\begin{equation}
			\left\{
			\begin{split}
				&\bullet \lim\limits_{n,p\to +\infty}\mathbb{E}_{\mathbb{Q}}\left[\sup_{0 \leq t \leq T}\left| \hat{Y}^{n,\mathbb{F}}_t-\hat{Y}^{p,\mathbb{F}}_t\right|^2 \right]=0.\\
				&\bullet \text{ For any } t \in [0,T],\\
				&\lim\limits_{n,p\to +\infty}\mathbb{E}_{\mathbb{Q}}\left[\left(\int_{0}^{t}\left( \hat{Z}^{n,\mathbb{F}}_s -\hat{Z}^{p,\mathbb{F}}_s\right)  dB^{[\rho]}_s\right)^2+\left(\int_{0}^{t}\left( \Theta^{n,\mathbb{F}}_s -\Theta^{p,\mathbb{F}}_s\right)  dB^{[\rho]}_s\right)^2\right] =0.\\
			\end{split}
			\right.
			\label{CV SYS}
		\end{equation}
		Furthermore, since $\Theta^{\rho,\ast} \in \mathcal{H}^2_{T}(\mathbb{F},\mathbb{Q})$, we can show using the second line of (\ref{CV SYS}), that
		\begin{equation}
			\lim\limits_{n,p\to +\infty}\mathbb{E}_{\mathbb{Q}}\left[\left|  \int_{0}^{t}\left( \hat{Z}^{n,\mathbb{F}}_s  -\hat{Z}^{p,\mathbb{F}}_s\right) \Theta^{\rho,\ast}_s ds \right| \right] =0
			,\quad  t \in [0,T].
			\label{CV SYST ADD}
		\end{equation}
		Now, we set $K^{n,\mathbb{F}}_t:=K^{n,\mathbb{F},+}_t-K^{n,\mathbb{F},-}_t$ for $t \in [0,T]$. From the BSDE (\ref{New penalization approximation}), we obtain using (\ref{CV SYS}) and (\ref{CV SYST ADD}), that
		$$
		\lim\limits_{n,p\to +\infty}\mathbb{E}_{\mathbb{Q}}\left[\sup_{0 \leq t \leq T}\left|  K^{n,\mathbb{F}}_t-K^{p,\mathbb{F}}_t \right| \right] =0.
		$$
		Thus, $\left\{K^{n,\mathbb{F}}\right\}_{n \in \mathbb{N}}$ is a Cauchy sequence in the complete space $\mathcal{L}^1_T(\mathbb{F},\mathbb{Q})$. Therefore, there exists an $\mathbb{F}$-predictable  process $K^{\mathbb{F}}$ such that $\mathbb{E}_{\mathbb{Q}}\left|  K^{\mathbb{F}}_T\right| <+\infty$. Additionally, from the uniform estimation (\ref{KnF+-}) satisfied by $\left\{K^{n,\mathbb{F},\pm}\right\}_{n \in \mathbb{N}}$ and Fatou's lemma, we get 
		\begin{equation*}
			\begin{split}
				\mathbb{E}_{\mathbb{Q}}\left|  K^{\mathbb{F}}_T\right|^2 & \leq  \liminf_{n \to +\infty} \mathbb{E}_{\mathbb{Q}}\left|  K^{n,\mathbb{F}}_T\right|^2 
				\leq 2 \left\{  \limsup_{n \to +\infty} \mathbb{E}_{\mathbb{Q}}\left|  K^{n,\mathbb{F},+}_T\right|^2+  \limsup_{n \to +\infty} \mathbb{E}_{\mathbb{Q}}\left|  K^{n,\mathbb{F},-}_T\right|^2\right\} \leq 4 \mathfrak{c}.
			\end{split}
		\end{equation*}
		Then, we may obtain, at least for a subsequence, that	
		\begin{equation}
			\lim\limits_{n\to +\infty}\mathbb{E}_{\mathbb{Q}}\left[\sup_{0 \leq t \leq T}\left|  K^{n,\mathbb{F}}_t-K^{\mathbb{F}}_t \right|^2\right] =0.
			\label{Convergence of Kn}
		\end{equation}
		\paragraph*{Part 6: BSDE satisfied by the limiting process $\left( \hat{Y}^{\mathbb{F}}_t,\hat{Z}^{\mathbb{F}}_t,K^{\mathbb{F}}_t,\Theta^{\mathbb{F}}_t \right)_{t \leq T}$}
		\emph{}\\
		Passing to the limit term by term in $\mathcal{L}^1_T(\mathbb{F},\mathbb{Q})$ as $n \to+\infty$ in (\ref{New penalization approximation}), we obtain that the quadruplet $(\hat{Y}^{\mathbb{F}}_t,\hat{Z}^{\mathbb{F}}_t,\Theta^{\mathbb{F}}_t,K^{\mathbb{F}}_t)_{t \leq T}$ is the unique quadruplet satisfying the following BSDE:
		\begin{equation}
			\begin{split}
				\hat{Y}^{\mathbb{F}}_t= 	\hat{\xi}^{\mathbb{F}}+\int_t^T \left\lbrace \hat{g}^{\mathbb{F}}(s)+\hat{Z}^{\mathbb{F}}_s \Theta^{\rho,\ast}_s \right\rbrace  ds+\left( K^{\mathbb{F}}_T-K^{\mathbb{F}}_t\right) -\int_t^T \hat{Z}^{\mathbb{F}}_s  d B^{[\rho]}_s-\int_{t}^{T} \Theta^{\mathbb{F}}_s  d B^{[\rho]}_s,
			\end{split}
			\label{Dyn in F}
		\end{equation}
		for all $t \in [0,T]$, $\mathbb{Q}$-a.s., with terminal value $
		\hat{\xi}^{\mathbb{F}}=\mathbb{E}_{\mathbb{Q}}\left[\xi \mid \mathcal{F}_T\right]$.\\	
		Observe that the sequence of processes $\left\{\hat{Y}^{n,\mathbb{F}}\right\}_{n \in \mathbb{N}}$, defined by the BSDEs in (\ref{New penalization approximation}), possesses continuous paths due to the fact that the process $B^{[\rho]}$ is an $(\mathbb{F},\mathbb{Q})$-Brownian motion. Consequently, considering the paths of the limiting process $\hat{Y}^{\mathbb{F}}$ as a uniform limit of continuous functions, we conclude that $\hat{Y}^{\mathbb{F}}$ is also continuous, exhibiting continuous trajectories on the interval $[0,T]$. Furthermore, we deduce that, the limit of the sequence $\left\{K^{n,\mathbb{F}}\right\}_{n \in \mathbb{N}}$ given by $K^{\mathbb{F}}$ also has continuous paths on $[0,T]$.
		
		Let us now show $L_t \leq \hat{Y}^{\mathbb{F}}_t \leq U_t$, for all $t \in [0,T]$. On one hand, we know that $L=L^{T \wedge \tau}$, $U=U^{T \wedge \tau}$, and $L_{T \wedge \tau} \leq \xi \leq U_{T \wedge \tau}$, by assumption. Therefore, $L_{\tau}=L_T=\mathbb{E}_{\mathbb{Q}}\left[L_T \mid \mathcal{F}_T\right]$ and $U_{\tau}=U_T=\mathbb{E}_{\mathbb{Q}}\left[U_T \mid \mathcal{F}_T\right]$, as $L$ and $U$ are $\mathbb{F}$-adapted by assumption. \\	
		On the other hand, since $L_t \leq Y_t \leq U_t$ $\mathbb{Q}$-a.s. for all $t \in [0,T \wedge \tau[$ and $Y_t=\xi$, $\mathbb{Q}$-a.s. on the set $\{t \geq T \wedge \tau\}$ since $Y=Y^{T \wedge \tau}$, we deduce that $L_t \leq \hat{Y}^{\mathbb{F}}_t \leq U_t$, $\forall t \in [0,T]$,  $\mathbb{Q}$-a.s. 
		\paragraph*{Part 7: Skorokhod condition}
		\emph{}\\
		Using the expressions $\mathsf{k}^{n,\mathbb{F},+}:=\left(n\mathbb{E}_{\mathbb{Q}} \left[ (L_t-Y^{n}_t)^{+}\mathds{1}_{\{t < \tau\}} \mid \mathcal{F}_s \right]\right)_{t \leq T}$ and $\mathsf{k}^{n,\mathbb{F},-}=\left(n\mathbb{E}_{\mathbb{Q}} \left[ (Y^{n}_t-U_t)^{+}\mathds{1}_{\{s < \tau\}} \mid \mathcal{F}_t \right]\right)_{t \leq T}$, we point out  that $\mathsf{k}^{n,\mathbb{F},+}_t=n\mathbb{E}_{\mathbb{Q}} \left[ (L_t-Y^{n}_t)^{+}\mathds{1}_{\{t \leq \tau\}} \mid \mathcal{F}_s \right]$ and $\mathsf{k}^{n,\mathbb{F},-}_t=n\mathbb{E}_{\mathbb{Q}} \left[ (Y^{n}_t-U_t)^{+}\mathds{1}_{\{s \leq  \tau\}} \mid \mathcal{F}_t \right]$ for $\mathbb{P} \otimes dt$ almost every $(\omega,t)$. Now, from the previous part concerning the sequence $\left\{K^{n,\mathbb{F},\pm}\right\}_{n \in \mathbb{N}}$, we have that $\mathsf{k}^{n,\mathbb{F},\pm}$ is a sequence of non-negative, $\mathbb{F}$-adapted processes, which is bounded in $\mathbb{L}^2(\Omega \times [0,T],\mathbb{P} \otimes dt)$. Consequently, there exist a non-negative $\mathbb{F}$-adapted measurable processes $\mathsf{k}^{\mathbb{F},\pm}$ such that 
		\begin{equation}
			\mathsf{k}^{n,\mathbb{F},\pm} \to \mathsf{k}^{\mathbb{F},\pm} ~\text{ as } n \to +\infty, ~\text{ weakly in } \mathbb{L}^2(\Omega \times [0,T],\mathbb{P} \otimes dt).
			\label{Weak convergence}
		\end{equation}
		On the other hand, note that
		\begin{equation*}
			\begin{split}
				\mathbb{E}_{\mathbb{Q}}\int_{0}^{T}\left(\hat{Y}^{n,\mathbb{F}}_s-L_s\right)dK^{n,\mathbb{F},+}_s 
				&= \mathbb{E}_{\mathbb{Q}}\int_{0}^{T}\left(\hat{Y}^{n,\mathbb{F}}_s-\hat{Y}^{\mathbb{F}}_s\right)\mathsf{k}^{n,\mathbb{F},+}_s ds+\mathbb{E}_{\mathbb{Q}}\int_{0}^{T}\left(\hat{Y}^{\mathbb{F}}_s-L_s\right)\mathsf{k}^{\mathbb{F},+}_s ds\\
				&\quad+\mathbb{E}_{\mathbb{Q}}\int_{0}^{T}\left(\hat{Y}^{\mathbb{F}}_s-L_s\right)\left( \mathsf{k}^{n,\mathbb{F},+}_s-\mathsf{k}^{\mathbb{F},+}_s \right) ds
			\end{split}
		\end{equation*}
		Using (\ref{KnF+-}), (\ref{System}), (\ref{CV SYS}), (\ref{Weak convergence}), and the fact that $\hat{Y}^{\mathbb{F}}-L \in \mathcal{D}^2_{T}\left( \mathbb{F},\mathbb{Q}\right)$, we deduce that
		\begin{equation*}
			\begin{split}
				&\mathbb{E}_{\mathbb{Q}}\int_{0}^{T}\left(\hat{Y}^{n,\mathbb{F}}_s-L_s\right)dK^{n,\mathbb{F},+}_s \xrightarrow[n \to+ \infty]{} \mathbb{E}_{\mathbb{Q}}\int_{0}^{T}\left(\hat{Y}^{\mathbb{F}}_s-L_s\right)\mathsf{k}^{\mathbb{F},+}_s ds.
			\end{split}
		\end{equation*}
		Following a similar argument, we get
		\begin{equation*}
			\begin{split}
				&\mathbb{E}_{\mathbb{Q}}\int_{0}^{T}\left(U_s-\hat{Y}^{n,\mathbb{F}}_s\right)dK^{n,\mathbb{F},-}_s \xrightarrow[n \to+ \infty]{}\mathbb{E}_{\mathbb{Q}}\int_{0}^{T}\left(U_s-\hat{Y}^{\mathbb{F}}_s\right)\mathsf{k}^{\mathbb{F},-}_s ds .
			\end{split}
		\end{equation*}
		On the other hand, from Remark \ref{strssfull remark}, Fubini's Theorem, and the $\mathbb{F}$-adaptation of the lower barrier $L$, we have for every $n \in \mathbb{N}$:
		\begin{equation*}
			\begin{split}
				\int_{0}^{T} \left(\hat{Y}^{n,\mathbb{F}}_s-L_s\right)\mathsf{k}^{n,\mathbb{F},+}_s ds&=n\int_{0}^{T}\mathbb{E}_{\mathbb{Q}}\left[\left(\hat{Y}^{n,\mathbb{F}}_s-L_s\right)\left(Y^n_s-L_s\right)^{-}\mathds{1}_{\{s \leq \tau\}} \mid \mathcal{F}_s\right]ds \\
				&=n\int_{0}^{T}\mathbb{E}_{\mathbb{Q}}\left[\left(\hat{Y}^{n,\mathbb{F}}_s-L_s\right)\left(Y^n_s-L_s\right)^{-}\mathds{1}_{\{s \leq \tau\}} \mid \mathcal{F}_T\right]ds \\
				&=n\mathbb{E}_{\mathbb{Q}}\left[\int_{0}^{T}\left(\hat{Y}^{n,\mathbb{F}}_s-L_s\right)\left(Y^n_s-L_s\right)^{-}\mathds{1}_{\{s \leq \tau\}}ds \mid \mathcal{F}_T\right]\\
				&=n\mathbb{E}_{\mathbb{Q}}\left[\int_{0}^{T}\left(\mathbb{E}_{\mathbb{Q}}\left[Y^n_s \mid \mathcal{F}_s\right]-L_s\right)\left(Y^n_s-L_s\right)^{-}\mathds{1}_{\{s \leq \tau\}}ds \mid \mathcal{F}_T\right]\\
				& \leq n\mathbb{E}_{\mathbb{Q}}\left[\int_{0}^{T}\left(\mathbb{E}_{\mathbb{Q}}\left[\left( Y^n_s-L_s\right)^{+} \mid \mathcal{F}_s\right]\right)\left(Y^n_s-L_s\right)^{-}\mathds{1}_{\{s \leq \tau\}}ds \mid \mathcal{F}_T\right]\\
				&=0.
			\end{split}
		\end{equation*}
		Clearly, the last statement stems from the monotonicity of the conditional expectation, and the following observation: if $Y^n_s \leq L_s$, then $\left( Y^n_s-L_s \right)^{+}= 0$ a.s. Otherwise, if $Y^n >L$, then $\left(Y^n_s-L_s\right)^{-} = 0$.\\ 
		Therefore, using this with the fact that  $\int_{0}^{T}(\hat{Y}^{\mathbb{F}}_s-L_s)\mathsf{k}^{\mathbb{F},+}_s ds \geq 0$ and $\hat{Y}^{\mathbb{F}} \geq L$, we get
		$$
		\int_{0}^{T}(\hat{Y}^{\mathbb{F}}_s-L_s)\mathsf{k}^{\mathbb{F},+}_s ds \geq 0 \geq \int_{0}^{T} \left(\hat{Y}^{n,\mathbb{F}}_s-L_s\right)\mathsf{k}^{n,\mathbb{F},+}_s ds,\quad\forall n \in \mathbb{N},
		$$
		and then, we derive that $\mathbb{E}_{\mathbb{Q}}\left[ \int_{0}^{T}(\hat{Y}^{\mathbb{F}}_s-L_s)\mathsf{k}^{\mathbb{F},+}_s ds\right]=0$, which implies
		\begin{equation}
			\int_{0}^{T}(\hat{Y}^{\mathbb{F}}_s-L_s)dK^{\mathbb{F},+}_s ,\quad \mathbb{Q}\text{-a.s.}
			\label{Eq1}
		\end{equation}
		A similar argument gives
		\begin{equation}
			\int_{0}^{T}(U_s-\hat{Y}^{\mathbb{F}}_s)d{K}^{\mathbb{F},-}_s=0,\quad \mathbb{Q}\text{-a.s.}
			\label{Eq2}
		\end{equation}
		It remain to show that the limiting process  $K^{\mathbb{F}}$ obtained in \textbf{Part 5} is equal to the difference process $K^{\mathbb{F},+}-K^{\mathbb{F},-}$. To this end, let us set
		\begin{equation}
			\bar{K}^{\mathbb{F}}_t:=K^{\mathbb{F},+}_t-K^{\mathbb{F},-}_t=\int_{0}^{t} \mathsf{k}^{\mathbb{F},+}_s ds-\int_{0}^{t} \mathsf{k}^{\mathbb{F},-}_s ds,\quad t \in [0,T].
			\label{Definition}
		\end{equation}
		Let $\sigma \in \mathcal{T}_{0,T}$. By Mazur's Theorem \cite[Theorem 2, pp. 120]{yosida2012functional}, for any $n \in \mathbb{N}$, there exists an integer $q(n) \geq n$, weights $(\gamma^{(\tau,n)}_j)_{j\in \{n,\cdots,q(n)\}} \subset \mathbb{R}^{+}$ and a convex combination $\sum_{j=n}^{q(n)} \gamma^{(\tau,n)}_j K^{j ,\pm}_{\tau}$ such that $\sum_{j=n}^q \gamma^{(\tau,n)}_j =1$ and
		\begin{equation}
			\bar{K}^{n,\mathbb{F},\pm}_{\sigma}:=\sum_{j=n}^{q(n)} \gamma^{(\sigma,n)}_j K^{j,\mathbb{F} ,\pm}_{\sigma} \xrightarrow[n \rightarrow +\infty]{} K^{\mathbb{F},\pm}_{\sigma} \text{ in } \mathcal{L}^2_{\sigma}(\mathbb{F},\mathbb{Q}).
			\label{Mazur's Lemmma that--Sk condition--T}
		\end{equation}
		Denoting $\bar{K}_{\sigma}^{n,\mathbb{F}}:=\bar{K}_{\sigma}^{n,\mathbb{F},+}-\bar{K}_{\sigma}^{n,\mathbb{F},-}$, it follows from (\ref{Mazur's Lemmma that--Sk condition--T}), that
		\begin{equation}
			\lim\limits_{n \rightarrow+\infty}\mathbb{E}_{\mathbb{Q}}\left[ \left| \bar{K}_{\sigma}^{n,\mathbb{F}}-\bar{K}^{\mathbb{F}}_{\sigma}  \right|^2 \right] =0.
			\label{Combining 1--Skorokhod Condition--T}
		\end{equation}
		From the convergence (\ref{Convergence of Kn}), we deduce that the inequality $\left\| K^{n,\mathbb{F}}_{\sigma}-K^{\mathbb{F}}_{\sigma} \right\|_{\mathcal{L}^2_{\sigma}(\mathbb{F},\mathbb{Q})}<\epsilon$ holds for arbitrary $\epsilon > 0$ and all sufficiently large $n\in \mathbb{N}$. Therefore, we also have as 
		\begin{equation*}
			\begin{split}
				\left\|\bar{K}^{n,\mathbb{F}}_{\sigma}-K^{\mathbb{F}}_{\sigma}  \right\|_{\mathcal{L}^2_{\sigma}(\mathbb{F},\mathbb{Q})}
				&=\left\| \sum_{j=n}^{q(n)} \gamma^{(\sigma,n)}_j \left( K^{j,\mathbb{F}}_{\sigma} -K^{\mathbb{F}}_{\sigma}\right)  \right\|_{\mathcal{L}^2_{\sigma}(\mathbb{F},\mathbb{Q})}\\
				&\leq \sum_{j=n}^{q(n)} \gamma^{(\sigma,n)}_j \left\|  K^{j,\mathbb{F}}_{\sigma} -K^\mathbb{F}_{\sigma} \right\|_{\mathcal{L}^2_{\sigma}(\mathbb{F},\mathbb{Q})} <\epsilon.
			\end{split}
		\end{equation*}
		So that
		\begin{equation}
			\lim\limits_{n \rightarrow +\infty} \mathbb{E}_{\mathbb{Q}}\left[ \left|\bar{K}^{n,\mathbb{F}}_{\sigma}-K^{\mathbb{F}}_{\sigma}   \right|^2  \right]=0.
			\label{Combining 2--Skorokhod condition--T}
		\end{equation}
		Combining (\ref{Combining 1--Skorokhod Condition--T}) and (\ref{Combining 2--Skorokhod condition--T}), we obtain $\bar{K}^{\mathbb{F}}_{\sigma}=K^{\mathbb{F}}_{\sigma}$ $\mathbb{P}$-a.s. Hence, by the section theorem \cite[Theorem 4.7]{bookHe}, we have $\bar{K}^{\mathbb{F}}_t=K^{\mathbb{F}}_t$ for all $t \in  [0,T]$, $\mathbb{Q}$-a.s.\\
		From the Definition (\ref{Definition}), we deduce that $K^{\mathbb{F}}=K^{\mathbb{F},+}-K^{\mathbb{F},-}$, and using (\ref{Eq1}) and (\ref{Eq2}), we obtain
		$$
		\int_{0}^{T}(\hat{Y}^{\mathbb{F}}_s-L_s)dK^{\mathbb{F},+}_s=\int_{0}^{T}(U_s-\hat{Y}^{\mathbb{F}}_s)dK^{\mathbb{F},-}_s =0,\quad \mathbb{Q}\text{-a.s.}
		$$
		Now, going back to (\ref{Dyn in F}), we derive
		\begin{equation*}
			\begin{split}
				\hat{Y}^{\mathbb{F}}_t= 	\hat{\xi}^{\mathbb{F}}&+\int_t^T \left\lbrace \hat{g}^{\mathbb{F}}(s)+\hat{Z}^{\mathbb{F}}_s \Theta^{\rho,\ast}_s \right\rbrace  ds+\left( K^{\mathbb{F},+}_T-K^{\mathbb{F},+}_t\right)-\left( K^{\mathbb{F},-}_T-K^{\mathbb{F},-}_t\right)\\
				&\qquad-\int_t^T \hat{Z}^{\mathbb{F}}_s  d B^{[\rho]}_s-\int_{t}^{T} \Theta^{\mathbb{F}}_s  d B^{[\rho]}_s,\quad t \in [0,T].
			\end{split}
		\end{equation*}
		\paragraph*{Part 8: Integrability property}
		\emph{}\\
		The integrability condition satisfied by $(\hat{Y}^{\mathbb{F}}_t,\hat{Z}^{\mathbb{F}}_t,\Theta^{\mathbb{F}}_t,K^{\mathbb{F},+}_t,K^{\mathbb{F},-}_t)_{t \leq T}$ is obtained from the construction of these processes, by applying Fatou's lemma for the corresponding sequences and the already obtained estimations. With this, the proof is complete.
	\end{proof}

	\begin{remark}
		Following the same argument as the one used in the proof of Theorem \ref{Proposition of comparison}, we can obtain a more general result in the case where the driver $f$ of the DRBSDE (\ref{basic equation}) is linear with respect to the $y$-variable. More precisely, let $T \in (0,+\infty)$ and assume that the coefficient $f$ is linear of the form
		\begin{equation*}
			f(t,y,z)=g(t)+\kappa_t y, \quad t \in [0,T],
			\label{Linear generator}
		\end{equation*} 
		where the process $(\kappa_t)_{t \leq T}$ is the same as the one presented in assumption ($\mathbf{H_2}$).\\ 
		Assume that conditions of Theorem \ref{Proposition of comparison} hold and that the increasing $\mathbb{F}$-adapted continuous process $\left(\varpi_t\right)_{t \leq T}:=\left(\int_{0}^{t}\kappa_s ds\right)_{t \leq T}$ is bounded. Then, there exists a unique quintuplet $(\hat{Y}^{\mathbb{F}}_t,\hat{Z}^{\mathbb{F}}_t,\Theta^{\mathbb{F}}_t,K^{\mathbb{F},+}_t,K^{\mathbb{F},-}_t)_{t \leq T}$ that belongs to  $\mathfrak{B}^2_{T}\left( \beta,\mathbb{F},\mathbb{Q}\right)$ and satisfies of the following special form of linear generalized doubly reflected BSDE :
		\begin{equation*}
			\left\{
			\begin{split}
				&\text{(i)}~\hat{Y}^{\mathbb{F}}_t= 	\hat{\xi}+\int_t^T \left\lbrace \hat{g}^{\mathbb{F}}(s)+\kappa_s\hat{Y}^{\mathbb{F}}_s+\Theta^{\rho,\ast}_s \hat{Z}^{\mathbb{F}}_s \right\rbrace  ds-\int_{t}^{T}\int_{s}^{\infty} \zeta_u dD^{o,\mathbb{F},\mathbb{Q}}_u ds+\left( K^{\mathbb{F},+}_T-K^{\mathbb{F},+}_t\right)\\
				&\qquad\qquad-\left( K^{\mathbb{F},-}_T-K^{\mathbb{F},-}_t\right)-\int_t^T \hat{Z}^{\mathbb{F}}_s  d B^{[\rho]}_s-\int_{t}^{T} \Theta^{\mathbb{F}}_s  d B^{[\rho]}_s,\quad t \in [0,T].\\
				&\text{(ii)}~L_t \leq \hat{Y}^{\mathbb{F}}_t \leq U_t,~ \forall t \in [0,T], ~ \mathbb{Q}\text{-a.s.} \\
				&\text{(iii)}~\int_{0}^{T}(\hat{Y}^{\mathbb{F}}_s-L_s)dK^{\mathbb{F},+}_s=\int_{0}^{T}(U_s-\hat{Y}^{\mathbb{F}}_s)dK^{\mathbb{F},-}_s =0,~ \mathbb{Q}\text{-a.s.}
			\end{split}
			\right.
		\end{equation*}
	\end{remark}
	Throughout the rest of the paper, we define the $\mathbb{G}$-stopping time $\tau^{T}$ as $\tau^{T}(\omega) := T \wedge \tau(\omega)$ for the sake of notational simplicity. 
	%
	%
	\section{$\mathcal{E}^f$-Dynkin games with default time and links with nonlinear doubly reflected BSDEs}
	\label{Dynkin games}
	In this section, our focus is on establishing a connection between the DRBSDE (\ref{basic equation}) and a Dynkin game with default $\tau$ involving two players. The payoffs for these players are derived from the DRBSDE data and are evaluated through non-linear $f$-expectations. To be more specific, we interpret the solution of our Doubly Reflected BSDE in the context of the value process associated with a carefully defined Dynkin game problem with default incorporating non-linear $f$-expectations.
	\subsection{Dynkin game with default time and late payment}
	\label{Dynkin game with default time and late payment}
	Now, let's introduce the model of our Dynkin game with a default feature and the mathematical structure that characterizes it.

	We begin by introducing the abstract structure of the Dynkin game and the
	strategies of the two players. The game involves two players with opposite
	objectives:
	\begin{itemize}
		\item \textbf{Player 1} ($\mathbf{P_1}$) is the maximizing player and
		chooses a stopping time $\sigma_1$;
		
		\item \textbf{Player 2} ($\mathbf{P_2}$) is the minimizing player and
		chooses a stopping time $\sigma_2$.
	\end{itemize}
	
	The game is considered over a finite deterministic horizon
	$T\in(0,+\infty)$ and is also subject to the random time $\tau$.
	Consequently, the effective horizon of the game is
	\begin{equation*}
		\tau^T := T\wedge\tau.
	\end{equation*}
	
	Each player chooses a stopping time before or at $\tau^T$, and the game
	terminates at the first of the two stopping times, subject to the terminal
	time $\tau^T$. The corresponding payoff is determined by the processes
	$L$, $U$, and $Q$, together with the terminal payoff $\xi$, as specified
	below. Since these payoff processes are not assumed to be non-negative,
	the interpretation of the game is given in terms of the maximizing and
	minimizing objectives of the two players rather than in terms of a fixed
	direction of payment between them.
	
	We consider a measurable space $(\Omega,\mathcal{F})$ endowed with two
	filtered probability settings. The first one is
	$(\mathbb{F},\mathbb{P})$, where $\mathbb{F}$ denotes the natural
	filtration generated by the Brownian motion $B$. The second one is the
	defaultable framework $(\mathbb{G},\mathbb{Q})$, where $\mathbb{G}$ is
	the progressively enlarged filtration defined in
	(\ref{Filtration}) and the probability measure $\mathbb{Q}$ is
	introduced in Lemma~\ref{lemma.Q}.
	
	We consider two $\mathbb{F}$-adapted processes
	$(L_t)_{t\leq T}$ and $(U_t)_{t\leq T}$ such that
	\[
	L\in\mathcal{S}^2_T(2\beta,\mathbb{F},\mathbb{P}),
	\qquad
	U\in\mathcal{S}^2_T(2\beta,\mathbb{F},\mathbb{P}),
	\]
	and
	\[
	L_t\leq U_t,
	\qquad t\in[0,T].
	\]
	
	In the enlarged filtration $\mathbb{G}$, we work with the processes
	stopped at the effective horizon
	\[
	\tau^T:=T\wedge\tau,
	\]
	namely
	\[
	L^{\tau^T}_t:=L_{t\wedge\tau^T},
	\qquad
	U^{\tau^T}_t:=U_{t\wedge\tau^T},
	\qquad t\in[0,T].
	\]
	
	We further assume the complete separation conditions
	\[
	L_t<U_t
	\quad\text{on }\llbracket 0,\tau^T\llbracket,
	\qquad \mathbb{Q}\text{-a.s.},
	\]
	and
	\[
	L_{t-}<U_{t-}
	\quad\text{on }\rrbracket 0,\tau^T\rrbracket,
	\qquad \mathbb{Q}\text{-a.s.}
	\]
	
	We recall that  $\mathcal{S}^2_{T}(2\beta,\mathbb{F},\mathbb{P}) \subset \mathcal{S}^2_{\tau^T}(2\beta,\mathbb{G},\mathbb{Q})$, and we assume that $L < U$ holds on the interval $\llbracket 0,\tau^T \llbracket$, and that $L_{-} < U_{-}$ holds on the interval $\llbracket 0,\tau^T \rrbracket$ $\mathbb{Q}$-a.s. We also consider a $\mathcal{G}_{T \wedge \tau}$ measurable random variable $\xi$ such that $L_{\tau^T} \leq \xi \leq U_{\tau^T}$  $\mathbb{Q}$-a.s.\\
	The players' strategies are represented by stopping times from the set $\mathfrak{T}_{0,\tau^T}(\mathbb{G})$, indicating that neither player possesses clairvoyance or perfect prediction in the game. Specifically, Player $\mathbf{P_1}$ selects a stopping time  $\sigma_1$ to end the game, while Player $\mathbf{P_2}$ chooses a stopping time $\sigma_2$ to cancel the game. As a result, the game concludes at time $\sigma_1 \wedge \sigma_2$, and the payoff (or reward) used to evaluate the outcome is denoted as $\mathcal{J}\left(\sigma_1,\sigma_2\right)$, which is defined as the following random variable:
	\begin{equation}
		\mathcal{J}\left(\sigma_1,\sigma_2\right)=L_{\sigma_1} \mathds{1}_{\{\sigma_1 < \sigma_2\}}+U_{\sigma_2} \mathds{1}_{\{\sigma_2 <\sigma_1\}}+Q_{\sigma_{1}} \mathds{1}_{\{\sigma_1 = \sigma_2<\tau^T\}}+\xi \mathds{1}_{\{\sigma_1=\sigma_2=\tau^T\}}
		\label{Payoff}
	\end{equation}
	where
	\begin{itemize}
		\item $(Q_t)_{t \leq T}$ is an $\mathbb{F}$-progressively measurable processes such that $L_t \leq Q_t \leq U_t$ for all $t \in [0,\tau^T[$, $\mathbb{Q}$-a.s.	
		
		\item 
		We consider a terminal random variable $\xi$ such that
		\[
		\xi\in\mathcal{F}_{T\wedge\tau}
		\subseteq
		\mathcal{G}_{T\wedge\tau},
		\]
		and
		\[
		L_{T\wedge\tau}
		\leq
		\xi
		\leq
		U_{T\wedge\tau},
		\qquad
		\mathbb{Q}\text{-a.s.}
		\]
		
		More precisely, by the definition of $\mathcal{F}_{T\wedge\tau}$,
		there exists an $\mathbb{F}$-optional RCLL process
		$\zeta=(\zeta_t)_{t\geq0}$ such that
		\[
		\xi=\zeta_{T\wedge\tau}.
		\]
		Consequently, $\xi$ admits the decomposition
		\begin{equation}
			\xi
			=
			\xi_1\mathds{1}_{\{T<\tau\}}
			+
			\xi_2\mathds{1}_{\{\tau\leq T\}},
			\label{Terminal cost}
		\end{equation}
		where
		\[
		\xi_1:=\zeta_T
		\qquad\text{and}\qquad
		\xi_2:=\zeta_\tau.
		\]
		In particular, $\xi_1$ is $\mathcal{F}_T$-measurable, while the
		default payoff $\xi_2$ is obtained by evaluating the same
		$\mathbb{F}$-optional process $\zeta$ at the random time $\tau$.
		Therefore, the terminal payoff remains
		$\mathcal{F}_{T\wedge\tau}$-measurable.
		
	\end{itemize}

	\begin{example}
		Let $B$ be the Brownian motion generating the filtration $\mathbb{F}$ and
		consider the bounded $\mathbb{F}$-adapted continuous process
		\[
		\zeta_t:=\tanh(B_t), \qquad t\in[0,T].
		\]
		Define the terminal random variable by
		\[
		\xi:=\zeta_{T\wedge\tau}.
		\]
		Then, by the definition of $\mathcal{F}_{T\wedge\tau}$,
		\[
		\xi\in\mathcal{F}_{T\wedge\tau}
		\subseteq\mathcal{G}_{T\wedge\tau}.
		\]
		More explicitly,
		\[
		\xi
		=
		\tanh(B_T)\mathds{1}_{\{T<\tau\}}
		+
		\tanh(B_\tau)\mathds{1}_{\{\tau\leq T\}}.
		\]
		Moreover, if
		\[
		L_t:=\tanh(B_t)-1
		\qquad\text{and}\qquad
		U_t:=\tanh(B_t)+1,
		\]
		then
		\[
		L_{T\wedge\tau}
		<
		\xi
		<
		U_{T\wedge\tau},
		\qquad
		\mathbb{Q}\text{-a.s.}
		\]
		Since $\xi$ is bounded, the required square-integrability condition is
		also satisfied.
	\end{example}

	\begin{remark}
		\begin{itemize}
			\item For each $\omega \in \Omega$ and every $t \in [0,\tau^T[$, the differences:
			\begin{itemize}
				\item $U_t(\omega)-L_t(\omega)$ is the money penalty that has from Player $\mathbf{P_2}$ to $\mathbf{P_1}$ for his unilateral decision to exit the game.
				\item $Q_t(\omega)-L_t(\omega)$, it is also a money penalty for Player $\mathbf{P_2}$ which corresponds to the situation where both parts agree to stop playing before $\tau^T$.
			\end{itemize}
			\item Given the terminal cost (\ref{Terminal cost}) as provided in (\ref{Payoff}), we can express it as
			$$
			\xi \mathds{1}_{\{\sigma_1=\sigma_2=\tau^T\}}=\xi_1  \mathds{1}_{\{\sigma_1=\sigma_2=T\}} \mathds{1}_{\{T <\tau\}}+\xi_2  \mathds{1}_{\{\sigma_1=\sigma_2=\tau\}} \mathds{1}_{\{\tau \leq T\}} \in \mathcal{G}_{\sigma_1 \wedge \sigma_2}.
			$$
			\item The payoff (\ref{Payoff}) of the game can be  interpreted as follows:
			\begin{equation*}
				\mathcal{J}\left(\sigma_1,\sigma_2\right)=\left\{
				\begin{split}
					& L_{\sigma_1} \text{ if Player } \mathbf{ P_1} \text{ stops the game  before the expiry time}\tau^T;\\
					& U_{\sigma_2} \text{ if Player } \mathbf{ P_2} \text{ stops the game first before time }\tau^T; \\
					&Q_{\sigma_1} \text{ If both players stops the game simultaneously }\\
					&\qquad \text{ before the expiry time }\tau^T;\\
					& \xi_1 \text{ If a default didn't occur, and neither of them }\\
					&\qquad \text{ have stopped the game until } T;\\
					& \xi_2 \text{ if the default occurs before } T \text{ and neither of them }\\
					& \qquad \text{ have stopped the game.} 
				\end{split}
				\right.
			\end{equation*}
		\end{itemize}
	\end{remark}
	The associated criterium of the above Dynkin game is given from the standpoint of time zero by $\mathcal{E}^f_{0,\sigma_1 \wedge \sigma_2}(\mathcal{J}\left(\sigma_1,\sigma_2\right))$, the $f$-evaluation of the payoff $\mathcal{J}\left(\sigma_1,\sigma_2\right)$.

	
	The nonlinear $f$-evaluation is defined in the enlarged filtration
	$\mathbb{G}$ under the probability measure $\mathbb{Q}$. More precisely,
	for any $\sigma_1,\sigma_2\in\mathfrak{T}_{0,\tau^T}(\mathbb{G})$, we set
	\[
	\mathcal{E}^f_{t,\sigma_1\wedge\sigma_2}
	\left(
	\mathcal{J}(\sigma_1,\sigma_2)
	\right)
	:=
	X^{\sigma_1,\sigma_2}_t,
	\qquad
	0\leq t\leq \sigma_1\wedge\sigma_2,
	\]
	where
	\[
	\left(
	X^{\sigma_1,\sigma_2},
	\pi^{\sigma_1,\sigma_2},
	\varTheta^{\sigma_1,\sigma_2}
	\right)
	\]
	is the solution, in the $(\mathbb{G},\mathbb{Q})$-framework, of the BSDE
	with driver $f$, terminal time $\sigma_1\wedge\sigma_2$, and terminal
	condition $\mathcal{J}(\sigma_1,\sigma_2)$. That is, the BSDE is considered on the filtered probability space
	$(\Omega,\mathcal{F},\mathbb{G},\mathbb{Q})$ and we have
	\begin{equation*}
		\left\{
		\begin{split}
			-dX^{\sigma_1,\sigma_2}_t&=-f(t,X^{\sigma_1,\sigma_2}_t,\pi^{\sigma_1,\sigma_2}_t)dt-\pi^{\sigma_1,\sigma_2}_t dB_t-d\varTheta^{\sigma_1,\sigma_2}_t,~t \in [0,\sigma_1 \wedge \sigma_2];\\
			X^{\sigma_1,\sigma_2}_{\sigma_1 \wedge \sigma_2}&=\mathcal{J}\left(\sigma_1,\sigma_2\right).
		\end{split}
		\right.
	\end{equation*}
	Player $\mathbf{P_1}$ seeks to maximize the criterion by selecting an optimal stopping time $\sigma_1 \in \mathfrak{T}_{0,\tau^T}(\mathbb{G})$ to end the game. On the other hand, Player $\mathbf{P_2}$ aims to minimize the cost paid to Player $\mathbf{P_1}$ before or at default by selecting a stopping time $\sigma_2 \in \mathfrak{T}_{0,\tau^T}(\mathbb{G})$.

	\subsection{Game Valuation: A doubly reflected BSDEs approach}
	In the field of stochastic control, we can introduce a dynamic structure to the game by incorporating a starting time, denoted by $\theta \in \mathfrak{T}_{0,\tau^T}(\mathbb{G})$. In this context, both players' strategy sets are defined over the time interval $[\theta, \tau^T]$, i.e., $\mathfrak{T}_{\theta,\tau^T}(\mathbb{G})$. Player $\mathbf{P_1}$ aims to maximize the criterion $\mathcal{E}^f_{\theta,\sigma_1 \wedge \sigma_2}(\mathcal{J}(\sigma_1,\sigma_2))$ by strategically selecting an optimal stopping time $\sigma_1 \in \mathfrak{T}_{\theta,\tau^T}(\mathbb{G})$. Conversely, Player $\mathbf{P_2}$ seeks to minimize the cost by exerting control over the game using the stopping time $\sigma_2 \in \mathfrak{T}_{\theta,\tau^T}(\mathbb{G})$. In this case and from Remark \ref{Remark on the construction of the BSDE}-(b), the $\mathcal{E}^f$-expectation of the game is defined by $\mathcal{E}^f_{\theta,\sigma_1 \wedge \sigma_2}(\mathcal{J}\left(\sigma_1,\sigma_2\right))=X_{\theta}^{\sigma_1,\sigma_2}$ with
	\begin{equation}
		\begin{split}
			X_{\theta}^{\sigma_1,\sigma_2}=\mathcal{J}\left(\sigma_1,\sigma_2\right)+\int_{\theta}^{\sigma_1\wedge\sigma_2}f(s,X^{\sigma_1,\sigma_2}_s,\pi^{\sigma_1,\sigma_2}_s)ds-\int_{\theta}^{\sigma_1\wedge\sigma_2}\pi^{\sigma_1,\sigma_2}_s dB_s- \int_{\theta}^{\sigma_1\wedge\sigma_2}d\varTheta^{\sigma_1,\sigma_2}_s,
		\end{split}
		\label{Ef expectation with f=0}
	\end{equation}
	\begin{remark}
		When $f=0$, we can leverage the martingale property of the stochastic integral featured in (\ref{Ef expectation with f=0}). Additionally, given that $X_{\theta}^{\sigma_1,\sigma_2} \in \mathcal{G}_{\theta}$, we can infer, after applying conditional expectation $\mathbb{E}_{\mathbb{Q}}\left[\cdot \mid \mathcal{G}_{\theta}\right]$ to (\ref{Ef expectation with f=0}), that $\mathcal{E}^0_{\theta,\sigma_1 \wedge \sigma_2}(\mathcal{J}\left(\sigma_1,\sigma_2\right))=\mathbb{E}_{\mathbb{Q}}\left[\mathcal{J}(\sigma_1,\sigma_2)\right]$.
	\end{remark}
	\begin{definition}
		\begin{itemize}
			\item Let us consider the game from time $\theta$ onwards, where $\theta$ runs through $\mathfrak{T}_{0,\tau^T}(\mathbb{G})$. The upper value, denoted as $\bar{V}(\theta)$, and the lower value, denoted as $\underline{V}(\theta)$, of the game at time $\theta$ are the random variables defined as follows:
			\begin{equation}
				\left\{
				\begin{split}
					&\bar{V}(\theta)=\essinf_{\sigma_{2} \in \mathfrak{T}_{\theta,T}(\mathbb{G})} \esssup_{\sigma_{1} \in \mathfrak{T}_{\theta,T}(\mathbb{G})} \mathcal{E}^f_{\theta,\sigma_{1} \wedge \sigma_{2}}\left(\mathcal{J}(\sigma_{1},\sigma_{2})\right),\\
					& \underline{V}(\theta)=\esssup_{\sigma_{1} \in \mathfrak{T}_{\theta,T}(\mathbb{G})} \essinf_{\sigma_{2} \in \mathfrak{T}_{\theta,T}(\mathbb{G})}\mathcal{E}^f_{\theta,\sigma_{1} \wedge \sigma_{2}}\left(\mathcal{J}(\sigma_{1},\sigma_{2})\right)
				\end{split}
				\right.
				\label{Game definition}
			\end{equation}
			\item We state that a value for the game at time $\theta$ exists if $\underline{V}(\theta)=\bar{V}(\theta)$, $\mathbb{Q}$-a.s.
			\item In the game, a pair $(\sigma_{1}^{\theta},\sigma_{2}^{\theta}) \in \left(  \mathfrak{T}_{\theta,\tau^T}(\mathbb{G})\right)^2$ is referred to be a saddle point at time $\theta$ if, for any $(\sigma_{1}^{},\sigma_{2}^{}) \in \left(  \mathfrak{T}_{\theta,\tau^T}(\mathbb{G})\right)^2$, we have
			$$
			\mathcal{E}^f_{\theta,\sigma_{1} \wedge \sigma_{2}^{\theta}}\left(\mathcal{J}\left(\sigma_{1},\sigma_{2}^{\theta}\right)\right) \leq 	\mathcal{E}^f_{\theta,\sigma_{1}^{\theta} \wedge \sigma_{2}^{\theta}}\left(\mathcal{J}\left(\sigma^{\theta}_{1},\sigma_{2}^{\theta}\right)\right) \leq 	\mathcal{E}^f_{\theta,\sigma_{1}^{\theta} \wedge \sigma_{2}^{}}\left(\mathcal{J}\left(\sigma^{\theta}_{1},\sigma_{2}^{}\right)\right),\quad \mathbb{Q}\text{-a.s.}
			$$
		\end{itemize}
	\end{definition}
	\begin{remark}
		Trivially, the inequality $\underline{V}(\theta) \leq \bar{V}(\theta)$, $\mathbb{Q}$-a.s., always holds true. Therefore, to establish that the game has a value at time $\theta$, we need to consider the inequality in the opposite direction.
	\end{remark}
	\begin{definition}
		\begin{itemize}
			\item We say that a process $Y \in \mathcal{S}^2_{ \tau^T}\left( \beta,\mathbb{G},\mathbb{Q}\right)$ is a strong $\mathcal{E}^f$-supermartingale (resp.  $\mathcal{E}^f$-submartingale), if $\mathcal{E}^f_{\sigma_{1},\sigma_{2}}(Y_{\sigma_{2}}) \leq Y_{\sigma_{1}}$ (resp. $\mathcal{E}^f_{\sigma_{1},\sigma_{2}}(Y_{\sigma_{2}}) \geq Y_{\sigma_{1}}$) $\mathbb{Q}$-a.s. on the set $\{\omega \in \Omega: \sigma_{1}(\omega)\leq \sigma_{2}(\omega)\}$ for all $(\sigma_{1},\sigma_{2}) \in \left(\mathfrak{T}_{0,\tau^T}(\mathbb{G})\right)^2$.
			\item We say that a process $Y \in \mathcal{S}^2_{ \tau^T}\left( \beta,\mathbb{G},\mathbb{Q}\right)$ is a strong $\mathcal{E}^f$-martingale if it is both a strong $\mathcal{E}^f$-supermartingale and a strong $\mathcal{E}^f$-submartingale.
		\end{itemize}
		\label{Definition of strong martingales}
	\end{definition}
	In the following discussion, we refer to $(Y_t)_{t \geq 0}$ as the first component of the solution to the DRBSDE (\ref{basic equation}) with parameters $(\xi,f,L,U)$  defined over the time interval $\llbracket0,\tau^T\rrbracket$ in accordance with Definition \ref{definition}.
	\begin{lemma}
		For each $\theta \in \mathfrak{T}_{0,\tau^T}(\mathbb{G})$, let
		\begin{equation}
			\sigma_{1}^{\theta}:=\inf\left\{t \geq \theta: Y_t = L_t  \right\} \wedge \tau^T \text{ and } \sigma_{2}^{\theta}:=\inf\left\{t \geq \theta: Y_t = U_t  \right\} \wedge \tau^T.
			\label{a Pair of stopping time without espsilon}
		\end{equation}
		Then, the process $(Y_t, \theta \leq t \leq \sigma_{1}^{\theta})$ is a strong $\mathcal{E}^f$-submartingale and the process $(Y_t, \theta \leq t \leq \sigma_{2}^{\theta})$ is a strong $\mathcal{E}^f$-supermartingale.
		\label{Lemma of strong submartingale}
	\end{lemma}
	\begin{proof}
		From the definition of the $\mathbb{G}$-stopping time $\sigma_{2}^{\theta}$, we have $Y_t <U_t$ for all $t \in [0,\sigma_{2}^{\theta}[$, $\mathbb{Q}$-a.s. Thus, from the Skorokhod condition, since $K^{+}$ has continuous paths, we deduce that $dK^{-}_t=0$ on $[\theta,\sigma_{2}^{\theta}]$. Therefore, the process $(Y_t)_{t \geq 0}$ satisfies 
		\begin{equation*}
			\begin{split}
				Y_{t}=Y_{\sigma_{2}^{\theta}}+\int_{t}^{\sigma_{2}^{\theta}}f(s,Y_s,Z_s)ds+\int_{t}^{\sigma_{2}^{\theta}} dK^{+}_s-\int_{t}^{\sigma_{2}^{\theta}} Z_s dB_s-\int_{t}^{\sigma_{2}^{\theta}} dM_s,~ t \in [\theta,\sigma_{2}^{\theta}],~\mathbb{Q}\text{-a.s.}
			\end{split}
		\end{equation*}
		It is worth mentioning that we can replace $t$ and $\sigma_{2}^{\theta}$ with any stopping times $\nu_1^{\theta}$ and $\nu_2^{\theta}$ in $\mathfrak{T}_{\theta,\sigma_{2}^{\theta}}(\mathbb{G})$ such that $\nu_1^{\theta} \leq \nu_2^{\theta}$, $\mathbb{Q}$-a.s. Then, since $K^{+}$ is an increasing process, we deduce, using the comparison theorem \ref{Theorem de comparison: General result}, that $Y_{\nu_1^{\theta}} \geq  \mathcal{E}^f_{\nu_1^{\theta},\nu_2^{\theta}}\left( Y_{\nu_2^{\theta}}\right) $. The  process $(Y_t, \theta \leq t \leq \sigma_{2}^{\theta})$ is thus a strong $\mathcal{E}^f$-supermartingale. Similarly, we can show that the process $(Y_t, \theta \leq t \leq \sigma_{1}^{\theta})$ is a strong $\mathcal{E}^f$-submartingale.
	\end{proof}
	\begin{corollary}
		The process $(Y_t, \theta \leq t \leq \sigma_{1}^{\theta} \wedge \sigma_{2}^{\theta})$, where $\sigma_{1}^{\theta}$ and $\sigma_{2}^{\theta}$ are given by (\ref{a Pair of stopping time without espsilon}), is a strong $\mathcal{E}^f$-martingale. 
		\label{Corolarry strong martingale}
	\end{corollary}

	We now apply the preceding DRBSDE results to the generalized
	$\mathcal{E}^f$-Dynkin game with default. The following theorem
	identifies the value process of the game with the state process of the
	associated DRBSDE and provides the corresponding saddle-point
	characterization.
	\begin{theorem}
		Consider a DRBSDE associated with data $(\xi,f,L,U)$, such that:
		\begin{itemize}
			\item 
			$\xi\in\mathcal{F}_{\tau^T}$ and
			$\xi\in
			\mathcal{L}^2_{\tau^T}
			\left(\beta,\mathbb{G},\mathbb{Q}\right)$.
			
			\item The driver $f$ satisfies $(\mathbf{H_2})$ and $\left( \frac{f(t,0,0)}{\alpha_{t}}\right)_{t \leq T}\in \mathcal{H}^2_{T}(\beta,\mathbb{F},\mathbb{P}^{})$.
			\item The obstacles: $(L,U) \in \left( \mathcal{S}^2_T(2\beta,\mathbb{F},\mathbb{P})\right)^2$, $L_t <U_t$ for all $t \in [0,\tau^T[$, $L_{t-} <U_{t-}$ for all $t \in [0,\tau^T]$ and $L_{\tau^T} \leq \xi \leq U_{\tau^T}$, $\mathbb{Q}$-a.s.
		\end{itemize}
		Let $(Y_t,Z_t,K^{+}_t,K^{-}_t,M_t)_{t \geq 0}$ denote the unique solution of the DRBSDE (\ref{basic equation}). Then, there exists a value function for the $\mathcal{E}^f$-Dynkin game (\ref{Game definition}) and, for each stopping time $\theta \in \mathfrak{T}_{0,\tau^T}$, we have
		$
		\underline{V}(\theta)=Y_{\theta} =  \bar{V}(\theta)
		, ~\mathbb{Q}\text{-a.s.}
		$ 
		Moreover, the pair $(\sigma_{1}^{\theta},\sigma_{2}^{\theta}) \in \left(  \mathfrak{T}_{\theta,\tau^T}(\mathbb{G})\right)^2$ defined by (\ref{a Pair of stopping time without espsilon}) is a saddle point for the game (\ref{Game definition}) at time $\theta$. 
		\label{Main Theorem for Dynkin}
	\end{theorem}
	\begin{proof}
		Let $\theta \in \mathfrak{T}_{0,\tau^T}(\mathbb{G})$. From Lemma \ref{Lemma of strong submartingale}, the process $(Y_t, \theta \leq t \leq \sigma_{1}^{\theta})$ is a strong $\mathcal{E}^f$-submartingale. Thus, from Definition \ref{Definition of strong martingales}, we have
		\begin{equation}
			Y_{\theta} \leq \mathcal{E}^f_{\theta,\sigma_{1}^{\theta} \wedge \sigma_{2}}\left(Y_{\sigma_{1}^{\theta} \wedge \sigma_{2}}\right) \text{ for any } \sigma_2 \in \mathfrak{T}_{\theta,\tau^T}(\mathbb{G}),~ \mathbb{Q}\text{-a.s.}
			\label{f conditional expecta 1}
		\end{equation}
		Obviously, 
		\begin{equation}
			\begin{split}
				Y_{\sigma_{1}^{\theta} \wedge \sigma_{2}}&=Y_{\sigma_{1}^{\theta}} \mathds{1}_{\{\sigma_{1}^{\theta} < \sigma_{2}\}} +Y_{\sigma_{2}} \mathds{1}_{\{\sigma_{2} < \sigma_{1}^{\theta}\}}+Y_{\sigma_{1}^{\theta}} \mathds{1}_{\{\sigma_{1}^{\theta} = \sigma_{2}<\tau^T\}}+Y_{ \sigma_{2}}\mathds{1}_{\{\sigma_{1}^{\theta}=\sigma_{2}=\tau^T\}}.
			\end{split}
			\label{formula that needs explication}
		\end{equation}
		On the event
		$\{\sigma_{1}^{\theta}=\sigma_{2}=\tau^T\}$,
		the terminal condition of the DRBSDE yields
		\[
		Y_{\sigma_{2}}
		=
		Y_{\tau^T}
		=
		\xi.
		\]
		Consequently,
		\begin{equation}
			Y_{\sigma_{2}}
			\mathds{1}_{\{\sigma_{1}^{\theta}
				=\sigma_{2}=\tau^T\}}
			=
			\xi
			\mathds{1}_{\{\sigma_{1}^{\theta}
				=\sigma_{2}=\tau^T\}}.
			\label{terminal condition that needs explication}
		\end{equation}
		On the one hand, using the right continuity of the processes $(Y_t)_{t \geq 0}$ and $(L_t)_{t \leq T}$, we can deduce that $Y_{\sigma_{1}^{\theta}}\mathds{1}_{\{\sigma_{1}^{\theta}<\tau^T\}}=L_{\sigma_{1}^{\theta}} \mathds{1}_{\{\sigma_{1}^{\theta}<\tau^T\}}\leq Q_{\sigma_{1}^{\theta}} \mathds{1}_{\{\sigma_{1}^{\theta}<\tau^T\}}$  $\mathbb{Q}$-a.s. By plugging this with (\ref{terminal condition that needs explication}) into (\ref{formula that needs explication}), we derive
		\begin{equation*}
			\begin{split}
				Y_{\sigma_{1}^{\theta} \wedge \sigma_{2}}&\leq L_{\sigma_{1}^{\theta}} \mathds{1}_{\{\sigma_{1}^{\theta} < \sigma_{2}\}} +U_{\sigma_{2}} \mathds{1}_{\{\sigma_{2} < \sigma_{1}^{\theta}\}}+Q_{\sigma_{1}^{\theta}} \mathds{1}_{\{\sigma_{1}^{\theta} = \sigma_{2}<\tau^T\}}+\xi\mathds{1}_{\{\sigma_{1}^{\theta}=\sigma_{2}=\tau^T\}}\\
				&=\mathcal{J}\left(\sigma_{1}^{\theta},\sigma_{2}\right).
			\end{split}
		\end{equation*}
		Then, by coming back to (\ref{f conditional expecta 1}), we have
		\begin{equation}
			Y_{\theta} \leq \mathcal{E}^f_{\theta,\sigma_{1}^{\theta} \wedge \sigma_{2}}\left(\mathcal{J}\left(\sigma_{1}^{\theta},\sigma_{2}\right)\right) \text{ for any } \sigma_2 \in \mathfrak{T}_{\theta,\tau^T}(\mathbb{G})~ \mathbb{Q}\text{-a.s.}
			\label{Saddle point game 1}
		\end{equation}
		Therefore,
		\begin{equation}
			\begin{split}
				Y_{\theta}  \leq \esssup_{\sigma_1 \in \mathfrak{T}_{\theta,\tau^T}(\mathbb{G})} \essinf_{\sigma_2 \in \mathfrak{T}_{\theta,\tau^T}(\mathbb{G})} \mathcal{E}^f_{\theta,\sigma_{1} \wedge \sigma_{2}}\left(\mathcal{J}\left(\sigma_{1},\sigma_{2}\right)\right)=\underline{V}(\theta),       \quad \mathbb{Q}\text{-a.s.}
			\end{split}
			\label{Value inf proof}
		\end{equation}
		Thanks again to Lemma \ref{Lemma of strong submartingale}, which allows us to derive that the process $(Y_t, \theta \leq t \leq \sigma_{2}^{\theta})$ is a strong $\mathcal{E}^f$-supermartingale. Thus, 
		\begin{equation}
			Y_{\theta} \geq \mathcal{E}^f_{\theta,\sigma_{1} \wedge \sigma_{2}^{\theta}}\left(Y_{\sigma_{1}^{} \wedge \sigma_{2}^{\theta}}\right) \text{ for any } \sigma_1 \in \mathfrak{T}_{\theta,\tau^T}(\mathbb{G})~ \mathbb{Q}\text{-a.s.}
			\label{f conditional expecta 2}
		\end{equation}
		Then, following similar arguments as the ones used to obtain inequality (\ref{Value inf proof}), we may derive
		\begin{equation}
			\begin{split}
				Y_{\theta}  \geq \essinf_{\sigma_2 \in \mathfrak{T}_{\theta,\tau^T}(\mathbb{G})} \essinf_{\sigma_1 \in \mathfrak{T}_{\theta,\tau^T}(\mathbb{G})} \mathcal{E}^f_{\theta,\sigma_{1} \wedge \sigma_{2}}\left(\mathcal{J}\left(\sigma_{1},\sigma_{2}\right)\right)=\bar{V}(\theta),      \quad \mathbb{Q}\text{-a.s.}
			\end{split}
			\label{Value sup proof}
		\end{equation}
		By combining, (\ref{Value inf proof}) and (\ref{Value sup proof}), we conclude that 
		\begin{equation*}
			\begin{split}
				\bar{V}(\theta) \leq Y_{\theta} \leq  \underline{V}(\theta) 
				,\quad \mathbb{Q}\text{-a.s.}  
			\end{split}
		\end{equation*}
		Thus, $\underline{V}(\theta) = Y_{\theta} =  \bar{V}(\theta)$ $\mathbb{Q}$-a.s. Which completes the proof of the first part. \\	
		The fact that the pair (\ref{a Pair of stopping time without espsilon}) is a saddle point for the game (\ref{Game definition}) at time $\theta$ can be seen directly by combining inequalities (\ref{Saddle point game 1}), (\ref{f conditional expecta 2}) and Corollary \ref{Corolarry strong martingale}, which yields to
		$$
		\mathcal{E}^f_{\theta,\sigma_{1} \wedge \sigma_{2}^{\theta}}\left(\mathcal{J}\left(\sigma_{1}^{},\sigma_{2}^{\theta}\right)\right)\leq Y_{\theta}=\mathcal{E}^f_{\theta,\sigma_{1}^{\theta} \wedge \sigma_{2}^{\theta}}\left(\mathcal{J}(\sigma_{1}^{\theta},\sigma_{2}^{\theta})\right) \leq \mathcal{E}^f_{\theta,\sigma_{1}^{\theta} \wedge \sigma_{2}}\left(\mathcal{J}\left(\sigma_{1}^{\theta},\sigma_{2}\right)\right),~ \mathbb{Q}\text{-a.s.,}
		$$
		for any  $\sigma_1, \sigma_2 \in \mathfrak{T}_{\theta,\tau^T}(\mathbb{G})$. This concludes the proof of the second part and, consequently, the proof of Theorem \ref{Main Theorem for Dynkin}.
	\end{proof}
	%
	%
	\section*{Funding}
	No funding was received for this paper.
	\section*{Disclosure statement}
	No potential conflict of interest was reported by the authors.
	
	

	\appendix
	\section{Martingale representation property}
	\label{Appendix A}
	In this section, we aims to establish a predictable representation property for all local martingales $M \in \mathcal{M}_{loc}\left(\mathbb{G},\mathbb{Q}\right)$ defined on the random time interval $\llbracket 0, T \wedge \tau\rrbracket$. This representation characterizes these martingales as stochastic integrals with respect to $B^{T \wedge\tau}$ and an additional orthogonal local martingale $M \in \mathcal{M}_{loc}\left(\mathbb{G},\mathbb{Q}\right)$ stopped at $T \wedge \tau$ (i.e. $M=M^{T \wedge \tau}$). To achieve this, we will first present some auxiliary results.
	\begin{theorem}
		Under \textbf{[P]}, for any local martingale $M \in \mathcal{M}_{loc} \left(\mathbb{F},\mathbb{P}\right)$, the process
		\begin{equation}
			\mathcal{T}\left(M\right)_t:=M^{\tau}-\int_{0}^{ t}\tilde{G}^{-1}_s \mathds{1}_{\{s \leq \tau\}} d \left[M,m\right]_s,\quad t \geq 0
			\label{operator T}
		\end{equation}
		is a local martingale in $\mathcal{M}_{loc} \left(\mathbb{G},\mathbb{P}\right)$.
		\label{Theorem of the operator}
	\end{theorem}
	\begin{proof}
		See \cite[Theorem 3, pp. 195]{aksamit2015optional}.
	\end{proof}
	Now, we prove the result of Lemma \ref{Lemma of the Brownian motion}.
	\begin{proof}[Proof of Lemma \ref{Lemma of the Brownian motion} ]
		Let $(M_t)_{t \geq 0}$ be an arbitrary element in $ \mathcal{M}_{loc}\left(\mathbb{F},\mathbb{P}\right)$. By utilizing Remark 2.5 in  \cite{choulli2022explicit}, we deduce that $\dfrac{\mathcal{E}_{s-}\left(\Gamma^{\mathcal{T}(m)}\right)}{\mathcal{E}_{s}\left(\Gamma^{\mathcal{T}(m)}\right)}=\dfrac{\tilde{G}_s}{G_{s-}}$ and 
		\begin{equation}
			\begin{split}
				\left[\Psi,\mathcal{T}(M)\right]_t&=\left[\mathcal{E}^{-1},\mathcal{T}(M)\right]_t
				\\
				&=\dfrac{G_{t-}}{\tilde{G}_t}\left[\mathcal{E}^{-1},M\right]^{\tau}_t=\dfrac{G_{t-}}{\tilde{G}_t}\left[\mathcal{E}\left(\Gamma^{\mathcal{T}(m)}\right),M\right]_t\\
				&=-\int_{0}^{t}\dfrac{G_{s-}}{\tilde{G}^2_s}\mathcal{E}_{s-}\left(\Gamma^{\mathcal{T}(m)}\right)d\left[m,M\right]^{\tau}_s=-\int_{0}^{t}\dfrac{\Psi_s}{\tilde{G}_s}\mathds{1}_{\{s \leq \tau\}}d\left[m,M\right]_s.
			\end{split}
			\label{explicit calculation}
		\end{equation}
		Now, from Theorem \ref{Theorem of the operator} and the Girsanov-Meyer theorem, we can deduce that the stopped process
		$$
		\mathcal{T}\left(M\right)+\int_{0}^{\cdot} \dfrac{1}{\Psi_s} d\left[\Psi, \mathcal{T}\left(M\right)\right]_s=M^{\tau}
		$$
		is a local martingale in $\mathcal{M}_{loc} \left(\mathbb{G},\mathbb{Q}\right)$. Therefore, since $B \in \mathcal{M}\left(\mathbb{F},\mathbb{P}\right)$ then $B^{\tau} \in \mathcal{M}\left(\mathbb{G},\mathbb{Q}^{}\right)$ and $B^{\tau \wedge T}$ is a $\left(\mathbb{G},\mathbb{Q}^{}\right)$-Brownian motion.
	\end{proof}
	Subsequently, we point out that, under condition \textbf{[P]}, for the integrable $\mathcal{G}_{T \wedge \tau}$-measurable random variable $\Psi_{T}$, we can consider  $\Psi_{\sigma}=\mathbb{E}\left[\Psi_{T}\mid \mathcal{G}_{\sigma}\right]$ for all $\sigma \in \mathfrak{T}_{0,T}(\mathbb{G})$. Therefore the process $\left(\mathbb{E}\left[\Psi_{T} \mid \mathcal{F}_t \right]\right)_{t \leq T}$ belongs to $\mathcal{M}\left(\mathbb{F},\mathbb{P}\right)$, and we can consider it's RCLL version denoted by $(\rho_t)_{t \leq T}$. 
	\begin{proposition}
		Define the probability measure $\mathbb{Q}^{\ast}$ on
		$(\Omega,\mathcal{F}_T)$ by
		\[
		\frac{d\mathbb{Q}^{\ast}}{d\mathbb{P}}
		\Big|_{\mathcal{F}_T}
		=
		\rho_T.
		\]
		Then, for every $t\in[0,T]$,
		\[
		\frac{d\mathbb{Q}^{\ast}}{d\mathbb{P}}
		\Big|_{\mathcal{F}_t}
		=
		\rho_t.
		\]
		Moreover, the restrictions of $\mathbb{Q}$ and $\mathbb{Q}^{\ast}$
		to $\mathcal{F}_t$ coincide for every $t\in[0,T]$.
		Equivalently,
		\[
		\mathbb{Q}|_{\mathcal{F}_t}
		=
		\mathbb{Q}^{\ast}|_{\mathcal{F}_t},
		\qquad t\in[0,T].
		\]
		\label{equiality of measures}
	\end{proposition}
	\begin{remark}
		Proposition~\ref{equiality of measures} shows that the restriction of
		$\mathbb{Q}$ to $\mathbb{F}$ is absolutely continuous with respect to
		$\mathbb{P}$, with density process $\rho$. In particular, for every
		$t\in[0,T]$,
		\[
		\frac{d\mathbb{Q}}{d\mathbb{P}}
		\Big|_{\mathcal{F}_t}
		=
		\rho_t.
		\]
		If $\rho$ is strictly positive, then $\mathbb{Q}$ and $\mathbb{P}$ are
		locally equivalent on $\mathbb{F}$.
	\end{remark}
	Using the result from Proposition \ref{equiality of measures}, predictable version of the Girsanov theorem \cite[Theorem III.8.36]{bookProtter} and  Lemma 2.5 presented in \cite{jeanblanc2015martingale}, we can derive the following result:
	\begin{lemma}
		If the angel brackets $\left\langle \rho,B \right\rangle$ exists in $\mathbb{F}$ for the $\mathbb{Q}$ probability, i.e. $[\rho,B]^{p,\mathbb{F},\mathbb{Q}}=\left\langle \rho,B \right\rangle$ exists, then the process
		$$
		B^{[\rho]}_t=B_t-\int_{0}^{t}\dfrac{1}{\rho_{s-}}d\left\langle \rho,B \right\rangle_s,\quad t \in  [0,T]
		$$
		is an $(\mathbb{F},\mathbb{Q})$-Brownian motion and possesses the martingale representation property in the filtration $\mathbb{F}$ under the probability measure $\mathbb{Q}$.
		\label{lemma of new Brownian in F}
	\end{lemma} 

	We recall the following Galtchouk--Kunita--Watanabe decomposition in the
	$(\mathbb{G},\mathbb{Q})$-framework.
	
	\begin{theorem}
		Let
		\[
		\xi\in L^2(\Omega,\mathcal{G}_{T\wedge\tau},\mathbb{Q}).
		\]
		Then there exist a $\mathbb{G}$-predictable process
		$Z$ and a square-integrable $(\mathbb{G},\mathbb{Q})$-martingale
		$M$, strongly orthogonal to $B^\tau$, such that
		\[
		Z
		=
		Z\mathds{1}_{\llbracket 0,T\wedge\tau\rrbracket},
		\qquad
		\mathbb{E}_{\mathbb{Q}}
		\left[
		\int_0^{T\wedge\tau}|Z_s|^2\,ds
		\right]
		<+\infty,
		\]
		and
		\[
		M_0=0,
		\qquad
		M=M^{T\wedge\tau},
		\]
		with
		\[
		\mathbb{E}_{\mathbb{Q}}
		\left[
		\xi\mid\mathcal{G}_t
		\right]
		=
		\mathbb{E}_{\mathbb{Q}}
		\left[
		\xi\mid\mathcal{G}_0
		\right]
		+
		\int_0^{t\wedge T\wedge\tau} Z_s\,dB^\tau_s
		+
		M_{t\wedge T\wedge\tau},
		\qquad
		t\in[0,T].
		\]
		Moreover,
		\[
		\left\langle M,B^\tau\right\rangle_t=0,
		\qquad
		t\in[0,T].
		\]
		\label{Theorem of representation in G,Q}
	\end{theorem}

	\begin{proof}
		The result follows from the classical
		Galtchouk--Kunita--Watanabe decomposition \cite{AnselStricker1993,ceci2014gkw,ceci2017follmer,ChoulliStricker1996} applied, in the
		$(\mathbb{G},\mathbb{Q})$-framework, to the square-integrable martingale
		\[
		\left(
		\mathbb{E}_{\mathbb{Q}}
		\left[
		\xi\mid\mathcal{G}_t
		\right]
		\right)_{t\in[0,T]}
		\]
		with respect to the martingale $B^\tau$.
	\end{proof}
	
	\begin{remark}
		The representation in Theorem~\ref{Theorem of representation in G,Q}
		may also be derived from the martingale representation results established
		in \cite{jeanblanc2015martingale}, together with
		\cite[Lemma 2.9, p.~34]{aksamit2017enlargement},
		Proposition~\ref{equiality of measures},
		\cite[Lemma 2.2]{jeanblanc2015martingale},
		and Lemmas~\ref{lemma of new Brownian in F}
		and~\ref{Lemma of the Brownian motion}.
		Thus, Theorem~\ref{Theorem of representation in G,Q} should be viewed
		as an application of known martingale representation results to the
		present $(\mathbb{G},\mathbb{Q})$-framework.
	\end{remark}

	%
	%
	\begin{remark}
			Note that the assumption of boundedness of the $\mathcal{G}_{T \wedge\tau}$-measurable random variable $\xi$ in the representation Theorem \ref{Theorem of representation in G,Q} is equivalent to the property that all
			$(\mathbb{G},\mathbb{Q})$-local martingales defined on $\llbracket 0, T \wedge \tau \rrbracket$ can be written as stochastic integrals with respect to $B^{\tau}$ and an orthogonal martingale $M$ on $[0,T]$ (see
			Theorem 13.4 in \cite{bookHe}).	Additionally, it is important to remark that, since we are dealing with default events and a progressive enlargement of the Brownian filtration, we cannot directly apply the representation property outlined in \cite[Lemma III.4.24, pp. 185]{bookJacod}. Instead, Theorem \ref{Theorem of representation in G,Q} provides a key formula that explicitly describes each component of the decomposition of processes within $\mathcal{M}_{loc}(\mathbb{G},\mathbb{Q})$ in terms of the default time $\tau$. 
		\label{Remark explaining how a probability measure is taking}
	\end{remark}
	\subsection*{A look at $\mathbb{G}$-predictable processes integrability}
	In this part, we aim to provide a connection between the integrability of predictable processes within the filtration $\mathbb{G}$ in relation to Brownian motion, especially when transitioning from the probability measure $\mathbb{P}$ to $\mathbb{Q}$. To this end, we first introduce the following lemma:
	\begin{lemma}
		Let $\left\{\sigma_n\right\}_{n \in \mathbb{N}}$ be an increasing sequence of stopping times in $\mathfrak{T}(\mathbb{G})$. Then, $\sigma_n \nearrow +\infty$ $\mathbb{P}$-a.s. if and only if $\sigma_n \nearrow +\infty$ $\mathbb{Q}$-a.s.
		\label{stopping times and measure changing}
	\end{lemma}
	\begin{remark}
		\begin{itemize}
			\item[(a)] Let $(Z_{t})_{t \geq 0}$ be a $\mathbb{G}$-predictable process such that $Z=Z\mathds{1}_{\llbracket 0,\tau \rrbracket}$ and $Z \in \mathcal{I}(\mathbb{G},\mathbb{P},B^{\tau})$, then from \eqref{operator T}, we deduce that $\int_{0}^{\cdot} Z_s dB_s$ is a $(\mathbb{G},\mathbb{P})$-semimartingale. Then, applying Lemma 2.2 in \cite{jeanblanc2015martingale}, we deduce that $$\mathcal{I}(\mathbb{G},\mathbb{P},B^{\tau})=\mathcal{I}(\mathbb{G},\mathbb{P},\mathcal{T}(B))\cap \mathcal{I}(\mathbb{G},\mathbb{P},V^{B}),$$
			where $V^B$ is the $\mathbb{G}$-adapted co,continuous processes defined by $V^B:=\int_{0}^{\cdot} \tilde{G}^{-1}_s \mathds{1}_{\{s \leq \tau\}} d[B,m]_s$ and $B^{\tau}=\mathcal{T}(M)+V^{B}$ is a special $(\mathbb{G},\mathbb{P})$-semimartingale.		
			\item[(b)] Note that since for $\mathbb{P}$-almost all $\omega \in \Omega$, $t \mapsto B_t(\omega)$ is a continuous function. Therefor, using the absolute continuity of $\mathbb{Q}$ with respect to $\mathbb{P}$, we get $\mathbb{Q}\big(\omega: t \mapsto B_t(\omega) \text{ is not continuous}\big)=0$. Hence, the process $(B^{\tau}_t)_{t \geq 0}$ has continuous paths in $\mathbb{R}_{+}$ for $\mathbb{Q}$-almost all $\omega \in \Omega$ and we get
			$$
			[B^{\tau},B^{\tau}]=\left[\mathcal{T}(B),\mathcal{T}(B)\right]+2\left[\mathcal{T}(B),V^B\right]+\left[\mathcal{T}(B),V^B\right]=\left[\mathcal{T}(B),\mathcal{T}(B)\right],
			$$
			since
			$$
			\left[\mathcal{T}(B),V^B\right]=\sum \tilde{G}^{-1}\mathds{1}_{\rrbracket0,\tau\rrbracket} \Delta \mathcal{T}(B)\Delta m \Delta B=0,
			$$
			and 
			$$
			\left[V^B,V^B\right]=\sum \tilde{G}^{-2}\mathds{1}_{\rrbracket0,\tau\rrbracket} \left( \Delta m \Delta B\right)^2 =0,
			$$
			where the quadratic variation is computed in the filtration $\mathbb{G}$ with respect to the probability measure $\mathbb{P}$. 		
			\item[(c)] Since  $\mathds{1}_{\llbracket 0,\tau \rrbracket}$ is a $\mathbb{G}$-predictable bounded process and using Theorem I.12 in \cite{bookProtter} and Theorem 1.11-(d) in \cite{book:63874}, we can write 
			$
			[B,B]^{\tau}_t=\int_{0}^{t} \mathds{1}_{\{s \leq \tau\}}d[B,B]_s
			$, $t \geq 0$, under the two equivalent probabilities $\mathbb{P}$ and $\mathbb{Q}$. Moreover, using Theorem  I.14 in \cite{bookProtter}, the process  $\int_{0}^{\cdot} \mathds{1}_{\{s \leq \tau\}}d[B,B]_s$ calculated under $\mathbb{P}$ is indistinguishable from $\int_{0}^{\cdot} \mathds{1}_{\{s \leq \tau\}}ds$ with respect to $\mathbb{Q}$. In other word, the stochastic integral $\int_{0}^{\cdot} \mathds{1}_{\{s \leq \tau\}}d[B,B]_s$ defined in the two senses gives the same $\mathbb{G}$-adapted continuous process with finite variation (see also Lemma 2.1 in  \cite{jeanblanc2015martingale}).
		\end{itemize}
		\label{Remarks for the brownian and quadratic variation}
	\end{remark}
	\begin{proposition}
		Let $(Z_t)_{t \geq 0}$ be a $\mathbb{G}$-predictable process such that $Z=Z\mathds{1}_{\llbracket0,\tau\rrbracket}$. Then, whenever  $Z \in  \mathcal{I}\left(\mathbb{G},\mathbb{P},\mathcal{T}(B)\right)$, we have $Z \in \mathcal{I}\left(\mathbb{G},\mathbb{Q},B^{\tau}\right)$.
		\label{integrability of Z}
	\end{proposition}
	\begin{proof}
		Let $Z \in \mathcal{I}\left(\mathbb{G},\mathbb{P},B^{\tau}\right)$ and let $\left\{\sigma_n\right\}_{n \in \mathbb{N}}$ be an increasing sequence of stopping times in $\mathfrak{T}\left(\mathbb{G}\right)$ such that $\sigma_n \nearrow +\infty$ $\mathbb{P}$-a.s., $\mathbb{E} \int_{0}^{\sigma_n} \left|Z_s \right|^2 \mathds{1}_{\{s \leq \tau\}} d\left[\mathcal{T}(B),\mathcal{T}(B)\right]^{\tau}_s<\infty $, for each $n \in \mathbb{N}$.\\	
		Consider the following sequence of stopping times :
		$$
		\theta_n:=\inf\{t \geq 0:  \Psi_t >n\} \wedge \sigma_n.
		$$
		Then $\left\{\theta\right\}_{n \in \mathbb{N}}$ belongs to $\mathfrak{T}\left(\mathbb{G}\right)$ such that $\theta_n \nearrow +\infty$ $\mathbb{P}$-a.s. as $\Psi$ is a $\mathbb{G}$-adapted RCLL process. Moreover, from Lemma \ref{stopping times and measure changing} and Remark \ref{Remarks for the brownian and quadratic variation}-(b)-(c), we have
		\begin{equation*}
			\begin{split}
				\mathbb{E}^{\ast}\int_{0}^{\theta_n} \left|Z_s \right|^2 \mathds{1}_{\{s \leq \tau\}} ds&=\mathbb{E}^{}\int_{0}^{\theta_n} \Psi_s\left|Z_s \right|^2  d\left[B^{\tau},B^{\tau}\right]_s\\
				&=\mathbb{E}^{}\int_{0}^{\theta_n} \Psi_s\left|Z_s \right|^2 \mathds{1}_{\{s \leq \tau\}} ds\leq n \mathbb{E}^{}\int_{0}^{\sigma_n} \left|Z_s \right|^2 d\left[\mathcal{T}(B),\mathcal{T}(B)\right]^{\tau}_s<\infty,
			\end{split}
		\end{equation*}
		for every $n \in \mathbb{N}$. Hence $Z \in \mathcal{I}\left(\mathbb{G},\mathbb{Q},B^{\tau}\right)$ and this completes the proof. 
	\end{proof}
	\begin{remark} For the reverse property of Proposition \ref{integrability of Z}, we require the condition that the quadratic variation $[m, m]$ of the BMO martingale $(m_t)_{t \geq 0}$ is locally integrable with respect to $\mathbb{Q}$. In this case, by following an argument similar to the fact that $\frac{d\mathbb{P}}{d\mathbb{Q}} = \left(\frac{d\mathbb{Q}}{d\mathbb{P}}\right)^{-1}$ and using an appropriate sequence of stopping times, we can establish the desired result for the $(\mathbb{G},\mathbb{P})$-martingale part and the $(\mathbb{G},\mathbb{P})$-local integrability for the finite variation part from Kunita-Watanabe inequality.
	\end{remark}
	%
	%
	%
	%
	\section{Auxiliary BSDE results}
	\label{Appendix C}
	
	For completeness, we collect in this section some standard results for
	BSDEs without reflection that are used as auxiliary tools in the proofs
	of the reflected and doubly reflected BSDE results developed in
	Section \ref{DRBSDE Problem section}. In particular, the existence, uniqueness, and comparison
	results below are not intended as new contributions; they correspond to
	the non-reflected counterparts of the reflected BSDE framework.
	
	We consider the following BSDE on the random horizon $T\wedge\tau$:
	\begin{equation}
		\left\{
		\begin{split}
			& \mathbb{Q}\text{-a.s. for all } t \in [0,T]\\
			&~ Y_{t \wedge \tau}=\xi+\int_{t \wedge \tau }^{T \wedge \tau} f(s,Y_s,Z_s)ds-\int_{t \wedge \tau }^{T \wedge \tau} Z_s d B_s-\int_{t \wedge \tau }^{T \wedge \tau}dM_s;\quad t \in [0,T].\\
			&~  Y_t=\xi,~Z_t=dM_t=0,\text{ on the set }\{t \geq T \wedge\tau\},~\mathbb{Q}\text{-a.s.},\\
		\end{split}
		\right.
		\label{basic classical BSDE}
	\end{equation}
	\begin{theorem}
		Assume that $\xi$ is $\mathcal{F}_{T\wedge\tau}$-measurable and
		\[
		\xi\in
		\mathcal{L}^2_{T\wedge\tau}
		\left(
		\beta,\mathbb{G},\mathbb{Q}
		\right).
		\]
		Assume also that $f$ satisfies $(\mathbf{H}_2)$ and
		\[
		\left(
		\frac{f(t,0,0)}{\alpha_t}
		\right)_{t\geq0}
		\in
		\mathcal{H}^2_T
		\left(
		\beta,\mathbb{F},\mathbb{P}
		\right).
		\]
		Then the BSDE~\eqref{basic classical BSDE} admits a unique solution
		\[
		(Y,Z,M)
		\in
		\mathfrak{U}^2_{T\wedge\tau}
		\left(
		\beta,\mathbb{G},\mathbb{Q}
		\right).
		\]
		\label{Existence and uniquess theorem BSDE}
	\end{theorem}
	\begin{remark}
		The preceding existence and uniqueness result is standard and is included
		only because it is used as an auxiliary result in the construction of the
		reflected and doubly reflected BSDE solutions in Section \ref{DRBSDE Problem section}.
	\end{remark}

	\begin{theorem}[Comparison theorem]
		et $(Y^i_t,Z^i_t,M^i_t)_{t \geq 0}$ be a solution to BSDE \eqref{basic classical BSDE} associated with parameters $(\xi^i,f^i)$ for $i=1, 2$. If $\xi_1 \leq \xi_2$ and $f^1(s,Y^1_s,Z^1_s)\leq f^2(s,Y^1_s,Z^1_s)$ for all $t \in [0,T\wedge \tau]$. Then $Y^1_t \leq Y^2_t$, $\forall t \geq 0$, $\mathbb{Q}$-almost surely.
		\label{Comparison theorem BSDE}
	\end{theorem}
	\begin{remark}
		By the comparison theorem, we deduce that $\mathcal{E}^f$ is a non-decreasing operator, i.e. for any $\sigma \in \mathfrak{T}_{0\tau^T}(\mathbb{G})$ and any $\xi_1,\xi_2 \in \mathcal{L}^2_{\sigma}\left(\beta,\mathbb{G},\mathbb{Q}^{}\right)$, 
		if $\xi_1 \geq \xi_2$, $\mathbb{Q}$-a.s. then $\mathcal{E}^f_{\cdot,\sigma}(\xi_1) \geq \mathcal{E}^f_{\cdot,\sigma}(\xi_2)$, $\mathbb{Q}$-a.s.
		\label{Monotonicity if the f-evaluation}
	\end{remark}
	\begin{theorem}[Special comparison theorem]
		Suppose that $f^1(s, y, z)$ and $f^2(s, y,z)$ are two drivers such that $f^1(s, \cdot, \cdot)$ and $f^2(s, \cdot, \cdot)$ are $\mathbb{F}$-progressively measurable and $f^2$ satisfies condition ($\mathbf{H}_2$), that $\xi_1$, $\xi_2$ two $\mathcal{G}_{T \wedge \tau}$-measurable random variables and that $K^1$ and $K^2$ are two RCLL, increasing processes.\\	
		If there exist a quadruplet $(Y^i_t,Z^i_t,M^i_t)_{t \geq 0}^{i=1;2}$ satisfying the equations :
		\begin{equation*}
			\left\{
			\begin{split}
				\text{(i)} &~ \mathbb{Q}\text{-a.s. for all } t \in [0,T]\\
				&~ Y^i_{t}=\xi^i+\int_{t }^{T} \mathds{1}_{\{s \leq \tau\}} f^i(s,Y^i_s,Z^i_s)ds+\left( K^{i}_{T}-K^{i}_{t}\right)-\int_{t }^{T } Z^i_s d B_s-\int_{t}^{T}dM^i_s.\\
				\text{(ii)}&~  Y^i_t=\xi^i,~Z^i_t=dK^{i}_t=dM^i_t=0,\text{ on the set }\{t \geq T \wedge\tau\},~\mathbb{Q}\text{-a.s.}
			\end{split}
			\right.
		\end{equation*}
		Moreover, if $f^1(s,Y^1_s,Z^1_s)-f^2(s,Y^1_s,Z^1_s)\leq 0$ for any $s \in [0,T \wedge \tau[$, $\xi_1 \leq \xi_2$  and  $K^2-K^1$ is increasing, then $Y^1_t \leq Y^2_t$, $\forall t \geq 0$, $\mathbb{Q}$-almost surely.
		\label{Theorem de comparison: General result}
	\end{theorem}
	%
	\section{Conditional expectations and integrable random variables on $\left(\mathbb{G}, \llbracket 0,\tau\llbracket,\mathbb{Q}\right) $}
	\label{Appendix B}
	\begin{lemma}
		\begin{itemize}
			\item[(a)] Let $\left( \mathfrak{X}_t\right)_{t \geq 0}$ be an RCLL $\mathbb{G}$-adapted process and  $\left( \mathfrak{Y}_t\right)_{t \geq 0}$ be an RCLL $\mathbb{G}$-semimartingale such that $\int_{0}^{\cdot}\mathfrak{X}_s d\mathfrak{Y}_s$ is well defined. Assume that $\int_{0}^{\cdot}\mathfrak{X}_s d\mathfrak{Y}_s$ is positive or $\sup_{0 \leq s \leq \cdot} \left( \int_{0}^{s}\mathfrak{X}_s d\mathfrak{Y}_s\right)^{+}$ is integrable. Then we have
			$$
			\mathbb{E}\left[\int_{0}^{\sigma}\mathfrak{X}_s d\mathfrak{Y}_s \mid \mathcal{G}_t \right]=\Psi^{-1}_t \mathbb{E}_{\mathbb{Q}}\left[\int_{0}^{\sigma}\mathcal{E}_s^{-1}\mathfrak{X}_s d\mathfrak{Y}_s \mid \mathcal{G}_t\right],\quad \forall \sigma \in \mathfrak{T}_{t,+\infty}(\mathbb{G})
			$$
			and
			$$
			\mathbb{E}\left[\mathfrak{X}_{\sigma}\right]=	\mathbb{E}_{\mathbb{Q}}\left[\mathfrak{X}_{\sigma}\Psi_{\sigma}^{-1}\right],\quad \forall \sigma \in \mathfrak{T}(\mathbb{G}).
			$$
			\item[(b)] For any integrable or positive process $\left( \mathfrak{X}_t\right)_{t \leq T}$, using the definition of the $\mathbb{P}$-conditional expectation and the equivalent measure $\mathbb{Q}$, we have for any $t \in [0,T]$ and each $\sigma \in \mathfrak{T}_{t,T}(\mathbb{F})$;
			$$
			\mathbb{E}_{\mathbb{Q}}\left[\mathfrak{X}_{\sigma} \mid \mathcal{F}_t\right]=\dfrac{1}{\rho_t}\mathbb{E}\left[\mathfrak{X}_{\sigma} \rho_{\sigma} \mid \mathcal{F}_t\right],~~\text{ with }~ \dfrac{d \mathbb{Q}}{d \mathbb{P}} \mid_{\mathcal{F}_t}=\rho_t.
			$$
			In particular, we have $\mathbb{Q}\left(\tau>t \mid \mathcal{F}_t\right)=G_t>0$.
		\end{itemize}
		\label{Lemma of integrability}
	\end{lemma}
	The two standard observations that follows has been used in the proof of Theorems \ref{Delicate thm} and \ref{Proposition of comparison}.
	\begin{remark}
		\begin{itemize}
			\item As $\Psi=\mathcal{E}\left(-\int_{0}^{\cdot} G^{-1}_{s-}d \mathcal{T}(m)_s\right)$ and $\Delta \mathcal{T}(m)=\frac{G_{-}}{\tilde{G}}\Delta m \mathds{1}_{\rrbracket 0,\tau \rrbracket}$, we derive that 
			$$1-G^{-1}_{-}\Delta \mathcal{T}(m)=\frac{G_{-}}{\tilde{G}}\mathds{1}_{\rrbracket 0, \tau\rrbracket}+\mathds{1}_{\rrbracket \tau,+\infty\llbracket}.$$ Then, under \textbf{[P]}, we derive that $G_{-}, \tilde{G}>0$, then $\Psi >0$. Additionally, by calculating $\Delta \Psi$, we get $ \Psi=\Psi_{-}\dfrac{G_{-}}{\tilde{G}}\mathds{1}_{\rrbracket 0,\tau\rrbracket}+\Psi_{-} \mathds{1}_{\rrbracket \tau,+\infty\llbracket}$, then under \textbf{[P]}, we derive $\Psi_{-}>0$.
			
			\item From (\ref{Definition of E and E tilde}), we have $\tilde{\mathcal{E}} >0$. Indeed, under assumption \textbf{[P]}, we have $1-\tilde{G}^{-1}\Delta D^{o,\mathbb{F},\mathbb{P}}=1-\tilde{G}^{-1}\left(\tilde{G}-G\right)=\tilde{G}^{-1}G>0$. Moreover, note that the left-limited process $\left(\tilde{\mathcal{E}}_{t-}\right)_{t \geq 0}$ with the convention $\tilde{\mathcal{E}}_{0-}=\tilde{\mathcal{E}}_{0}$ satisfies $\tilde{\mathcal{E}}_{-}=\frac{\tilde{G}}{G}\tilde{\mathcal{E}}$, from which we also derive that $\tilde{\mathcal{E}}_{-}>0$. 
		\end{itemize}
		\label{Remark of epsilon}
	\end{remark}
	\begin{remark}
		Note that the processes $\left(K^{n,\mathbb{F},+}_t\right)_{t \leq T}$ and $\left( K^{n,\mathbb{F},-}_t\right)_{t \leq T}$ defined in the proof of Theorem \ref{Proposition of comparison} are well-defined Stieltjes integrals.\\
		To see this, we first examine the properties satisfied by the processes $\mathsf{k}^{n,\mathbb{F},+}:=\left(n\mathbb{E}_{\mathbb{Q}} \left[ (L_t-Y^{n}_t)^{+}\mathds{1}_{\{t < \tau\}} \mid \mathcal{F}_t \right]\right)_{t \leq T}$ and $\mathsf{k}^{n,\mathbb{F},-}:=\left(n\mathbb{E}_{\mathbb{Q}} \left[ (L_t-Y^{n}_t)^{+}\mathds{1}_{\{s < \tau\}} \mid \mathcal{F}_t \right]\right)_{t \leq T}$. These will ensure that the Lebesgue integrals $\left(K^{n,\mathbb{F},+}_t\right)_{t \leq T}$ and $\left( K^{n,\mathbb{F},-}_t\right)_{t \leq T}$ are increasing continuous $\mathbb{F}$-adapted Stieltjes integrals.\\
		To address this, we utilize the classical $\mathbb{F}$-decomposition of $\mathbb{G}$-optional processes on the time interval $\llbracket 0,\tau\llbracket$. Since $Y^n$, $L$, and $U$ are $\mathbb{G}$-optional processes on $[0,T]$, and since condition \textbf{[P]} is satisfied, we can deduce from \cite[Lemma 4.4]{book:63874} the existence of unique $\mathbb{F}$-optional processes $\mathcal{K}^{n,\mathbb{F},+}$ and $\mathcal{K}^{n,\mathbb{F},-}$ such that: 
		$$
		(L_t-Y^{n}_t)^{+}\mathds{1}_{\{t < \tau\}}=\mathcal{K}^{n,\mathbb{F},+}_t\mathds{1}_{\{t < \tau\}},\text{ and } (Y^{n}_t-U_t)^{+}\mathds{1}_{\{t < \tau\}}=\mathcal{K}^{n,\mathbb{F},-}_t\mathds{1}_{\{t < \tau\}}
		$$	 
		Since $Y^n$, $L$, and $U$ are RCLL processes, we deduce that $\mathcal{K}^{n,\pm}$ are also RCLL processes. Moreover, from \cite[Lemma B.1]{AksamitChoulli2017}, we may choose $\mathcal{K}^{n,\mathbb{F},\pm}$ to be non-negative.\\
		Taking the conditional expectation (which is well defined due to the integrability property satisfied by $Y^n$, $L$ and $U$), we get
		$$
		\mathsf{k}^{n,\mathbb{F},+}_t=n\mathcal{K}^{n,+}_t\mathbb{E}_{\mathbb{Q}}\left[\mathds{1}_{\{t<\tau\}} \mid \mathcal{F}_t \right],\text{ and } \mathsf{k}^{n,\mathbb{F},-}_t=n\mathcal{K}^{n,\mathbb{F},-}_t\mathbb{E}_{\mathbb{Q}}\left[\mathds{1}_{\{t<\tau\}} \mid \mathcal{F}_t \right].
		$$
		As $\left( \mathbb{E}_{\mathbb{Q}}\left[\mathds{1}_{\{t<\tau\}} \mid \mathcal{F}_t \right]\right)_{t \leq T}$ is an RCLL $\mathbb{F}$-supermartingale (there exists a unique modification which is RCLL form \cite[Theorem I.9]{bookProtter}), we deduce that $\mathsf{k}^{n,\mathbb{F},\pm}$ is an RCLL $\mathbb{F}$-adapted process. Hence $\mathbb{F}$-optional. Then, using \cite[Proposition I.3.5]{bookJacod}, we deduce that $\left(K^{n,\mathbb{F},+}_t\right)_{t \leq T}$ and $\left( K^{n,\mathbb{F},+}_t\right)_{t \leq T}$ are well-defined, finite-valued Stieltjes integrals, which are increasing $\mathbb{F}$-adapted and continuous.	  	
		\label{RMQ of Well defined Steiltjes integrals}
	\end{remark}
	\begin{lemma}
		Let $F$ be an integrable $\mathcal{G}_s$-measurable random variable ($s \in [0,T]$) and $\mathcal{F}_{s,t}:=\sigma\big(B_u-B_s;s \leq u \leq t,~t \in [0,T]\big)$. Assume that $\tau$ and $\mathcal{F}_{T}$ are independent under $\mathbb{Q}$. Then, we have:
		$$
		\mathbb{E}_{\mathbb{Q}}\left[F \mathds{1}_{\{s<\tau\}} \mid \mathcal{F}_s \vee \mathcal{F}_{s,t}\right]=\mathbb{E}_{\mathbb{Q}}\left[F \mathds{1}_{\{s<\tau\}} \mid \mathcal{F}_s \right],~ \forall t \in [s,T].
		$$
		\label{Lemma of conditional expectation}
	\end{lemma}
	\begin{proof}
		The proof of this result is akin to the proof of Lemma 5.3 in \cite{AkdimDiakhabyOuknine} (and also Lemma 5 in \cite{el2009bsdes}) with some modifications, which we clarify for the reader's convenience. Now, when we consider $A \in \mathcal{F}_t$ and $C \in \mathcal{F}_{s,t}$, then
		\begin{equation}
			\begin{split}
				\mathbb{E}_{\mathbb{Q}}\left[\mathds{1}_A \mathds{1}_C \mathbb{E}\left[F \mid \mathcal{F}_s\right]\right]&=\mathbb{E}_{\mathbb{Q}}\left[\mathds{1}_A  \mathbb{E}\left[F \mid \mathcal{F}_s\right]\right] \mathbb{E}_{\mathbb{Q}}\left[ \mathds{1}_C\right]. 
			\end{split}
			\label{equality becomes}
		\end{equation}
		Since $\mathds{1}_C$ is independent of $\mathcal{F}_s$, we can proceed. Now, consider that the $\sigma$-algebra $\mathcal{G}_s$ is generated by random variables of the form $F_s h(s \wedge \tau)$, where $F_s$ is an $\mathcal{F}_s$-measurable random variable, and $h$ is a  deterministic Borel function on $\mathbb{R}_{+}$. Therefore, by a monotone class argument, it is enough to prove the property of Lemma \ref{Lemma of conditional expectation} for $\mathcal{G}_s$-measurable random variables of the form $F_s  h(s)\mathds{1}_{\{s<\tau \}}$, where $F_s$ is an integrable random variable. By using again the independence property of $\mathds{1}_A$ with respect to $\mathcal{F}_s$ together with the $\mathbb{F}$-optionality of the Azéma's supermartingale $G^{\mathbb{Q}}_s:=\mathbb{Q}\left(\tau >s \mid \mathcal{F}_s\right)$ , the equality \eqref{equality becomes} can be expressed as
		\begin{equation*}
			\begin{split}
				\mathbb{E}_{\mathbb{Q}}\left[\mathds{1}_A \mathds{1}_C \mathbb{E}\left[F \mid \mathcal{F}_s\right]\right]=\mathbb{E}_{\mathbb{Q}}\left[\mathds{1}_A \mathds{1}_CF_s h(s) G^{\mathbb{Q}}_s\right]&=\mathbb{E}_{\mathbb{Q}}\left[\mathds{1}_A F_s h(s)G^{\mathbb{Q}}_s\right]\mathbb{E}_{\mathbb{Q}}\left[ \mathds{1}_C\right]\\
				&=\mathbb{E}_{\mathbb{Q}}\left[\mathds{1}_A \mathbb{E}\left[F \mathds{1}_{\{s<\tau\}} \mid \mathcal{F}_s\right] \right]\mathbb{E}\left[ \mathds{1}_C\right]\\
				&=\mathbb{E}_{\mathbb{Q}}\left[\mathds{1}_A \mathds{1}_C F\mathds{1}_{\{s<\tau\}}\right],  
			\end{split}
		\end{equation*}
		where in the last equality, we have used the independence of $\tau$ and $\mathcal{F}_{T}$. This completes the proof.
	\end{proof}
	\begin{remark}
		In general, the assumption of $\tau$ and $\mathcal{F}_{T}$ being independent is satisfied, for example, in the particular case where $\tau$ is an initial time, subject to certain additional specific conditions. For further details, we direct readers to \cite[Proposition 2.3]{jeanblanc2009progressive} and the references therein. 
	\end{remark}
	\begin{remark} 
		\begin{itemize}
			\item[(a)] If $\tau$ and $\mathcal{F}_{T}$ are independent under $\mathbb{Q}$, Lemma \ref{Lemma of conditional expectation} enables us to express:
			$$
			\mathbb{E}_{\mathbb{Q}}\left[F  \mathds{1}_{\{s <\tau\}} \mid \mathcal{F}_t\right] = \mathbb{E}_{\mathbb{Q}}\left[F \mathds{1}_{\{s <\tau\}} \mid \mathcal{F}_s\right],~ \forall t \in [s,T]
			$$
			for any  bounded $\mathcal{G}_s$-measurable ($s \in [0,T]$) random variable.
			\item[(b)] Let $(Z_t)_{t \geq 0}$  a $\mathbb{G}$-predictable process such that $Z \in \mathcal{I}(\mathbb{G},\mathbb{Q},B)$ and $Z=Z \mathds{1}_{\llbracket 0,\tau \rrbracket}$, then using the continuity of the Brownian motion $(B^{\tau}_t)_{t \geq 0},$ we obtain
			$$
			\int_{0}^{t}Z_s dB_s=\int_{0}^{t}(1-D_{s-})Z_s dB_s=\int_{0}^{t} \mathds{1}_{\{s <\tau\}} Z_s dB_s,~ \forall t \in [0,T],~\mathbb{Q}\text{-a.s.}
			$$
		\end{itemize}
		\label{strssfull remark}
	\end{remark}
	We are now prepared to provide the proof for Lemma \ref{result of interchanging the conditional expectation}.
	\begin{proof}[Proof of Lemma \ref{result of interchanging the conditional expectation}]
		First, note that under condition \textbf{[P]}, the $(\mathbb{G},\mathbb{P})$-martingale density $\left(\varPsi_t\right)_{t \geq 0}$ can be expressed as $\varPsi_t=\mathbb{E}\left[\varPsi_{T} \mid \mathcal{G}_t\right]$, $t \in [0,T]$. Then, taking the conditional expectation in both sides, we obtain $\mathbb{E}\left[\varPsi_t \mid \mathcal{F}_t\right]=\rho_t$ for all $t \in [0,T]$. \\	
		Now, let $D_t$ denotes the derivative operator and $\mathbb{D}^{1,2}$ it's domain in $\mathbb{L}^2$, defined as the space of the random variables that are differentiable in the Malliavin sense, more precisely $\mathbb{D}^{1,2}$ is the closure of the class of smooth random variables with respect to a given norm under which $\mathbb{D}^{1,2}$ is an Hilbert space (see \cite{Malliavin, oksendal1997introduction} for more details). Thus, to show the result of Lemma \ref{result of interchanging the conditional expectation}, it suffices to prove that for $F \in \mathds{D}^{1,2}$ be bounded $\mathcal{F}_t$-measurable then $
		\mathbb{E}_{\mathbb{Q}}\left[F\int_{0}^{t} \left| Z_s\right|^2 ds  \right]<\infty
		$, for all $t \in [0,T]$. Therefore, using the duality relation (see \cite{Malliavin} for example) in the setting $\left(\mathbb{F},\mathbb{P}\right)$ and Proposition \ref{equiality of measures}, we get
		\begin{equation}
			\begin{split}
				\mathbb{E}_{\mathbb{Q}}\left[F\int_{0}^{t} Z_s dB_s  \right]&=\mathbb{E}\left[F\int_{0}^{t} \rho_{s} Z_s dB_s  \right]
				=\mathbb{E}\left[\int_{0}^{t}\left(D_s F\right) \rho_{s} Z_s ds  \right]\\
				&=\int_{0}^{t}\mathbb{E}\left[\left(D_s F\right) \rho_{s} Z_s \right] ds 
				=\int_{0}^{t}\mathbb{E}\left[ \mathbb{E}\left[\left\lbrace \left(D_s F\right) \rho_{s} Z_s\right\rbrace  \mid \mathcal{F}_s \right] \right]  ds.
			\end{split}
			\label{Comming back in Malliavin}
		\end{equation} 
		From Lemma \eqref{Lemma of conditional expectation} (see also Remark \ref{strssfull remark}-(a)) and the $\mathcal{F}_t$-measurability of $F$, we obtain $F=\mathbb{E}\left[F \mid \mathcal{F}_t\right]=\mathbb{E}\left[F \mid \mathcal{F}_s\right]$ on the set $\{s<\tau\}$, next applying Proposition 1.2.8 in \cite{Malliavin}, we have $D_s F =\mathbb{E}\left[D_s F \mid \mathcal{F}_s\right]$, hence $D_s F \in \mathcal{F}_s$ on the set $\{s<\tau\}$. Returning to \eqref{Comming back in Malliavin} and substituting these relations with the statement of Remark \ref{strssfull remark}-(b), we derive
		\begin{equation*}
			\begin{split}
				\mathbb{E}_{\mathbb{Q}}\left[F\int_{0}^{t} Z_s dB_s  \right]
				&=\int_{0}^{t}\mathbb{E}\left[\left(D_s F\right) \rho_{s} \mathbb{E}\left[ Z_s \mid \mathcal{F}_s \right] \right]  ds
				=\mathbb{E}\left[\int_{0}^{t} \left(D_s F\right) \rho_{s} \mathbb{E}\left[ Z_s  \mid \mathcal{F}_s \right] ds\right]\\
				&=\mathbb{E}\left[F\int_{0}^{t}  \rho_{s} \mathbb{E}\left[ Z_s  \mid \mathcal{F}_s \right] dB_s\right]
				=\mathbb{E}_{\mathbb{Q}}\left[F\int_{0}^{t}  \mathbb{E}\left[ Z_s  \mid \mathcal{F}_s \right] dB_s\right]\\
				&=\mathbb{E}_{\mathbb{Q}}\left[F\int_{0}^{t}  \mathbb{E}_{\mathbb{Q}}\left[ Z_s  \mid \mathcal{F}_s \right] dB_s\right].
			\end{split}
		\end{equation*} 
		In the last equality we have used Proposition \ref{equiality of measures}, and the $(\mathbb{F},\mathbb{P})$-martingale property of $(\rho_t)_{t \leq T}$ which yields to the fact that $\mathbb{E}_{\mathbb{Q}}\left[ Z_s  \mid \mathcal{F}_s \right] =\mathbb{E}\left[ \rho_{s} Z_s  \mid \mathcal{F}_s \right]\frac{1}{\rho_s}=\mathbb{E}\left[ Z_s  \mid \mathcal{F}_s \right]$. The proof is then complete.
		
	\end{proof}


\end{document}